\documentclass[11pt]{amsart}
\usepackage[utf8]{inputenc}
\usepackage[margin=1.5in]{geometry}
\usepackage[all]{xy}
\usepackage{tikz}
\usepackage{tikz-cd}
\usepackage{enumitem}
\usepackage{amssymb}
\usepackage{url}
\usepackage{hyperref}

\newtheorem{theorem}[subsection]{Theorem}

\newtheorem{proposition}[subsection]{Proposition}
\newtheorem{corollary}[subsection]{Corollary}

\theoremstyle{definition}
\newtheorem{definition}[subsection]{Definition}

\theoremstyle{remark}
\newtheorem{example}[subsection]{Example}
\newtheorem{remark}[subsection]{Remark}

\newcommand{\Top}{\mathsf{Top}_*}

\newcommand{\Spt}{\mathsf{Spt}}
\newcommand{\TQ}{\mathsf{TQ}}
\newcommand{\Sp}{\Sigma^{\infty}}
\newcommand{\U}{\Omega^{\infty}}
\newcommand{\Tot}{\operatorname{\mathsf{Tot}}}
\newcommand{\holim}{\operatorname{holim}}

\newcommand{\Alg}[1]{ {\mathsf{Alg}_{\mathcal{#1}}} }

\newcommand{\AlgN}[1]{\mathsf{Alg}^{\omega}_{\mathsf{lev}}\left(\mathcal{#1}\right)}
\newcommand{\AlgSym}[1]{\mathsf{Alg}^{\Sigma}_{\mathsf{lev}}\left(\mathcal{#1}\right)}

\newcommand{\SymSeq}{\mathsf{SymSeq}}

\newcommand{\Oper}{\mathsf{Oper}}

\newcommand{\Map}{\operatorname{Map}}
\newcommand{\cro}{\operatorname{cr}}

\newcommand{\hofib}{\operatorname{hofib}}

\newcommand{\tohofib}{\operatorname{tohofib}}
\newcommand{\tohocofib}{\operatorname{tohocofib}}
\newcommand{\colim}{\operatorname{colim}}

\newcommand{\id}{\mathrm{id}}
\newcommand{\Id}{\mathtt{Id}}
\newcommand{\Cobar}{\operatorname{Cobar}}
\newcommand{\Barr}{\operatorname{Bar}}

\newcommand{\pt}{*}

\newcommand{\Cal}[1]{\mathcal{#1}}
\newcommand{\Ncirc}[1]{\mathbf{N}^{\hat{\circ}#1}}

\newcommand{\circhat}{\hat{\circ}}
\newcommand{\circdot}{{\odot}}
\newcommand{\tensorhat}{\hat{\otimes}}
\newcommand{\tensordot}{\dot{\otimes}}
\newcommand{\boxcirc}{\mathring{\square}}

\newcommand{\Nl}{\mathbf{N}_{\mathsf{lev}}}

\newcommand{\Isym}{\Cal{I}^{\Sigma}}

\newcommand{\der}[1]{\partial_* \Id}

\newcommand{\SD}{\Sigma{\cdot}\Delta^{\bullet}}
\newcommand{\bdel}{\mathbf{\Delta}}
\newcommand{\hocolim}{\operatorname{hocolim}}

\newcommand{\coend}[1]{\mathsf{coEnd}(\Cal{#1})}
\newcommand{\Dres}{\bdel^{\mathsf{res}}}
\newcommand{\Prim}{\operatorname{\mathsf{Prim}}}

\newcommand{\fco}{C(\Cal{O})} 

\newcommand{\ptC}{\pmb{\phi}}

\title[Derivatives of the identity]{On the Goodwillie derivatives of the identity in structured ring spectra}
\author{Duncan A. Clark}
\address{Department of Mathematics, Ohio State University, 231 W 18th Ave, Columbus, OH 43210-1174, USA}
\email{clark.1843@osu.edu}

\begin{document}
\date{\today}
\maketitle

\begin{abstract}
    The aim of this paper is three-fold: (i) we construct a natural highly homotopy coherent operad structure on the derivatives of the identity functor on structured ring spectra which can be described as algebras over an operad $\Cal{O}$ in spectra, (ii) we prove that every connected $\Cal{O}$-algebra has a naturally occurring left action of the derivatives of the identity, and (iii) we show that there is a naturally occurring weak equivalence of highly homotopy coherent operads between the derivatives of the identity on $\Cal{O}$-algebras and the operad $\Cal{O}$. 
    
    Along the way, we introduce the notion of $\mathbf{N}$-colored operads with levels which, by construction, provides a precise algebraic framework for working with and comparing highly homotopy coherent operads, operads, and their algebras. 
\end{abstract}


\section{Introduction}

A slogan of functor calculus widely expected to hold is that the symmetric sequence of Goodwillie derivatives of the identity functor on a suitable model category $\mathsf{C}$, denoted $\partial_* \Id_{\mathsf{C}}$, ought to come equipped with a natural operad structure. A result of this type was first proved by Ching in \cite{Ching-thesis} for $\mathsf{C}=\mathsf{Top}_*$ and more recently in the setting of $\infty$-categories in \cite{Ching-day-conv}. In this paper, we construct an explicit ``highly homotopy coherent'' operad structure for the derivatives of the identity functor in the category of algebras over a reduced operad $\Cal{O}$ in spectra.

The derivatives of the identity in $\Alg{O}$ have previously been studied (\cite{Per-thesis}, \cite{KuhPer}) and it is known that $\Cal{O}[n]$ is a model for $\partial_n \Id_{\Alg{O}}$ -- the $n$-th Goodwillie derivative of $\Id_{\Alg{O}}$. It is further conjectured (see, e.g.,  Arone-Ching \cite{AC-1}) that $\partial_* \Id_{\Alg{O}}$ and $\Cal{O}$ be equivalent as operads: a main difficulty of which is describing an intrinsic operad structure on the derivatives of the identity which may be compared with that of the operad $\Cal{O}$. Our main theorem addresses this conjecture. 

\begin{theorem} \label{main} Let $\Cal{O}$ be an operad in spectra such that $\Cal{O}[n]$ is $(-1)$-connected for $n\geq 1$ and $\Cal{O}[0]=\pt$. Then,
    \begin{enumerate}[label=(\alph*)]
        \item The derivatives of the identity in $\Alg{O}$ can be equipped with a natural highly homotopy coherent operad structure
        \item Moreover, with respect to this structure, $\partial_* \Id_{\Alg{O}}$ is equivalent to $\Cal{O}$ as highly homotopy coherent operads. 
    \end{enumerate}
\end{theorem}

The proofs of parts (a) and (b) to Theorem \ref{main} may be found in Sections \ref{sec:1a} and \ref{O-and-dern-equiv}, respectively. Our technique is to avoid working with the identity directly by replacing it with the Bousfield-Kan cosimplicial resolution provided by the stabilization adjunction $(Q,U)$ for $\Cal{O}$-algebras. The strong cartesianness estimates of Blomquist \cite{Blom} (see also Ching-Harper \cite{Ching-Harper-BM}) allow us to then express $\partial_* \Id_{\Alg{O}}$ as the homotopy limit of the cosimplicial diagram (showing only coface maps) \begin{equation} \label{eq:der-model-stab} \partial_*(QU)^{\bullet+1}=  \left(\xymatrix{
    \partial_*(UQ) \ar@<-.5ex>[r] \ar@<.5ex>[r] &
    \partial_*(UQ)^2 \ar@<-1ex>[r] \ar[r] \ar@<1ex>[r] &
    \partial_*(UQ)^3  \cdots 
}\right)\end{equation}
whose terms $\partial_*(QU)^{k+1}$ may be readily computed by an $\Cal{O}$-algebra analogue of the Snaith splitting. We thus obtain a natural cosimplicial resolution $C(\Cal{O})$ of the derivatives of the identity such that $\partial_* \Id_{\Alg{O}}\simeq \holim_{\Delta} C(\Cal{O})$ which furthermore may be identified as the $\TQ$ resolution of $\Cal{O}$ as a left $\Cal{O}$-module. Our approach is influenced by the work of Arone-Kankaanrinta \cite{AK} wherein they use the cosimplicial resolution offered by the stabilization adjunction between spaces and spectra to analyze the derivatives of the identity in spaces via the classic Snaith splitting. 

We induce a highly homotopy coherent operad structure (i.e., $A_{\infty}$-operad) on $\partial_* \Id_{\Alg{O}}$ by constructing a pairing of the resolution $C(\Cal{O})$ with respect to the box product $\square$ for cosimplicial objects (see Batanin \cite{Bat-box}). Thus, we extend to the monoidal category of symmetric sequences a technique utilized in McClure-Smith \cite{MS-box}: specifically, that if $X$ is a $\square$-monoid in cosimplicial spaces or spectra then $\Tot(X)$ is an $A_{\infty}$-monoid (with respect to the closed, symmetric monoidal product for spaces or spectra).

There are some subtleties that arise in that (i) the box product is not as well-behaved when working with the composition product $\circ$ of symmetric sequences, and (ii) the extra structure encoded by $\circ$ leads us to work with $\mathbf{N}$-colored operads to express $A_{\infty}$-monoids with respect to composition product.  As such, one of the main developments of this paper is that of $\mathbf{N}$-colored operads \textit{with levels} (i.e., $\Nl$-operads) as useful bookkeeping tools designed to algebraically encode operads (i.e., strict composition product monoids) and ``fattened-up'' operads as their algebras. Within this framework of $\Nl$-operads we can also describe algebras over an $A_{\infty}$-operad. 

\subsection{Remark on Theorem \ref{main}} \label{sec:remark}
In the statement of Theorem \ref{main} the phrase ``naturally occuring'' means that we refrain from endowing $\partial_* \Id_{\Alg{O}}$ with the operad structure from $\Cal{O}$ directly. Rather, we produce a method for intrinsically describing operadic structure possessed by the derivatives of the identity that should carry over to other model categories suitable for functor calculus. In particular, the constructions of such an operad structure on the derivatives of the identity should: \begin{enumerate}[label=(\roman*)] 
\item Recover the ($A_{\infty}$-) operad structure endowed on $\partial_* \Id_{\mathsf{Top}_*}$ described by Ching in \cite{Ching-thesis} 

\item Endow the derivatives of an arbitrary homotopy functor $F\colon \Alg{O}\to \Alg{O'}$ with a natural bimodule structure over $(\partial_* \Id_{\Alg{O'}}, \partial_* \Id_{\Alg{O}})$ suitable for describing a chain rule (as in Arone-Ching \cite{AC-1})

\item Be fundamental enough to describe an operad structure on $\partial_* \Id_{\mathsf{C}}$ and chain rule for a suitable model category $\mathsf{C}$ (e.g.,  one in which one can do functor calculus).
\end{enumerate}

\subsection{Future applications} The three facets outlined above are all matters of ongoing work and will not be pursued in this document. We note however that our constructions are anticipated to underlie a ``highly homotopy coherent chain rule'' for composable functors $F,G$ of structured ring spectra. That is, a comparison map $\partial_*F\circ \partial_*G\to \partial_*(FG)$ which, under the identification of $\partial_* \Id_{\Alg{O}}\simeq \Cal{O}$, prescribes a suitably coherent $(\Cal{O}',\Cal{O})$-bimodule structure on the derivatives of an arbitrary functor $F\colon \Alg{O}\to \Alg{O'}.$ Such a result would extend work of Arone-Ching \cite{AC-1} (see also Klein-Rognes \cite{Klein-Rog}, Ching \cite{Ching-chain-rule}, and Yeakel \cite{Yeak}) to categories of structured ring spectra and lend to a more robust analysis of functors thereof.

Item (iii) above is perhaps the most lofty and also the most tempting. We are interested in utilizing our techniques to endow $\partial_* \Id_{\mathsf{C}}$ with a naturally occurring operad structure for a suitable model category $\mathsf{C}$. One application of such a result would be in providing homotopy descent data in the form of an equivalence of categories between  (a suitable subcategory of) $\mathsf{C}$ and algebras over the operad $\partial_* \Id_{\mathsf{C}}$ (see Hess \cite{Hess}, Behrens-Rezk \cite{BR}, and Francis-Gaitsgory \cite{FG}). Such an extension of our work seems to rely crucially on the existence of Snaith splittings associated to the stabalization adjunction $(\Sp_{\mathsf{C}}, \U_{\mathsf{C}})$ between $\mathsf{C}$ and $\mathsf{Spt}(\mathsf{C})$ in order to provide a cosimplicial model $\partial_*\Id_{\mathsf{C}}$. Such a splitting is necessarily a statement about the Taylor tower of the associated comonad $\mathsf{K}_{\mathsf{C}}=\Sp_{\mathsf{C}}\U_{\mathsf{C}}$ and the properties of its derivatives. If $\Spt(\mathsf{C})\simeq \Spt$ then Arone-Ching provide a model for the derivatives of $\mathsf{K}_{\mathsf{C}}$ in \cite{AC-1} and Lurie outlines a model for $\partial_*\mathsf{K}_{\mathsf{C}}$ as an $\infty$-cooperad in \cite[\textsection 5.2 and \textsection 6]{Lur}. A more rigid description for the cooperad structure in general is the subject of ongoing work and will not be further pursued in this paper.

\subsection{Outline of the argument}
Our main tool is to utilize the Bousfield-Kan cosimplicial resolution of an $\Cal{O}$-algebra $X$ with respect to the $\TQ$-homology adjunction \[ \xymatrix{
    \Alg{O} \ar@<.75ex>[r]^-{Q} & \textsf{Alg}_{J}\simeq \mathsf{Mod}_{\Cal{O}[1]}. \ar@<.5ex>[l]^-{U}
    }\] 
Here, $J$ denotes a suitable replacement of $\tau_1\Cal{O}$, the truncation of $\Cal{O}$ above level 1 (see Section \ref{sec:stab}). Of important note is that the pair $(Q,U)$ is equivalent to the stabilization adjunction for $\Cal{O}$-algebras (see Section \ref{sec:stab}) and that $\mathsf{Alg}_J$ and $\mathsf{Mod}_{\Cal{O}[1]}$ are Quillen equivalent. 

Using the strong connectivity estimates offered by Blomquist's higher stabilization theorems \cite[\textsection 7]{Blom}, we first show that $\partial_* \Id_{\Alg{O}}$ is equivalent to $\holim_{\Delta} \partial_*(UQ)^{\bullet+1}$ (see (\ref{eq:der-model-stab})). Similarly to Arone-Kankaanrinta \cite{AK}, in which they compute the $n$-excisive approximations (resp. $n$-th derivatives) of the identity functor on $\Top$ in terms of the $n$-excisive approximations (resp. $n$-th derivatives) of iterates of stabilization $\U \Sp$ by means of the Snaith splitting, we then analyze the terms $\partial_*(UQ)^{k+1}$ via an analog of the Snaith splitting in $\Alg{O}$.

 Essentially a statement about the Taylor tower of the associated comonad $QU$, the Snaith splitting in $\Alg{O}$ permits equivalences of symmetric sequences\[\partial_*(QU)\simeq  |\Barr(J, \Cal{O}, J)|\simeq J\circ^{\mathsf{h}}_{\Cal{O}} J =:B(\Cal{O})\] as $(J,J)$-bimodules (here, $\circ^{\mathsf{h}}$ denotes the derived composition product). By iterated applications of the splitting, we may compute \[\partial_*(UQ)^{k+1}\simeq \underbrace{B(\Cal{O}) \circ_J \cdots \circ_J B(\Cal{O})}_{k}\simeq \underbrace{J\circ_{\Cal{O}}\cdots  \circ_{\Cal{O}} J}_{k+1}= C(\Cal{O})^k\] and moreover that $\partial_*(UQ)^{\bullet+1}\simeq C(\Cal{O})$ as cosimplicial symmetric sequences. Here, $C(\Cal{O})$ is given by \[\xymatrix{
    J \ar@<.5ex>[r] \ar@<-.5ex>[r] &
    J\circ_{\Cal{O}}J \ar@<-1ex>[r] \ar[r] \ar@<1ex>[r] &
    J\circ_{\Cal{O}}J \circ_{\Cal{O}}J \ar@<-1.5ex>[r] \ar@<-.5ex>[r] \ar@<.5ex>[r] \ar@<1.5 ex>[r] &
    J\circ_{\Cal{O}}J \circ_{\Cal{O}}J\circ_{\Cal{O}} J \cdots
}\]
with coface map $d^i$ induced by inserting $\Cal{O}\to J$ at the $i$-th position (see Remark \ref{prop:der_n-model-equivs} along with (\ref{TQ-res-1})). 

Note, $B(\Cal{O})$ is (at least up to homotopy) a cooperad with a coaugmentation map $J\to B(\Cal{O})$, and our $C(\Cal{O})$ is essentially a rigid cosimplicial model for the cobar construction on $B(\Cal{O})$. In particular, this allows us to bypass referencing any particular model for the comultiplication on $B(\Cal{O})$ (e.g.,  that of Ching \cite{Ching-thesis}, see also Section \ref{B(O)}).  

We construct a pairing $m\colon C(\Cal{O})\square C(\Cal{O})\to C(\Cal{O})$ with respect to the box product (Definition \ref{def:box-product}) of cosimplicial symmetric sequences via compatible maps of the form (induced by the operad structure maps $J\circ J\to J$) \[ m_{p,q}\colon \underbrace{J\circ_\Cal{O} \cdots \circ_{\Cal{O}}(J}_{p+1}\circ \underbrace{J)\circ_\Cal{O} \cdots \circ_{\Cal{O}}J}_{q+1} \to \underbrace{J\circ_\Cal{O}\cdots \circ_{\Cal{O}} J\circ_{\Cal{O}} \cdots \circ_{\Cal{O}}J}_{p+q+1}\] along with a unit map $u\colon \underline{I}\to C(\Cal{O})$, where $\underline{I}$ denotes the constant cosimplicial symmetric sequence on $I$. Our argument is then to induce an $A_{\infty}$-monoidal pairing on $\partial_* \Id_{\Alg{O}}$ --- modeled as $\Tot C(\Cal{O})$ --- via $m$ and $u$ (compare with McClure-Smith \cite{MS-box}).

One difficulty which arises is that the composition product of symmetric sequences is not as well-behaved of a product as, say, cartesian product of spaces or smash product of spectra. Thus, we \textit{do not} obtain $m$ as a strictly monoidal pairing on the level of cosimplicial diagrams. In resolving this issue we introduce a specialized category of $\mathbf{N}$-colored operads \textit{with levels} (i.e., $\Nl$-operads) designed specifically to overcome these technical subtleties of the composition product. As a result, a large portion of this document is dedicated to carefully developing the framework of $\Nl$-operads and their algebras. 

With these details in tow it is then possible to produce an $A_{\infty}$-operad structure on $\partial_* \Id_{\Alg{O}}$. Let $\Tot$ denote restricted totalization $\mathrm{Tot}^{\mathsf{res}}$ (see Section \ref{sec:Use-of-tot}), we then obtain an $A_{\infty}$-monoidal pairing \[\Tot C(\Cal{O})\circ \Tot C(\Cal{O})\to \Tot C(\Cal{O})\] described as an algebra over a certain $\Nl$-operad which is a naturally ``fattened-up'' replacement of the $\Nl$-operad whose algebras are strict operads (see Definition \ref{def-of-Oper} along with Propositions \ref{prop:Oper-is-Nl} and \ref{oper-oper}). Moreover, the coaugmentation $\Cal{O}\to C(\Cal{O})$ provides a comparison between $\Cal{O}$ and $\partial_* \Id_{\Alg{O}}$ which we show yields an equivalence of $A_{\infty}$-operads, thus resolving the aforementioned conjecture. 

\begin{remark} It is worth noting that Ching \cite{Ching-day-conv} has recently proved a similar result in the context of $\infty$-categories using the Day convolution. The author expects a comparison should be possible between the arguments presented in this document and those due to Ching, but is not aware of any explicit description at present. 
\end{remark}

\subsection{Organization of paper} 
Section 2 provides an overview of the relevant details of working with $\Cal{O}$-algebras and their $\TQ$-completions. In sections 3 and 4 we provide an overview of the calculus of homotopy functors between categories of operadic algebras and describe the particular model for the derivatives of the identity that we employ. Section 5 is devoted to the box-product of cosimplicial objects. Much of the technical bulk of our paper occurs in the last three sections: Section 6 provides the framework for describing our notion of (symmetric) $\mathbf{N}$-colored operads \textit{with levels}. Section 7 provides proofs regarding of $\Nl$-operads of interest and Section 8 contains the proofs of our main theorems on the derivatives of the identity.

\subsection{Acknowledgements} The author is particularly grateful for the continued advising and feedback of John Harper. The author also wishes to extend thanks to Michael Ching for a stimulating visit to Amherst College and helpful remarks on an early draft of this paper, and to the anonymous referee whose comments greatly improved the readability of this document. The author also benefited from discussions with Nikolas Schonsheck, Matt Carr and Jens Jakob Kjaer, and was partially supported by National Science Foundation grant DMS-1547357 and Simons Foundation: Collaboration Grants for Mathematicians \#638247.

\section{Operads of spectra and their algebras}
We work in the category algebras over a reduced operad in a closed, symmetric monoidal category of spectra $(\Spt, \wedge, S)$. For convenience we will use the category of $S$-modules as in Elmendorf-Kriz-Mandell-May \cite{EKMM} and refer to such objects as \textit{spectra}. The main technical benefit of working with $S$-modules is that all spectra will be fibrant (and thus $\Tot \Cal{X}$ of a levelwise fibrant diagram will already correctly model $\holim_{\Delta} \Cal{X}$), though we note that similar results should hold in the category of symmetric spectra by utilizing suitable fibrant replacement monads. 

We observe that $\Spt$ is a cofibrantly generated, closed symmetric monoidal model category (see, e.g.,  \cite[Definition 1.12]{AC-1}) and write $\Map^{\Spt}(X,Y)$ for the internal mapping object of $\Spt$. When it is clear from context we write $\Map$ for $\Map^{\Spt}$. We let $\mathsf{Top}$ denote the category of compactly generated Hausdorff spaces. In \cite{EKMM}, it is shown that $\Spt$ admits a tensoring of $\Top$ which may be extended to $\mathsf{Top}$ by first adding a disjoint basepoint. In particular for $K\in \mathsf{Top}$, $X,Y\in \Spt$ there are natural isomorphisms \[\hom(K_+ \wedge X, Y)\cong \hom(X, \Map^{\Spt} (K_+, Y)).\]  Though we will not make explicit use of it, we define a simplicial tensoring of $\Spt$ via $K\wedge X:=|K|\wedge X$ for $K\in \mathsf{sSet}_*$ and $X\in \Spt$. 


\subsection{Symmetric sequences}
\label{sec:comp-prod-rmks} Let $(\mathsf{C}, \otimes, \mathbf{1})$ be a closed symmetric monoidal category and write $\Map^{\mathsf{C}}$ for the mapping object in $\mathsf{C}$. When $\mathsf{C}$ is clear from context we write $\Map$ for $\Map^{\mathsf{C}}$. We will require that $\mathsf{C}$ be cocomplete, and write $\ptC$ for the initial object of $\mathsf{C}$; particular categories of interest are $\Spt$ and $\Top$.

Recall that a \textit{symmetric sequence} in $\mathsf{C}$ is a collection $X[n]\in\mathsf{C}$ for $n\geq 0$ such that $X[n]$ admits a (right) action by $\Sigma_n$. We let $\SymSeq_{\mathsf{C}}$ denote the category of symmetric sequences in $\mathsf{C}$ and action preserving morphisms. A symmetric sequence $X$ is \textit{reduced} if $X[0]=\ptC$ (some authors require in addition that $X[1]\cong \mathbf{1}$, however we omit this condition). When $\mathsf{C}$ is clear from context we will simply write $\SymSeq$.
Note that $\SymSeq$ comes equipped with a monoidal product $\circ$, the \textit{composition product} (also called \textit{circle product}) defined as follows (see also \cite{Rezk-thesis} or \cite{Har-1}). 

\subsection{The composition product of symmetric sequences} For $X,Y\in \SymSeq$ we define $X\circ Y$ at level $k$ by \begin{equation}
    (X {\circ} Y)[k] = \coprod_{n\geq 0} X[n] \otimes_{\Sigma_n} Y^{\check{\otimes} n} [k].
\end{equation}
Here, $\check{\otimes}$ denotes the \textit{tensor} of the symmetric sequences (e.g., as in \cite{Har-1}). For $n,k\geq 0$, $Y^{\check{\otimes} n}[k]$ is computed as \[ \coprod_{\mathbf{k}\xrightarrow{\pi} \mathbf{n}} Y[\pi_1] {\otimes} \cdots {\otimes} Y[\pi_n] \cong \coprod_{k_1+\cdots+k_n=k} \Sigma_k \cdot_{\Sigma_{k_1} \times\cdots\times \Sigma_{k_n}} Y[k_1] {\otimes} \cdots {\otimes} Y[k_n] \]
where $\pi$ runs over all surjections $\mathbf{k}=\{1,\dots,n\}\to \{1,\dots,n\}=\mathbf{n}$ and we set $\pi_i:=|\pi^{-1}(i)|$ for $i\in \mathbf{n}$. The composition product admits a unit $I$ given by $I[1]=\mathbf{1}$ and $I[k]=\ptC$ otherwise.

For our purposes, we find it convenient to work with with a slightly modified version of the composition product for reduced symmetric sequences. Let $X,Y\in \SymSeq$ be reduced. Let $(k_1,\dots,k_n)$ denote a sequence of integers $k_1,\dots,k_n\geq 1$ (allowing for repetition of entries) and set $\mathsf{Sum}_n^k$ to be the collection of orbits $(k_1,\dots,k_n)_{\Sigma_n}$ such that $\sum_{i=1}^n k_i=k$. 

\begin{definition}\label{def:H} Given $k_1,\dots,k_n\geq 1$ we define $H(k_1,\dots,k_n)$ as the collection of block permutation matrices $\Sigma_{k_1}\times \cdots \times \Sigma_{k_n}\leq \Sigma_k$, along with the $\Sigma_{p_i}$ permutations of those blocks such that $k_j=d_i$. 
\end{definition}

\begin{remark} 
    We observe that orbits $(k_1,\dots,k_n)_{\Sigma_n}$ are in bijective correspondence to partitions $k=d_1p_1+\cdots+d_mp_m$ where $1\leq d_1<\cdots<d_m$ and $p_i\geq 1$. Given an orbit $(k_1,\dots,k_n)_{\Sigma_n}$ let $1\leq d_1\leq \cdots d_m$ be the distinct entries of multiplicity $p_i$. We note that there is an isomorphism (here, $\Sigma_m^{\wr n}$ denotes the \textit{wreath product} $\Sigma_m\wr \Sigma_n :=\Sigma_m^{\times n}\rtimes \Sigma_n$)
    \[H(k_1,\dots,k_n)\cong \Sigma_{d_1}^{\wr p_1}\times\cdots \times \Sigma_{d_m}^{\wr p_m}.\] 
    
    Moreover $H(k_1,\dots,k_n)$ admits a natural $\Sigma_n$ action by permutation of elements $k_i$ and the induced map $H(k_1,\dots,k_n) \to H(k_{\sigma(1)},\dots, k_{\sigma(n)})$ is an isomorphism for all $\sigma\in \Sigma_n$.\end{remark} 
    
    Though we will not need this fact, we remark that $H(k_1,\dots,k_n)$ may be identified with the stabilizer of the $\Sigma_k$ action on partitions of $\{1,\dots,k\}$ into sets of size $k_1,\dots,k_n$ (see, e.g.,  \cite[\textsection 1.12]{Ching-chain-rule}).

For $k\geq 0$ we set $\mathbf{\Sigma}[k]:=\coprod_{\sigma\in \Sigma_k} \mathbf{1}$.

\begin{remark} The composition product $X\circ Y$ may be equivalently written as \begin{equation} \label{comp-prod-def-orb} 
    (X{\circ} Y)[k]\cong \coprod_{n\geq 0} \coprod_{(k_1,\dots,k_n)_{\Sigma_n}\in \mathsf{Sum}_n^k} \mathbf{\Sigma}[k] \otimes_{H(k_1,\dots,k_n)} X[n] \otimes Y[k_1]\otimes \cdots \otimes Y[k_n].
\end{equation}
Here, the  action of $H(k_1,\dots,k_n)$ on $\mathbf{\Sigma}[k]$ is induced by that on $\Sigma_k$ and the action on $X[n]\otimes Y[k_1] \otimes \cdots \otimes Y[k_n]$ is given as follows (see also Ching \cite[1.13]{Ching-chain-rule})
\begin{itemize}
    \item $\Sigma_{p_1}\times \cdots \times \Sigma_{p_m}\leq \Sigma_n$ acts on $X[n]$ 
    \item for $i=1,\dots,m$, $\Sigma_{d_i}^{\wr p_i}$ acts on the factors $Y[k_j]$ such that $k_j=d_i$ by
    \subitem (i) permuting the $p_i$ factors $Y[d_i]$  
    \subitem (ii) acting by corresponding $\Sigma_{d_i}$ factor on each $Y[d_i]$. 
\end{itemize}
\end{remark}

We also make the following definition for the \textit{nonsymmetric composition product} $X\hat{\circ} Y$ (note that our definition differs from \cite{Har-2}) \begin{equation}
\label{comp-prod-nonsym}
    (X\hat{\circ} Y)[k] := \coprod_{n\geq 0} \coprod_{(k_1,\dots,k_n)_{\Sigma_n}\in \mathsf{Sum}_n^k} X[n]\otimes Y[k_1]\otimes \cdots \otimes Y[k_n].
\end{equation}

Note that $\hat{\circ}$ is not associative, our primary use for $\hat{\circ}$ will be as a bookkeeping tool for indexing the factors involved in expanding iterates of $\circ$ from the left (as in Section \ref{sec:N-lev}).

\subsection{Operads as monads}
A \textit{reduced operad} in $\mathsf{C}$ is a  symmetric sequence $\Cal{O}$ which is a monoid with respect to $\circ$, i.e., there are maps $\Cal{O}\circ \Cal{O}\to \Cal{O}$ and $I\to \Cal{O}$ which satisfy additional associativity and unitality relations (see, e.g.,  Rezk \cite{Rezk-thesis}). We will only consider reduced operads, and interpret \textit{operad} to mean reduced operad.

Any symmetric sequence $M$ gives rise to a functor $M\circ (-)$ on $\mathsf{C}$ given as follows (note $X^{\otimes 0}=\mathbf{1}$) \[ X\mapsto M\circ (X) = \bigvee_{n\geq 0} M[n] \otimes_{\Sigma_n} X^{\otimes n}. \] If $\Cal{O}$ is an operad, then the associated functor $\Cal{O}\circ(-)$ is a monad on $\mathsf{C}$ which we will frequently conflate with the operad $\Cal{O}$. We let $\mathsf{Alg}_{\Cal{O}}^{\mathsf{C}}$ denote the category of algebras for the monad associated to an operad $\Cal{O}$ in $\mathsf{C}$. 

When $\mathsf{C}=\Spt$ and $\Cal{O}$ is an operad of spectra, then $\Alg{O}=\mathsf{Alg}_{\Cal{O}}^{\Spt}$ is a pointed simplicial model category (see, e.g.,  \cite[\textsection 7]{Ching-Harper-Koszul}) when endowed with projective model structure from $\Spt$. For a further overview of notation and terminology we refer the reader to \cite[\textsection 3]{Har-1} or \cite[\textsection 2]{Rezk-thesis}. 

\subsection{Assumptions on $\Cal{O}$} \label{cofibrancy} 
From now on in this document we assume that $\Cal{O}$ is a reduced operad in $\Spt$ which obeys some mild cofibrancy conditions that are satisfied if, e.g.,  $\Cal{O}$ arises via the suspension spectra of a cofibrant operad in spaces. In particular, we require that the underlying symmetric sequence of  $\Cal{O}[n]$ be $\Sigma$-cofibrant (see, e.g.,  \cite[\textsection 9]{AC-1}) and that the terms $\Cal{O}[n]$ be $(-1)$-connected for all $n \geq 1$.


\subsection{Use of restricted totalization} \label{sec:Use-of-tot}
We systematically interpret $\Tot$ of a cosimplicial diagram to mean restricted totalization (see also  \cite[\textsection 8]{Ching-Harper-Koszul}) \[\Tot:=\mathrm{Tot}^{\mathsf{res}}\cong\Map_{\Dres}^{\Spt} \left( \Delta^{\bullet}, -\right):=\left(\Map^{\Spt}(\Delta^{\bullet},-)\right)^{\Dres}.\] Here, $\bdel$ denote the usual simplicial category of finite totally-ordered sets $[n]:=\{0<1< \cdots<n\}$ and order preserving maps, $\Dres\subset \bdel$ is the subcategory obtained by omitting degeneracy maps, and $\Delta^{\bullet}$ denotes the usual cosimplicial space of topological $n$-simplices. For convenience, if  $C^{\bullet}$ is cosimplicial object, we will write $\Tot C^{\bullet}$ instead of the more technically correct $\Tot (C^{\bullet}|_{\Dres})$.

Diagrams shaped on $\Dres$ are referred to as \textit{restricted} cosimplicial diagrams. Importantly the inclusion $\Dres \to \bdel$ is homotopy left cofinal\footnote{The main property we are interested in here is that such functors induce equivalences on homotopy limits.} and so if $C^{\bullet}$ is a cosimplicial diagram in $\Alg{O}$ which is levelwise fibrant (as opposed to the stronger condition of Reedy fibrancy), there are equivalences  \[ \holim_{\bdel} C^{\bullet}\simeq \holim_{\Dres} C^{\bullet}\simeq  \Tot C^{\bullet}.\]

\subsection{Truncations of $\Cal{O}$} \label{sec:trunc}
For $n\geq 0$ we define $\tau_n\colon \SymSeq\to \SymSeq$ to be the \textit{$n$-th truncation functor} given at a symmetric sequence $M$ by \begin{equation*} 
(\tau_n M)[k] = \begin{cases} M[k] & k\leq n \\ \pt & k > n\end{cases} \end{equation*}
with natural transformations $\tau_n\to\tau_{n-1}$. We let $i_n$ be the fiber of $\tau_n\to \tau_{n-1}$, i.e., $i_n M [k] = \pt$ for $k\neq n$ and $i_n M[n]=M[n]$ in which case we say $i_nM$ is \textit{concentrated at level $n$}.

For $M=\Cal{O}$ the truncations $\tau_n \Cal{O}$ assemble into a tower of $(\Cal{O},\Cal{O}$)-bimodules which receives a map from $\Cal{O}$ of the form \begin{equation} \label{eq:truncation-towerc} 
\mathcal{O}\to \cdots \to \tau_3 \Cal{O} \to \tau_2 \Cal{O} \to \tau_1 \Cal{O}.
\end{equation}
The tower (\ref{eq:truncation-towerc}) is well studied and plays a central role in examining the homotopy completion of a structured ring spectrum as in \cite{Har-Hess}. Note as well that $\Cal{O}\to \tau_1\Cal{O}$ is a map of operads and there is a well-defined composite $\tau_1\Cal{O}\to\Cal{O}\to \tau_1\Cal{O}$ which factors the identity on $\tau_1\Cal{O}$.

\subsection{Change of operad adjunction} \label{change-of-Op} Associated to a map $f\colon \Cal{O}\to\Cal{O}'$ of operads there is a Quillen adjunction of the form (see, e.g.,  \cite{Rezk-thesis}) \begin{equation*} 
    \xymatrix{
    \Alg{O} \ar@<.75 ex>[r]^{f_*} &
    \Alg{O'} \ar@<.75 ex>[l]^{f^*}
    }
\end{equation*}
in which the left adjoint $f_*$ is given by the (reflective) coequalizer \[f_*(X):=\Cal{O'} \circ_{\Cal{O}} (X)=\colim \left(\xymatrix{ \Cal{O}'\circ\Cal{O}\circ(X)\ar@<.75ex>[r] \ar@<-.5ex>[r] & \Cal{O}'\circ (X) } \right)\] and the right adjoint $f^*$ is the forgetful functor along $f$. If $f$ is a levelwise equivalence then the above adjunction is a Quillen equivalence and furthermore the left derived functor $\mathsf{L}f_*$ may be calculated via a simplicial bar construction as follows (see, e.g.,  \cite{Har-1})  \begin{equation*}
    \mathsf{L}f_*(X):= \Cal{O}' \circ^{\mathsf{h}}_{\Cal{O}}(X)\simeq |\Barr(\Cal{O}', \Cal{O}, X^c)|.
\end{equation*}

\subsection{Stabilization of $\Cal{O}$-algebras}
\label{sec:stab}In order to have a well-defined calculus of functors on $\Alg{O}$ it is necessary to understand the stabilization of the category of such algebras. Note that $\Alg{\Cal{O}}$ is tensored over simplicial sets (see, e.g.,  \cite[\textsection 7]{Ching-Harper-Koszul}) and thus one can define $\mathsf{Sp}(\Alg{O})$, the category of Bousfield-Friendlander spectra of $\Cal{O}$-algebras, which is Quillen equivalent to the category of left $\Cal{O}[1]$-modules, $\mathsf{Mod}_{\Cal{O}[1]}$ (see, e.g.,  \cite{Bast-Man} or \cite[\textsection 2]{Per-thesis}). 

The stabilization map for $\Cal{O}$-algebras is thus equivalent to the left adjoint of (\ref{change-of-Op}) with respect to the map of operads $\Cal{O}\to \tau_1\Cal{O}$, i.e., \[\Sp_{\Alg{O}} X\simeq \tau_1\Cal{O} \circ_{\Cal{O}} (X)\] for $\Cal{O}$-algebras $X$. By analogy, $\U_{\Alg{O}}$ gives an $\Cal{O}[1]$-module trivial $\Cal{O}$-algebra structure above level $2$. Moreover, if $\Cal{O}[1]\cong S$, then the stabilization of $\Alg{O}$ is equivalent to the underlying category $\Spt$. 

As in \cite{Ching-Harper-Koszul}, we replace $\tau_1\Cal{O}$ by a ``fattened-up'' operad $J$ to produce an iterable model for $\TQ$-homology with the right homotopy type. That is, let $J$ be any factorization \[\Cal{O}\xrightarrow{\; h\;} J \xrightarrow{\; g \;} \tau_1 \Cal{O}\] in the category of operads, where $h$ is a cofibration and $g$ a weak equivalence. There are then change of operads adjunctions \begin{equation} \label{stab-of-O-alg}
    \xymatrix{
    \Alg{O} \ar@<.75 ex>[r]^{Q} &
    \mathsf{Alg}_{J} \ar@<.5 ex>[l]^{U} \ar@<.75ex>[r]^{g_*} &
    \mathsf{Alg}_{\tau_1\Cal{O}} \ar@<.5ex>[l]^{g^*}
    }  \cong \mathsf{Mod}_{\Cal{O}[1]}
\end{equation}
such that $(g_*,g^*)$ is a Quillen equivalence and, notably, $U$ preserves cofibrant objects (see \cite[5.49]{Har-Hess}). We refer to the pair $(Q,U)$ as the stabilization adjunction for $\Cal{O}$-algebras and use $\mathsf{Alg}_J$ as our model for the stabilization of $\Alg{O}$.
\subsection{$\TQ$-homology}
The total left derived functor $\mathsf{L}Q(X)=:\TQ(X)$ is called the \textit{$\TQ$-homology spectrum} of $X$ and the composite $\mathsf{R}U(\mathsf{L}Q(X))$ is the \textit{$\TQ$-homology $\Cal{O}$-algebra} of $X$. We note that the $\TQ$-homology spectrum of $X$ may be calculated in the underlying category $\Spt$ as \begin{equation*}
    \mathsf{L}Q(X)\simeq |\Barr(J,\Cal{O}, X^c)|\simeq |\Barr(\tau_1\Cal{O}, \Cal{O}, X^c)|.
\end{equation*}
For simplicity, we will assume the $\Cal{O}$-algebras we work with are cofibrant by first replacing $X$ by $X^c$, where $(-)^c$ denotes a functorial cofibrant replacement in $\Alg{O}$.

\subsection{The Bousfield-Kan resolution with respect to $\TQ$} \label{sec:TQ-completion}
Associated to the stabilization adjunction for $\Cal{O}$-algebras $(Q,U)$ there is a comonad $\mathsf{K}:=QU$ on $\mathsf{Alg}_J$. Given $Y$ a $\mathsf{K}$-coalgebra, we let $C(Y)$ denote the cosimplicial object $\Cobar(U, \mathsf{K}, Y)$.  

For $X\in \Alg{O}$, let $X\to C(X):= C(QX)$ be the coaugmented cosimplicial object given below \begin{align}
\label{TQ-res-1}
    X \to &\left( \xymatrix{
    UQ(X) \ar@<.5 ex>[r] \ar@<-.5ex>[r] &
    (UQ)^2 (X) \ar@<1 ex>[r] \ar[r] \ar@<-1 ex>[r] &
    (UQ)^3(X) \cdots
    }\right)\\
    \cong &\left(\xymatrix{ J\circ_{\Cal{O}}(X) \ar@<.5 ex>[r] \ar@<-.5ex>[r] &
    J\circ_{\Cal{O}}J\circ_{\Cal{O}}(X) \ar@<1 ex>[r] \ar[r] \ar@<-1 ex>[r] &
    J\circ_{\Cal{O}}J\circ_{\Cal{O}}J\circ_{\Cal{O}}(X) \cdots }\right) \nonumber
\end{align} 
Coface maps $d^i$ in (\ref{TQ-res-1}) are induced by inserting $\Cal{O}\to J$ at the $i$-th position, i.e., \[  J \circ_{\Cal{O}} \cdots \circ_{\Cal{O}} J \cong J \circ_{\Cal{O}} \cdots \circ_{\Cal{O}} \Cal{O} \circ_{\Cal{O}} \cdots \circ_{\Cal{O}} J \to J \circ_{\Cal{O}} \cdots \circ_{\Cal{O}} J \circ_{\Cal{O}} \cdots \circ_{\Cal{O}} J\] and codegeneracy maps $s^j$ are induced by $J\circ_{\Cal{O}} J\to J\circ_J J\cong J$ at the $j$-th position. 

\begin{remark} \label{TQ-completion} The totalization of the diagram (\ref{TQ-res-1}) above is called the $\TQ$-completion of an $\Cal{O}$-algebra $X$, defined by \[X_{\TQ}^{\wedge}:=\Tot C(X) \simeq \holim_{\Delta} C(X).\]
It is known that $X\simeq X_{\TQ}^{\wedge}$ for any $0$-connected $\Cal{O}$-algebra $X$ (see, e.g., \cite{Ching-Harper-BM}).
\end{remark}

\subsection{Cubical diagrams} \label{sec:cubical-diags}
Let $\Cal{P}(n)$ denote the poset of subsets of the set $\{1,\dots,n\}$. A functor $\Cal{Z}\colon \Cal{P}(n)\to \mathsf{C}$ is called an \textit{$n$-cube in $\mathsf{C}$} or also an \textit{$n$-cubical diagram}. We use the following notation $\Cal{P}_0(n):=\Cal{P}(n)\setminus \{\emptyset\}$ and $ \Cal{P}_1(n):=\Cal{P}(n)\setminus \{\{1,\dots,n\}\}$ and refer to diagrams shaped on either $\Cal{P}_0(n)$ or $\Cal{P}_1(n)$ as \textit{punctured $n$-cubes}. The \textit{total homotopy fiber} of an $n$-cube $\Cal{Z}$, denoted $\tohofib \Cal{Z}$, is defined to be the homotopy fiber of the natural comparison map $\chi_0\colon \Cal{Z}(\emptyset)\to \holim_{\Cal{P}_0(n)} \Cal{Z}$. If the comparison $\chi_0$ is a weak equivalence (resp. $k$-connected) we say that $\Cal{Z}$ is \textit{homotopy cartesian} (resp. $k$-cartesian). 

Dually, the \textit{total homotopy cofiber} of $\Cal{Z}$ is the homotopy cofiber of the comparison map $\chi_1\colon\hocolim_{\Cal{P}_1(n)} \Cal{Z} \to \Cal{Z}(1,\dots,n)$ which we denote by $\tohocofib \Cal{Z}$. If $\chi_1$ is a weak equivalence (resp. $k$-connected) we say that $\Cal{Z}$ is \textit{homotopy cocartesian} (resp. $k$-cocartesian). We note that the total homotopy fiber (resp. cofiber) of a cube may be calculated by iterated homotopy fibers (resp. cofibers), see e.g.,  \cite[3.2]{BJM-cross}.

\begin{example}[Coface $n$-cube]
    Let $\Cal{Z}^{-1}\xrightarrow{d^0} \Cal{Z}^{\bullet}$ be a coaugmented cosimplicial object. There are associated coface $n$-cubes $\Cal{Z}_n$ whose subfaces encode the relation on coface maps (see, e.g.,  Ching-Harper \cite[\textsection 2.3]{Ching-Harper-Koszul}). We demonstrate $\Cal{Z}_2$ and $\Cal{Z}_3$ below \[ \xymatrix{ \Cal{Z}^{-1} \ar@{.>}[rr]^-{d^0} \ar@{.>}[dd]^-{d^0} && \Cal{Z}^0 \ar[dd]^-{d^0} \\ && \\ \Cal{Z}^0 \ar[rr]^-{d^1} && \Cal{Z}^1 }\hspace{1.5cm}
     \xymatrix{ 
     \Cal{Z}^{-1} \ar@{.>}[rr]^-{d^0} \ar@{.>}[dd]^-{d^0} \ar@{.>}[dr]^-{d^0} & & \Cal{Z}^0\ar[dr]^-{d^0} \ar[dd]|\hole^<(.7){d^0} && \\
     & \Cal{Z}^0\ar[rr]^<(.2){d^1} \ar[dd]_<(.2){d^1} && \Cal{Z}^1 \ar[dd]^-{d^1} \\
     \Cal{Z}^0 \ar[rr]|\hole^<(.7){d^1} \ar[dr]^-{d^0} && \Cal{Z}^1 \ar[dr]^-{d^0} & \\
     & \Cal{Z}^1 \ar[rr]^-{d^2} && \Cal{Z}^2 }\]
\end{example}

\subsection{Higher stabilization for $\Cal{O}$-algebras}
\label{delta-leq-k}
For $k\geq 0$, let $\bdel^{\leq k}$ denote the full subcategory of $\bdel$ comprised of sets $[\ell]\in \bdel$ for $\ell\leq k$ (note $\bdel^{\leq -1}=\emptyset$). There are inclusions of categories \[\emptyset=\bdel^{\leq -1} \to \bdel^{\leq 0} \to \bdel^{\leq 1} \to \dots \to \bdel^{\leq k}\to \cdots \to \bdel\] and moreover $\holim_{\bdel} Y$ may be computed as limit of the tower $\{\holim_{\bdel^{\leq k}} Y\}$ (see, e.g.,  \cite[\textsection 8.11]{Ching-Harper-Koszul} for a detailed write-up). There is a natural homotopy left cofinal inclusion  $\Cal{P}_0(n)\to \bdel^{\leq n-1}$ which, in particular, allows us to model the comparison $X\to \holim_{\bdel^{\leq n-1}} C(X)$ via the map $\chi_0$ (see Section \ref{sec:cubical-diags}) for the coface $n$-cube associated to $X\to C(X)$. 

By careful examination of the connectivities of these maps, Blomquist is able to obtain the following strong convergence estimates as a corollary to \cite[7.1]{Blom} (see also Dundas \cite{Dundas-1} and Dundas-Goodwillie-McCarthy \cite{DGM}).

\begin{proposition}
\label{higher-stabilization}
    Let $\Cal{O}$ be an operad in $\Spt$ whose entries are $(-1)$-connected, $X \in \Alg{O}$ $k$-connected, and $C(X)$ as in (\ref{TQ-completion}). Then, for any $n\geq 0$ the induced map $X\to \holim_{\bdel^{\leq n-1}}C(X)$ is $(k+1)(n+1)$-connected.
\end{proposition}

These estimates show, in particular, if $X$ is $0$-connected then $X\xrightarrow{\sim} X_{\TQ}^{\wedge}$ (see also Ching-Harper \cite{Ching-Harper-BM}). 

\section{Functor calculus and Goodwillie derivatives in $\Alg{O}$}
Functor calculus was introduced by Goodwillie in a landmark series of papers \cite{Goo-Calc1, Goo-Calc2,Goo-Calc3} as a means of analyzing homotopy functors to or from $\Top$ or $\Spt$. Since, the theory been recognized as a general phenomenon which, in particular, relates a suitable model category to its stabilization. We refer the reader to \cite{Handbook} for an overview and exposition of some recent applications of the theory. 

In this document we will only consider functors of structured ring spectra described as algebras over a reduced operad $\Cal{O}$ in $\Spt$. We refer the reader to Pereira \cite{Per-thesis} for a more detail on functor calculus in categories of structured ring spectra.

\subsection{The Taylor tower}
A central construction in functor calculus is that of the \textit{Taylor tower} (sometimes referred to also as the \textit{Goodwillie tower}) of $n$-excisive approximations associated to a  functor $F\colon \Alg{O}\to\Alg{O}$ as follows
\begin{equation}\label{Taylor-tower}
    \xymatrix{
    & & D_nF\ar[d] & &&\\
    F\ar[r] & \cdots \ar[r] & P_nF \ar[r] & P_{n-1}F \ar[r] & \cdots \ar[r] & P_0F.}
\end{equation}

The functor $P_nF$ is called the \textit{$n$-th excisive approximation} to $F$ and is initial in the homotopy category of $n$-excisive functors receiving a map from $F$. In this work, all of our approximations are based at the zero object $\pt \in \Alg{O}$. The functor $D_n F$ is called the \textit{$n$-th homogeneous layer} and is defined as \[D_nF :=\hofib(P_nF\to P_{n-1}F).\]

Note that $P_0F$ is a constant functor taking value $F(\pt)$. We call $F$ \textit{reduced} if $F(\pt)\simeq \pt$ and note that for reduced functors we have $P_1F\simeq D_1F$. We refer the reader to \cite[\textsection 3]{Goo-Calc2} for the definition and overview of the theory of $n$-excisive functors; though remark that such functors share similar properties as the $n$-th Taylor polynomial associated to a function from calculus of one variable.

\subsection{Analytic functors}  If $F$ satisfies additional connectivity conditions on certain cubical diagrams (e.g.,  if $F$ is suitably stably $n$-excisive for all $n$ as in \cite[4.1]{Goo-Calc2}) we call $F$ \textit{analytic}, or more specifically \textit{$\rho$-analytic}: a key feature being that an analytic functor $F$ may be recovered as the homotopy limit of the tower (\ref{Taylor-tower}) on $\rho$-connected inputs $X$, i.e., \[F(X)\simeq \holim_{n} P_n F(X).\] 

For instance, the identity functor on $\Top$ is $1$-analytic by the higher Blakers-Massey theorems (see, e.g.,  \cite[\textsection 2]{Goo-Calc2}) and the analogous results for structured ring spectra of Ching-Harper \cite{Ching-Harper-BM} demonstrate that the identity functor on $\Alg{O}$ is $0$-analytic. 

\subsection{Cross effects and derivatives}
Let $\Cal{S}_n(X_1,\dots,X_n)$ denote the $n$-cube \[\text{$T\mapsto \bigvee_{t\notin T} X_t$, for $T\in \Cal{P}(n)$.}\] The \textit{$n$-th cross effect} of $G$ is the $n$-variable functor defined by \[\cro_n G(X_1,\dots,X_n):= \tohofib G( \Cal{S}_n(X_1,\dots, X_n)).\]

Our work concerns the derivatives of a functor $F$, which are certain spectra which classify the homogeneous layers $D_nF$ (under some mild conditions on $F$) and are computable via cross effects. We recall first that a functor $G$ is \textit{$n$-homogeneous} if $G$ is $n$-excisive and $P_k G\simeq \pt$ for $k<n$ and that $G$ is \textit{finitary} if $G$ commutes with filtered homotopy colimits. 

A major triumph of functor calculus is the classification of $n$-homogeneous functors. Proposition \ref{prop:functor-calc-recap} below is summarized from Goodwillie \cite{Goo-Calc3} (for functors of spaces) and Pereira \cite{Per-thesis} (for functors of $\Cal{O}$-algebras) and highlights the relevant properties of the homogeneous layers $D_nF$ and derivatives $\partial_n F$ associated to a functor $F$. For notational convenience we let $\widetilde{\TQ}$ denote the composite $g_* \TQ \simeq \tau_1\Cal{O} \circ^{\mathsf{h}}_{\Cal{O}}(-)$.

\begin{proposition}
\label{prop:functor-calc-recap}
Let $F\colon \Alg{O}\to \Alg{O}$ be a homotopy functor, $X\in \Alg{O}$, and $n\geq 1$. Then: \begin{enumerate}[label=(\roman*)]
    \item $D_nF$ is \textit{$n$-homogeneous}.

    \item There are $n$-homogeneous functors $\mathbb{D}_n F$ and $\widetilde{\mathbb{D}_n}F$ such that the following diagram commutes
    \begin{equation} \label{eq:fact}
    \xymatrix{
    \Alg{O} \ar[r]^{Q} \ar[d]^-{D_nF} &
    \mathsf{Alg}_{J} \ar[r]^{g_*} \ar[d]^-{\mathbb{D}_n F} &
    \mathsf{Mod}_{\Cal{O}[1]}  \ar[d]^-{\widetilde{\mathbb{D}_n}F} \\
    \Alg{O}  &
    \mathsf{Alg}_{J} \ar[l]^{U}  &
    \mathsf{Mod}_{\Cal{O}[1]} \ar[l]^{g^*}
    }  
\end{equation}

    \item There is a $(J,J)$-bimodule $\partial_*F$, whose $n$-th entry $\partial_n F$ is called the \textit{$n$-th Goodwillie derivative of $F$}, and such that there are equivalences of underlying spectra \begin{equation*}\label{homog-layer-2}
    D_n F(X) \simeq  i_n (\partial_* F) \circ^{\mathsf{h}}_J (\TQ(X)) \end{equation*} 
    
    \item $D_nF$ is characterized by an $(\Cal{O}[1],\Cal{O}[1]^{\wr n})$-bimodule\footnote{That is, a left module over $\Cal{O}[1]$ and right module over $\Cal{O}[1]^{\wr n}$
        (see Definition \ref{rem:wreath-products} for the definition of the wreath product $\Cal{O}[1]^{\wr n}$)
    } $\widetilde{\partial_n}F$ which has underlying spectrum equivalent to that of $\partial_nF$. 
    
    \item There are equivalences of underlying spectra \begin{equation}\label{homog-layer-1}
    D_n F(X) \simeq (\widetilde{\partial_n} F \wedge^{\mathsf{L}}_{\Cal{O}[1]^{\wedge n}} \widetilde{\TQ}(X)^{\wedge^{\mathsf{L}} n})_{h\Sigma_n}
    \simeq \widetilde{\partial_n} F\wedge^{\mathsf{L}}_{\Cal{O}[1]^{\wr n}} \widetilde{\TQ}(X)^{\wedge^{\mathsf{L}} n}. \end{equation}

    \item The $n$-th derivative may be calculated via $n$-th cross effects $\cro_n$ as \[\partial_n F \simeq \widetilde{\partial_n}F\simeq \cro_n \widetilde{\mathbb{D}_n} F (\Cal{O}[1],\dots, \Cal{O}[1])\] with right $\Cal{O}[1]^{\wr n}$-action granted by permuting the inputs.
\end{enumerate}
\end{proposition}

\begin{remark} 
    The above equivalence (\ref{homog-layer-1}) hold in general for finite cell $\Cal{O}$-algebras $X$ and, if $F$ further is \textit{finitary} (i.e., $F$ commutes with filtered homotopy colimits), then the equivalences may be extended to arbitrary $\Cal{O}$-algebras $X$. The notation $\wedge^{\mathsf{L}}$ and $\circ^{\mathsf{h}}$ denote the derived smash product and circle product, respectively. We will often omit the latter notation and understand our constructions to be implicitly derived.\end{remark}

    The careful reader might note that the $n$-th Goodwillie derivative of $F$ is only defined up to weak equivalence, and so the choice $\partial_n F$ vs. $\widetilde{\partial_n}F$ may seem a pedantic distinction. For our purposes, this distinction is beneficial to the readibility of several of the upcoming proofs. Further, there are equivalences \[ \mathsf{L} g_* \partial_n F \simeq \widetilde{\partial_n} F \text{ and } \partial_n F \simeq \mathsf{R} g^* \widetilde{\partial_n}F, \] and for concreteness, the model for the derivatives of the identity we employ is as a $(J,J)$-bimodule, $\Tot C(\Cal{O})$ (see (\ref{al:CO})).

    Of note is that the choice of $\mathbb{D}_nF$ (resp. $\widetilde{\mathbb{D}_n}F$) may be made functorial in $F$ by a straightforward modification of the argument presented in \cite[2.7]{AC-1}. In particular if $F$ is finitary, then for any $Y\in \mathsf{Mod}_{\Cal{O}[1]}$ we have \begin{equation} \label{eq:D_nF(Y)} \widetilde{\mathbb{D}_n} F(Y)\simeq \widetilde{\partial_n} F\wedge_{\Cal{O}[1]^{\wr n}} Y^{\wedge n}.\end{equation}

\subsection{A note on wreath products}
We use $\Cal{O}[1]^{\wr n}$ to denote the twisted group ring (i.e., \textit{wreath product}) $(\Sigma_n)_+\wedge \Cal{O}[1]^{\wedge n}$. We recall some pertinent details of wreath products of ring spectra below.

\begin{definition} \label{rem:wreath-products}
    Given a ring spectrum $\Cal{R}$ we define \[
        \Cal{R}^{\wr n}:= \Sigma_n {\cdot} \Cal{R}^{\wedge n}\cong (\Sigma_n)_+\wedge \Cal{R}^{\wedge n}\]
    with multiplication given by \[(\sigma \wedge x) \wedge (\tau \wedge y)\mapsto \sigma\tau \wedge x\sigma(y).\] 
\end{definition}

Our main use of such objects stems from the following proposition (see also \cite[Lemma 14]{Law-1}, \cite[\textsection 2]{KuhPer}). Note that a right $\Cal{R}^{\wr n}$-module is a (right) $\Sigma_n$ object via the unit map $I\to \Cal{R}$. 

\begin{proposition} \label{eq:wreath-product-coeq}
    Let $\Cal{R}$ be a ring spectrum, $X$ a left $\Cal{R}$-module and $M$ a right $\Cal{R}$-module with $n$ commuting actions of $\Cal{R}$ (i.e., right $\Cal{R}^{\wedge n}$-module). Then, there is an isomorphism \begin{equation*} 
        (M \wedge_{\Cal{R}^{\wedge n}} X^{\wedge n})_{\Sigma_n}\cong M\wedge_{\Cal{R}^{\wr n}} X^{\wedge n}.
    \end{equation*}
\end{proposition}

\begin{remark}The right-hand equivalence of (\ref{homog-layer-1}) is an instance of this equivalence. Of note is that if $X$ is a cofibrant $\Cal{O}$-algebra, then $\TQ(X)$ is cofibrant in $\mathsf{Mod}_{\Cal{O}[1]}$ and therefore Proposition \ref{eq:wreath-product-coeq} provides that $\TQ(X)^{\wedge n}$ is a cofibrant as a left $\Cal{O}[1]^{\wr n}$-module. 
\end{remark}

In addition, the $(\Cal{O}[1], \Cal{O}[1]^{\wr n})$-bimodule structure on the derivatives $\widetilde{\partial_n} F$ for all $n\geq 1$ induces $(\tau_1\Cal{O}, \tau_1\Cal{O})$-bimodule structure on the symmetric sequence $\widetilde{\partial_*} F$ which is further compatible with the $(J,J)$-bimodule structure on $\partial_*F$ via the $(g_*,g^*)$ adjunction. In the simplified case that $\Cal{O}[1]\cong S$, an $(S,S^{\wr n})$-bimodule is just a spectrum with a right action by $\Sigma_n$ and (\ref{homog-layer-1}) reduces to an equivalence of underlying spectra \[D_n F(X)\simeq \widetilde{\partial_n}F\wedge_{\Sigma_n} \widetilde{\TQ}(X)^{\wedge n}.\]

\subsection{Taylor towers of certain functors $\Alg{O}\to \Alg{O}$}
Let $M$ be a cofibrant $(\Cal{O}, \Cal{O})$-bimodule with $M[0]=\pt$ whose terms are $(-1)$-connected. We define a functor $F_M\colon \Alg{O}\to \Alg{O}$ at $X\in \Alg{O}$ by the simplicial bar construction \begin{equation} 
    F_M(X) = |\Barr(M,\Cal{O},X)| \simeq M \circ^{\mathsf{h}}_{\Cal{O}}(X).\end{equation}

Note $F_M$ is finitary and the left $\Cal{O}$ action on $M$ induces a left $\Cal{O}$ action on $F_M(X)$. The following proposition may be summarized from Harper-Hess \cite{Har-Hess} and Kuhn-Pereira \cite[\textsection 2.7]{KuhPer} and further provides a model for the Taylor tower of functors $F_M$. For completion, we sketch proofs of the relevant details. 

\begin{proposition} \label{prop:tay-tow-FM}
    Let $M$ and $F_M$ be as described above. Then there are equivalences (natural in $M$) \begin{enumerate}[label=(\roman*) ]
        \item $P_n F_M\simeq \tau_n M \circ^{\mathsf{h}}_{\Cal{O}} (-)$
        
        \item $D_n F_M\simeq i_n M\circ^{\mathsf{h}}_{\Cal{O}}(-) \simeq i_n M \circ^{\mathsf{h}}_J (\TQ(-))$ 
        
        \item $\widetilde{\mathbb{D}_n} F_M (-)\simeq M[n] \wedge^{\mathsf{L}}_{\Cal{O}[1]^{\wr n}} (-)^{\wedge n}$
        
        \item $\widetilde{\partial_n} F_M \simeq M[n]$
    \end{enumerate}
    such that the Taylor tower for $F_M$ is equivalent to \[\label{eq:taylor-tower-FM-AlgOO}
    \xymatrix{
    && i_n M \circ^{\mathsf{h}}_{\Cal{O}}(-) \ar[d] & & 
    \\
    F_M\ar[r] & 
    \cdots \ar[r] & 
    \tau_n M \circ^{\mathsf{h}}_{\Cal{O}} (-) \ar[r] & 
    \tau_{n-1} M\circ^{\mathsf{h}}_{\Cal{O}}(-) \ar[r] & \cdots \ar[r] & \tau_1M \circ^{\mathsf{h}}_{\Cal{O}}(-).
    }\]
\end{proposition}

\begin{proof}
We will write $\circ$ for $\circ^{\mathsf{h}}$ and $\wedge$ for $\wedge^{\mathsf{L}}$. The equivalence (i) holds as $\tau_n M\circ_{\Cal{O}}(-)$ is $n$-excisive (see, e.g.,  \cite[4.3]{Per-thesis}) and by a connectivity argument (see \cite[1.14]{Har-Hess}). For (ii) we note that morphisms $\tau_n M \to \tau_{n-1} M$ give rise to the comparison maps on excisive approximations $P_n F_M\xrightarrow{q_n} P_{n-1} F_M$ and moreover the fiber sequence \[ i_n M \to \tau_n M\to \tau_{n-1} M\] identifies $i_n M\circ_{\Cal{O}} (-)$ with the fiber of $q_n$. Moreover, as the right $\Cal{O}$-action on $i_n M$ factors through $\tau_1\Cal{O}$ there are then equivalences of underlying spectra
\begin{align*} 
    D_n F_M(X) &\simeq (i_n M\circ_{\tau_1 \Cal{O}} \tau_1\Cal{O}) \circ_{\Cal{O}}(X) \\
               &\simeq i_n M  \circ_{\tau_1\Cal{O}} (\tau_1 \circ_{\Cal{O}}(X)) 
               \simeq i_n M \circ_J (\TQ(X)).
    \end{align*} Note that (iii) follows from the observation that any $Y\in \mathsf{Mod}_{\Cal{O}[1]}$ \[ 
        i_n M \circ_{\tau_1{\Cal{O}}} (Y)\simeq M[n] \wedge_{\Cal{O}[1]^{\wr n}} Y^{\wedge n}. \] 
    
The proof of (iv) follows from the equivalence $\cro_n F\simeq\cro^nF$ between cross-effects and \textit{co-cross-effects} of functors landing in a stable category as in Ching \cite{Ching-chain-rule} (see also McCarthy \cite{McC-dual}), where latter is defined dually to $\cro_n$ as follows \[\cro^n G(X_1,\dots,X_n)=\tohocofib\left(\Cal{P}(n)\ni T\mapsto G\left( \prod_{t\in T} X_t\right) \right).\]

In particular, taking co-cross-effects will commute with $\wedge_{\Cal{O}[1]^{\wr n}}$ and so 
    \[ \cro_n \widetilde{\mathbb{D}_n} F_M \simeq \cro_n (M[n]\wedge_{\Cal{O}[1]^{\wr n}} (-)^{\wedge n})\simeq  M[n]  \wedge_{\Cal{O}[1]^{\wr n}} \cro_n ((-)^{\wedge n}).\] 
Via the computation $\cro_n ((-)^{\wedge n})\simeq (\Sigma_n)_+ \wedge (-)^{\wedge n}$ we then obtain \[\widetilde{\partial_n}F_M \simeq M[n]\wedge_{\Cal{O}[1]^{\wr n}} \Cal{O}[1]^{\wr n}\simeq M[n].\]
\end{proof}

\begin{definition}
    For functors of the form $F_M$ we take as our models for $P_n F_M$, $D_n F_M$ and $\widetilde{\partial_n} F_M$ those from Proposition \ref{prop:tay-tow-FM}. A map $M\to M'$ of cofibrant $(\Cal{O}, \Cal{O})$-bimodules induces natural transformations $P_n F_M \to P_n F_{M'}$ and $D_n F_M\to D_n F_{M'}$, and also that $\widetilde{\partial_n} F_M \to \widetilde{\partial_n} F_{M'}$ is equivalent to $M[n]\to M'[n]$.
\end{definition}

\subsection{The Taylor tower of the identity on $\Alg{O}$}
Note that for $M=\Cal{O}$, the functor $F_{\Cal{O}}$ is equivalent to the identity via $\Cal{O} \circ_{\Cal{O}} (-)\simeq \Id_{\Alg{O}}$. Moreover, there are natural transformations $\Id_{\Alg{O}} \to \tau_n \Cal{O}\circ_{\Cal{O}}(-)$ provided by the unit map of the change of operads adjunction (Section \ref{change-of-Op}) applied to the map of operads $\Cal{O}\to \tau_n \Cal{O}$. The Taylor tower of the identity in $\Alg{O}$ then is equivalent to 
\begin{equation}\label{eq:taylor-tower-Id-AlgOO}
    \xymatrix{
    && i_n \Cal{O} \circ_{\Cal{O}}(-) \ar[d] & & 
    \\
    \Id_{\Alg{O}}\ar[r] & 
    \cdots \ar[r] & 
    \tau_n\Cal{O}\circ_{\Cal{O}} (-) \ar[r] & 
    \tau_{n-1}\Cal{O}\circ_{\Cal{O}}(-) \ar[r] & \cdots \ar[r] & \tau_1\Cal{O}\circ_{\Cal{O}}(-)
    }
\end{equation}

This tower (\ref{eq:taylor-tower-Id-AlgOO}) has previously been studied by Harper-Hess \cite{Har-Hess} in relation to homotopy completion of $\Cal{O}$-algebras (see also Kuhn \cite{Kuh} and McCarthy-Minasian \cite{MM}). Moreover, Ching-Harper provide $\Alg{O}$ analogues of the higher Blakers-Massey theorems in \cite{Ching-Harper-BM}  which in particular show that $\Id_{\Alg{O}}$ is $0$-analytic. That is, for $0$-connected $X$ the following comparison map is an equivalence \[X \to \holim_n \tau_n \Cal{O} \circ_{\Cal{O}}(X).\] 

As a corollary to Proposition \ref{prop:tay-tow-FM}, we obtain equivalences of underlying spectra (see also \cite{Har-Hess}) \[
    D_n \Id_{\Alg{O}}(X)\simeq i_n \Cal{O} \circ_{\Cal{O}}(X)\simeq \Cal{O}[n] \wedge_{\Cal{O}[1]^{\wr n}} \widetilde{\TQ}(X)^{\wedge n}\] 
and also observe that $\widetilde{\partial_n} \Id_{\Alg{O}}\simeq \Cal{O}[n]$ as a $(\Cal{O}[1], \Cal{O}[1]^{\wr n})$-bimodule for all $n\geq 1$. Therefore, with a view toward the operad structure on $\partial_*\Id_{\mathsf{Top}_*}$ constructed by Ching in \cite{Ching-thesis} we are lead to the following question, found in Arone-Ching \cite{AC-1}.

\subsection{Main question} \label{question} Is it possible to endow $\partial_* \Id_{\Alg{O}}$ with a naturally occurring operad structure such that $\partial_* \Id_{\Alg{O}}\simeq \Cal{O}$ \textit{as operads}?\\

A key idea to our approach is taken from Arone-Kankaanrinta \cite{AK} where they show that $\partial_* \Id_{\mathsf{Top}_*}$ may be better understood by utilizing the cosimplicial resolution from the stabilization adjunction $(\Sp,\U)$ by means of the Snaith splitting. Within the realm of $0$-connected $\Cal{O}$-algebras, the $(Q,U)$ adjunction between $\Alg{O}$ and $\mathsf{Alg}_J$ (the latter, recall, is Quillen equivalent to $\mathsf{Mod}_{\Cal{O}[1]}$) is the exact analogue of stabilization. We provide an $\Alg{O}$ analogue of the Snaith splitting in Section \ref{sec:Snaith-O}.

\section{A model for derivatives of the identity in $\Alg{O}$}
The aim of this section is to describe specifically the model for the derivatives of the identity we employ, as $\Tot$ of a certain cosimplicial symmetric sequence $C(\Cal{O})$ which may be motivated as the totalization of the cosimplicial object arising from a calculation of the $n$-th derivative of $(QU)^k$ via the Snaith splitting in $\Alg{O}$. We are further motivated by work of Arone-Kankaanrinta \cite{AK} which utilizes the Snaith splitting in spaces (\ref{eq:Snaith-top}) to provide a model for the derivatives of the identity in spaces. 

\subsection{The Snaith splitting} 
We first recall the Snaith splitting in $\Top$, that is, the existence of an equivalence (see, e.g.,  Snaith \cite{Snaith} or Cohen-May-Taylor \cite{Coh-May-Tay}) \begin{equation} \label{eq:Snaith-top}
    \Sp \U \Sp (X)\simeq \bigvee_{n \geq 1} \Sp X^{\wedge n}_{\Sigma_n}\cong \bigvee_{n\geq 1} S\wedge_{\Sigma_n}\Sp X^{\wedge n}
\end{equation} 
where $\Sigma_n$ acts on $S$ trivially. We interpret the above to mean that the Taylor tower for the associated comonad to the suspension adjunction, $\Sp\U$, splits on the image of $\Sp(-)$ as the coproduct of its homogeneous layers and moreover that $\partial_n (\Sp \U)\simeq S$ with trivial $\Sigma_n$-action. Via this splitting in spaces one obtains \[\partial_n (\U\Sp)^{k+1}\simeq \underline{S}^{\circ k}[n]\] where $\underline{S}$ denotes the reduced symmetric sequence with $\underline{S}[n]=S$ with trivial $\Sigma_n$ action. Furthermore, $\underline{S}$ inherits a natural cooperad structure and $\partial_* \Id_{\mathsf{Top}_*}$ is equivalent to the cobar construction on $\underline{S}$ (see \cite{AK}, \cite{Ching-thesis}).

\begin{remark} \label{rem:C(S)} In the sequel to this work \cite{Clark-2}, we describe a $\boxcirc$-monoid (see Definition \ref{boxcirc-monoid}) $C(\underline{S})\in \SymSeq_{\Spt}^{\bdel}$  whose totalization is equivalent to the cobar construction on $\underline{S}$. This allows for a new description of an operad structure on $\partial_* \Id_{\Top}$ using the methods from this document and addresses item (i) of Section \ref{sec:remark}.
\end{remark}

\subsection{The Snaith splitting in $\Alg{O}$}
\label{sec:Snaith-O}There is an analogous result for $\Cal{O}$-algebras, wherein the adjunction $(\Sp, \U)$ is replaced by $(Q, U)$ from (\ref{stab-of-O-alg}). Let $B(\Cal{O})$ be the $(J,J)$-bimodule \[B(\Cal{O})= J\circ^{\mathsf{h}}_{\Cal{O}}J\simeq |\Barr(J,\Cal{O},J)|\] and note that given $Y\in \mathsf{Alg}_J$ cofibrant there is a zig-zag of equivalences. 
\[ QU(Y) \xleftarrow{\sim} |\Barr(J, \Cal{O}, Y)| \xrightarrow{\sim} |\Barr(J, \Cal{O}, J)| \circ_J (Y)= B(\Cal{O}) \circ_J (Y).\]

The \textit{$\Alg{O}$ Snaith splitting} is then the equivalence \begin{equation}\label{eq:TQ-Snaith-splitting}
    Q U(Y)\simeq B(\Cal{O}) \circ_J (Y).
\end{equation}

\begin{remark} 
    At first blush, (\ref{eq:TQ-Snaith-splitting}) may not seem like a proper ``splitting'' in the style of (\ref{eq:Snaith-top}). This is more an artifact of our use of $\mathsf{Alg}_J$ for the stabilization of $\Alg{O}$. Indeed, given instead $\widetilde{Y}\in \mathsf{Mod}_{\Cal{O}[1]}$, the associated comonad arising from the adjunction $(g_*Q, Ug^*)$ between $\Alg{O}$ and $\mathsf{Mod}_{\Cal{O}[1]}$ has a natural splitting \[
    g_* Q Ug^* (\widetilde{Y})\simeq \bigvee_{k\geq 1} \widetilde{B}(\Cal{O})[k] \wedge_{\Cal{O}[1]^{\wr k}} \widetilde{Y}^{\wedge k}\]
    such that $\widetilde{B}(\Cal{O})= \tau_1\Cal{O} \circ^{\mathsf{h}}_{\Cal{O}} \tau_1 \Cal{O} \simeq |\Barr(\tau_1\Cal{O}, \Cal{O}, \tau_1\Cal{O})| \simeq |\Barr(J,\Cal{O},J)|\simeq B(\Cal{O})$. 
\end{remark}
    
\subsection{Cooperad structure on $B(\Cal{O})$} \label{B(O)} It is known that $B(\Cal{O})$ (resp. $\widetilde{B}(\Cal{O})$) is a coaugmented cooperad, at least in the homotopy category of spectra (see, e.g.,  Ching \cite{Ching-thesis} for the topological case, Lurie \cite[\textsection 5]{Lur} for an $\infty$-categorical approach, or Ginzburg-Kapranov \cite{Ginz-Kap} for the chain complexes case) via the natural comultiplication  \[ J \circ^{\mathsf{h}}_{\Cal{O}} J \simeq J \circ^{\mathsf{h}}_{\Cal{O}} \Cal{O} \circ^{\mathsf{h}}_{\Cal{O}} J \to J \circ^{\mathsf{h}}_{\Cal{O}} J \circ^{\mathsf{h}}_{\Cal{O}} J \simeq (J \circ^{\mathsf{h}}_{\Cal{O}} J) \circ_J (J \circ^{\mathsf{h}}_{\Cal{O}} J) .\]
    
    We would like to say that the $\Alg{O}$ Snaith splitting allows one to immediately recognize $\partial_* \Id_{\Alg{O}}$ as the cobar construction on $B(\Cal{O})$, however the splittings provided seem to be too weak to justify this claim (a similar problem is enocuntered in Arone-Kankaanrinta \cite{AK} for the classic Snaith splitting). As such, one benefit of our work is that we do not require any more rigid cooperad structure on $B(\Cal{O})$ to produce our model for $\partial_* \Id_{\Alg{O}}$. 
    
    Also of note is that the $\Alg{O}$ Snaith splitting may be interpreted to say that any $Y\in \mathsf{Alg}_J$ (resp. $\widetilde{Y}\in \mathsf{Mod}_{\Cal{O}[1]}$) is naturally a divided power coalgebra over $B(\Cal{O})$ (resp. $\widetilde{B}(\Cal{O})$), at least in the homotopy category, and that the functor $X\mapsto \TQ(X)$ underlies the left-adjoint to the conjectured Quillen equivalence (i.e., Koszul duality equivalence) between nilpotent $\Cal{O}$-algebras and nilpotent divided power $B(\Cal{O})$-coalgebras from Francis-Gaitsgory \cite{FG} (which has since been partially resolved by Ching-Harper \cite{Ching-Harper-Koszul}). 

\subsection{Interaction of the stabilization resolution with Taylor towers} \label{sec:fibrant-model}
We now provide the explicit model we employ for $\partial_* \Id_{\Alg{O}}$. Our argument is essentially to show that one can ``move the $\partial_*$ inside the $\holim$'' on the right hand side of (\ref{TQ-res-1}) by higher stabilization and then use the $\Alg{O}$ Snaith splitting to recognize the resulting diagram. Let us write $\Id$ for $\Id_{\Alg{O}}$.

\begin{proposition} \label{eq:P_n-equiv}
Let $k\geq n\geq 1$, then $P_n \Id \xrightarrow{\sim} \holim_{\bdel^{\leq k-1}} P_n ((UQ)^{\bullet+1})$.
\end{proposition} 

\begin{proof} 
The estimates from Proposition \ref{higher-stabilization} suffice to show that  the map \[c_k\colon \Id\to \holim_{\bdel^{\leq k-1}} C(-)\] agrees to order $n$ on the subcategory of $0$-connected objects (see \cite[1.2]{Goo-Calc3}) in which case $P_n (c_k)$ is an equivalence via \cite[1.6]{Goo-Calc3}. Further, \[P_n (\holim_{\bdel^{\leq k-1}} C(-)) \simeq \holim_{\bdel^{\leq k-1}} P_n  ((UQ)^{\bullet+1})\] as $P_n(-)$ commutes with \textit{very finite}\footnote{
    Recall that a \textit{very finite} homotopy limit is one taken over a diagram whose nerve has only finitely many nondegenerate simplices, and that such homotopy limits will commute with filtered homotopy colimits. Homotopy limits over $n$-cubes and punctured $n$-cubes are very finite
} homotopy limits by construction (cf. Section \ref{delta-leq-k}).  
\end{proof} 

Since $D_n(-)$ and $\partial_n(-)$ are built from $P_n(-)$ by very finite homotopy limits, Proposition \ref{eq:P_n-equiv} extends to an equivalence on homogeneous layers and derivatives as well. Moreover, the restriction map \[\holim_{\bdel} P_n((UQ)^{\bullet+1})\to \holim_{\bdel^{\leq k-1}} P_n ((UQ)^{\bullet+1})\] 
is an equivalence for $k\geq n\geq1$ as the objects as a corollary to the higher stabilization estimates from Proposition \ref{higher-stabilization} (resp. for $D_n$ and $\partial_n$). 

Let $M$ be an $(\Cal{O}, \Cal{O})$-bimodule. For notational convenience, for $k\geq 1$, we set \begin{equation} \label{eq:Mcirck} 
    M^{(k)} = \underbrace{M\circ_{\Cal{O}} \cdots \circ_{\Cal{O}} M}_k.
    \end{equation} Note that $J^{(k)}$ is a cofibrant $(\Cal{O}, \Cal{O})$-bimodule with $(UQ)^{k+1}(X)= J^{(k+1)}\circ_{\Cal{O}} (X)$. By Proposition \ref{prop:tay-tow-FM}, there are then equivalences  \begin{align*}
    P_n \Id \xrightarrow{\sim} & \holim_{\bdel^{\leq k-1}} \left( \xymatrix{
    P_n (UQ) \ar@<-.5ex>[r] \ar@<.5ex>[r] &
    P_n ((UQ)^2) \ar@<-1ex>[r] \ar[r] \ar@<1ex>[r] & 
    P_n ((UQ)^3) \cdots
    }\right)\\
      \simeq & \holim_{\bdel^{\leq k-1}} \left(\xymatrix{ 
    \tau_n J^{(1)}\circ_{\Cal{O}}(-) \ar@<.5 ex>[r] \ar@<-.5ex>[r] &
    \tau_n J^{(2)} \circ_{\Cal{O}}(-) \ar@<1 ex>[r] \ar[r] \ar@<-1 ex>[r] &
    \tau_n J^{(3)} \circ_{\Cal{O}}(-) \cdots }\right) 
\end{align*}
and \begin{align*}
    D_n \Id \xrightarrow{\sim} & \holim_{\bdel^{\leq k-1}} \left( \xymatrix{
    D_n (UQ) \ar@<-.5ex>[r] \ar@<.5ex>[r] &
    D_n ((UQ)^2)  \ar@<-1ex>[r] \ar[r] \ar@<1ex>[r] & 
    D_n ((UQ)^3) \cdots
    }\right) \\
     \simeq & \holim_{\bdel^{\leq k-1}} \left(\xymatrix{ 
    i_n J^{(1)}\circ_{\Cal{O}}(-) \ar@<.5 ex>[r] \ar@<-.5ex>[r] &
    i_n J^{(2)} \circ_{\Cal{O}}(-) \ar@<1 ex>[r] \ar[r] \ar@<-1 ex>[r] &
    i_n J^{(3)} \circ_{\Cal{O}}(-) \cdots }\right)
\end{align*}
whenever $k\geq n\geq 1$. 

Note there is an equivalence of restricted diagrams \[(\tau_n J^{(\bullet+1)}\circ_{\Cal{O}}(-) )|_{\bdel^{\leq k-1}} \simeq P_n((UQ)^{\bullet+1})|_{\bdel^{\leq k-1}}\] (resp. $(i_n J^{(\bullet+1)}\circ_{\Cal{O}}(-))|_{\bdel^{\leq k-1}} \simeq D_n((UQ)^{\bullet+1})|_{\bdel^{\leq k-1}}$) by first replacing the coface $k$-cube associated to \[\Id \to (UQ)^{\bullet+1}\] by the $k$-cube $\Cal{Z}_k$ (see (\ref{eq:Z_k}) below) and then applying $\tau_n$ (resp. $i_n$) objectwise. \begin{equation} \label{eq:Z_k} \{\Cal{P}(k) \ni T \mapsto \Cal{Z}_k(T) = (Z_1 \circ_{\Cal{O}} \cdots \circ_{\Cal{O}} Z_k) \circ_{\Cal{O}} (-) \}  \text{ such that $Z_i=\begin{cases} J & i\in T \\ \Cal{O} & i \notin T \end{cases}$} \end{equation} 

 We then use the corresponding models for $\widetilde{\mathbb{D}_n}$ from Proposition \ref{prop:tay-tow-FM} and compute the $n$-th derivatives via cross effects to obtain equivalences \begin{align} \label{eq:model-for-nth-deriv}
    \widetilde{\partial_n} \Id \xrightarrow{\sim} & 
    \holim_{\bdel^{\leq k-1}} \left( \xymatrix{
    \widetilde{\partial_n} (UQ) \ar@<-.5ex>[r] \ar@<.5ex>[r] &
    \widetilde{\partial_n} ((UQ)^2) \ar@<-1ex>[r] \ar[r] \ar@<1ex>[r] & 
    \widetilde{\partial_n} ((UQ)^3) \cdots
    }\right) 
    \\
    \simeq & \holim_{\bdel^{\leq k-1}} \left(\xymatrix{ 
     (\tau_1\Cal{O})^{(1)}[n] \ar@<.5 ex>[r] \ar@<-.5ex>[r] &
    (\tau_1\Cal{O})^{(2)}[n] \ar@<1 ex>[r] \ar[r] \ar@<-1 ex>[r] &
    (\tau_1\Cal{O})^{(3)}[n] \cdots }\right). \nonumber
\end{align}
for $k\geq n\geq 1$.

\begin{example} \label{ex:C(O)} We sketch this process for $k=n=2$. Note, there is an isomorphism of square diagrams of the form \[ 
    \xymatrix{ \Id \ar[r]^-{d^0} \ar[d]^-{d^0} & UQ \ar[d]^-{d^0} \\ UQ \ar[r]^-{d^1} & UQUQ } \;\; \raisebox{-4ex}{$\cong$} \;\; 
    \xymatrix{ (\Cal{O} \circ_{\Cal{O}} \Cal{O}) \circ_{\Cal{O}} (-) \ar[r]^-{d^0} \ar[d]^-{d^0} &
        (\Cal{O} \circ_{\Cal{O}} J) \circ_{\Cal{O}} (-) \ar[d]^-{d^0} \\
        (J \circ_{\Cal{O}} \Cal{O}) \circ_{\Cal{O}} (-) \ar[r]^-{d^1} &
        (J \circ_{\Cal{O}} J)\circ_{\Cal{O}} (-).}\]
Taking $2$-homogeneous layers, we obtain an equivalence of homotopy pullback squares
\[ 
    \xymatrix{ D_2\Id \ar[r]^-{d^0} \ar[d]^-{d^0} & D_2(UQ) \ar[d]^-{d^0} \\ D_2(UQ) \ar[r]^-{d^1} & D_2(UQUQ) } \;\; \raisebox{-4ex}{$\simeq$} \;\; 
    \xymatrix{ i_2(\Cal{O} \circ_{\Cal{O}} \Cal{O}) \circ_{\Cal{O}} (-) \ar[r]^-{d^0} \ar[d]^-{d^0} &
        i_2(\Cal{O} \circ_{\Cal{O}} J) \circ_{\Cal{O}} (-) \ar[d]^-{d^0} \\
        i_2(J \circ_{\Cal{O}} \Cal{O}) \circ_{\Cal{O}} (-) \ar[r]^-{d^1} &
        i_2(J \circ_{\Cal{O}} J)\circ_{\Cal{O}} (-).}\]

The associated lifts $\widetilde{\mathbb{D}_2}(-)$ to functors on $\mathsf{Mod}_{\Cal{O}[1]}$ from Proposition \ref{prop:tay-tow-FM} then fit into a homotopy pullback square
\[ 
    \xymatrix{ 
        \widetilde{\mathbb{D}_2}\Id \ar[r]^-{d^0} \ar[d]^-{d^0} & 
        (\Cal{O}\circ_{\Cal{O}} \tau_1 \Cal{O})[2] \wedge_{\Cal{O}[1]^{\wr 2}} (-)^{\wedge 2} \ar[d]^-{d^0} \\
        (\tau_1 \Cal{O} \circ_{\Cal{O}} \Cal{O})[2] \wedge_{\Cal{O}[1]^{\wr 2}} (-)^{\wedge 2} \ar[r]^-{d^1} &
        (\tau_1 \Cal{O} \circ_{\Cal{O}} \tau_1 \Cal{O} )[2] \wedge_{\Cal{O}[1]^{\wr 2}} (-)^{\wedge 2}  }\]
which by taking cross effects $\cro_2$ then provides an equivalence of homotopy pullback squares 
\[ 
    \xymatrix{ \widetilde{\partial_2}\Id \ar[r]^-{d^0} \ar[d]^-{d^0} & \widetilde{\partial_2} (UQ) \ar[d]^-{d^0} \\ \widetilde{\partial_2}(UQ) \ar[r]^-{d^1} & \widetilde{\partial_2} (UQUQ) }  \raisebox{-4ex}{$\simeq$}  
    \xymatrix{ \widetilde{\partial_2} \Id \ar[d]^-{d^0} \ar[r]^-{d^0} &
        \tau_1 \Cal{O}[2] \ar[d]^-{d^0} \\
        \tau_1 \Cal{O} [2]\ar[r]^-{d^1} &
        (\tau_1\Cal{O} \circ_{\Cal{O}} \tau_1\Cal{O})[2]}
        \raisebox{-4ex}{$\simeq$} 
    \xymatrix{ \partial_2 \Id \ar[d]^-{d^0} \ar[r]^-{d^0} &
        J[2] \ar[d]^-{d^0} \\
        J[2]\ar[r]^-{d^1} &
        (J\circ_{\Cal{O}} J)[2].}\]
\end{example}

\begin{remark} \label{prop:der_n-model-equivs} It follows then that $\partial_* \Id$ is obtained as $\holim_{\bdel} C(\Cal{O})\simeq \Tot C(\Cal{O})$, where $C(\Cal{O})$ is the following cosimplicial diagram (showing only coface maps)\begin{align} \label{al:CO}
    C(\Cal{O}) &= \left( \xymatrix{
    J\circ_{\Cal{O}} \Cal{O} \ar@<-.5ex>[r] \ar@<.5ex>[r] &
    J\circ_{\Cal{O}} J\circ_{\Cal{O}}\Cal{O} \ar@<-1ex>[r] \ar[r] \ar@<1ex>[r] & 
    J\circ_{\Cal{O}} J\circ_{\Cal{O}} J\circ_{\Cal{O}}\Cal{O} \cdots
    }\right) 
    \\
    & \cong \left( \xymatrix{
    J \ar@<-.5ex>[r] \ar@<.5ex>[r] &
    J\circ_{\Cal{O}} J \ar@<-1ex>[r] \ar[r] \ar@<1ex>[r] & 
    J\circ_{\Cal{O}} J\circ_{\Cal{O}} J\cdots
    }\right), \nonumber
\end{align}
with coface maps as in (\ref{TQ-res-1}), i.e., $C(\Cal{O})= J^{(\bullet+1)}$. In other words $C(\Cal{O})$ provides a rigidification of the diagram $\partial_* (UQ)^{\bullet+1}$ whose terms are \textit{a priori} defined only up to homotopy.
\end{remark} 

\section{The box-product of cosimplicial objects}
The aim of this section is to introduce the box product $\square$ for cosimplicial objects in a monoidal category $(\mathsf{C}, \otimes, \mathbf{1})$ as first introduced by Batanin \cite{Bat-box}. For nice categories $\mathsf{C}$ (e.g.,  $\mathsf{C}$ closed, symmetric monoidal), the box product endows $\mathsf{C}^{\bdel}$ with a monoidal structure, and cosimplicial objects which admits a monoidal pairing with respect to $\square$ inherit an $A_{\infty}$-monoidal pairing on their totalizations (see, e.g.,  McClure-Smith \cite[3.1]{MS-box}).

Our use of the box product will be to produce a homotopy-coherent (i.e., $A_{\infty}$-) composition on the derivatives of the identity, modeled as $\Tot C(\Cal{O})$, by demonstrating a natural pairing $C(\Cal{O})\square C(\Cal{O})\to C(\Cal{O})$ (Example \ref{ex:C(O)-box}).

\begin{definition} \label{def:box-product} 
Let $(\mathsf{C}, \otimes, \mathbf{1})$ be a monoidal category and $X,Y\in \textsf{C}^{\bdel}$. Define their \textit{box product} $X\square Y\in \textsf{C}^{\bdel}$ at level $n$ by \[
(X\square Y)^n:=
\colim \left( \xymatrix{ 
    \displaystyle{\coprod_{p+q=n}} X^{p}\otimes Y^{q} & 
    \ar@<-1 ex>[l] \ar@<.5 ex>[l] 
    \displaystyle{\coprod_{r+s=n-1}} X^{r}\otimes Y^{s} 
}\right) \]
where the maps are induced by $\id\otimes d^0$ and $d^{r+1}\otimes \id$. The object $X\square Y$ inherits cosimplicial structure via coface maps $d^i\colon (X\square Y)^{n}\to (X\square Y)^{n+1}$ induced by
\[\begin{cases} 
X^p\otimes Y^q \xrightarrow{\;\; d^i\otimes \id\;\;} X^{p+1}\otimes Y^{q} & i \leq p
\\
X^p\otimes Y^q \xrightarrow{\;\; \id \otimes d^{i-p} \;\;} X^{p}\otimes Y^{q+1} & i > p
\end{cases}
\]
and codegeneracy maps $s^j\colon (X\square Y)^n \to (X\square Y)^{n-1}$ induced by
\[\begin{cases} 
X^p\otimes Y^q \xrightarrow{\;\; s^j\otimes \id\;\;} X^{p-1}\otimes Y^{q} & j < p
\\
X^p\otimes Y^q \xrightarrow{\;\; \id \otimes s^{j-p} \;\;} X^{p}\otimes Y^{q-1} & j \geq p
\end{cases}
\]
see also Ching-Harper \cite[\textsection 4]{Ching-Harper-Koszul}.
\end{definition}

\begin{remark}Note, $(X\square Y)^0\cong X^0\otimes Y^0$, $(X\square Y)^{1}$ and $(X\square Y)^{2}$ may be computed as the colimits of \[ \xymatrix{ X^0 \otimes Y^1 & \\ X^0 \otimes Y^0 \ar[r]_-{d^1\otimes \id} \ar[u]^-{\id \otimes d^0} & X^1 \otimes Y^0} \; \raisebox{-8ex}{\text{ and }}  \; \xymatrix{ 
    X^0 \otimes Y^2 & &
    \\ 
    X^0 \otimes Y^1 \ar[r]_-{d^1\otimes \id} \ar[u]^-{\id \otimes d^0} 
    & 
    X^1 \otimes Y^1 & \\
    & X^1 \otimes Y^0 \ar[u]^-{\id \otimes d^0} \ar[r]_-{d^2\otimes \id} 
    &
    X^2 \otimes Y^0}
    \] respectively, and in general $(X\square Y)^n$ may be computed as the colimit of the staircase diagram \begin{equation}\xymatrix{
    X^0 \otimes Y^n & &\\
    X^0 \otimes Y^{n-1} \ar[r]_{d^1\otimes \id} \ar[u]^{\id\otimes d^0} & X^1 \otimes Y^{n-1} &\\
    & \ddots \ar[u] \ar[r]& X^{n-1}\otimes Y^1 &\\
    && X^{n-1}\otimes Y^0 \ar[u]^{\id\otimes d^0} \ar[r]_{d^n\otimes \id} &X^n\otimes Y^0 
}\end{equation}
\end{remark}

In particular, if $(\mathsf{C},\otimes, \mathbf{1})$ is closed, symmetric monoidal then $\square$ defines a monoidal category $(\textsf{C}^{\bdel}, \square, \underline{\mathbf{1}})$, here $\underline{\mathbf{1}}$ is the constant cosimplicial object on the unit $\mathbf{1}\in\mathsf{C}$ (see, e.g.,  Batanin \cite{Bat-box}). 

\begin{example} \label{ex:C(O)-box}
Recall the cosimplicial symmetric sequence $C(\Cal{O})=J^{(\bullet+1)}$ from \eqref{al:CO}. We observe that $C(\Cal{O})$ admits a pairing $C(\Cal{O}) \boxcirc C(\Cal{O}) \xrightarrow{m} C(\Cal{O})$, where $\boxcirc$ denotes the box product in $\SymSeq^{\bdel}_{\Spt}$, induced as follows. Let $c$ denote the operad composition map $c\colon J\circ J\to J$. Then, \begin{align*} 
    &(C(\Cal{O}) \boxcirc C(\Cal{O}))^0\cong J \circ J \xrightarrow{c} J=C(\Cal{O})^0 \end{align*}
    
For level $1$ we observe that there are maps \[ m_{0,1} \colon J\circ J\circ_{\Cal{O}} J \to J\circ_{\Cal{O}} J \; \text{ and } \; m_{1,0} \colon J \circ_{\Cal{O}} J \circ J \to J\circ_{\Cal{O}} J \]
induced by $J\circ J\to J$ which induces $m$ via the following commuting square
\begin{equation}\xymatrix{
    J \circ J\circ_{\Cal{O}} J \ar[r]^-{m_{0,1}}
    & 
    J \circ_{\Cal{O}} J 
    \\
    J \circ J \ar[u]^-{\id\circ  d^0} \ar[r]^-{d^1\circ \id}
    &
    J\circ_{\Cal{O}} J\circ J \ar[u]^-{m_{1,0}}
} 
\end{equation}

More generally, there are maps of the form \[m_{p,q}\colon J^{(p)} \circ J^{(q)} \to J^{(p+q)} \;\; \text{for $p+q=n$, $p,q\geq 0$}\] induced by $c$, which induces the pairing $m$ at level $n$. 
\end{example}

\begin{remark}
    The above construction is entirely analogous to the following example found in McClure-Smith \cite{MS-box} that the based loop space $\Omega X$ of $X\in \Top$ admits an $A_{\infty}$ composition induced by an underlying $\square$-pairing. In this case, $\Omega X$ is modeled as the totalization of the cobar complex $c(X)$ built with respect to the natural comultiplication (with coaugmentation) given by the diagonal $\delta\colon X\to X\times X$. 
    
    It follows that $c(X)^{p}\cong X^{\times p}$ and the pairing $c(X)\square c(X)\to c(X)$ is induced by the natural isomorphisms $X^{\times p}\times X^{\times q}\cong X^{\times p+q}$. Further, McClure-Smith  show that $\Tot c(X)$ is an algebra over the (nonsymmetric) coendmorphism operad on $\Delta^{\bullet}$, i.e., \[ \mathbb{A}[n]=\Map_{\Dres}^{\mathsf{Top}}\left(\Delta^{\bullet}, (\Delta^{\bullet})^{\square n} \right)\] which satisfies $\mathbb{A}[0]=\pt$ and $\mathbb{A}[n]\xrightarrow{\sim} \pt$ for $n\geq 1$ (in fact $\Delta^n$ and $(\Delta^{\bullet} \square \Delta^{\bullet})^n$ are \textit{homeomorphic}), and that with respect to this structure $\Tot c(X)\simeq \Omega X$ as $A_{\infty}$-monoids.  \label{ex:loop-space-A-infty-monoid}
\end{remark}

\subsection{The box product in $\SymSeq^{\bdel}$}
Our aim now is to build a framework in which we can work with the structure captured by Example \ref{ex:C(O)-box}, e.g.,  by considering the box-product in the category of cosimplicial objects in $(\SymSeq_{\mathsf{C}}, \circ, I)$ of symmetric sequences for $(\mathsf{C},\otimes, \mathbf{1})$ some closed symmetric monoidal category. 

The main difficulty is that the composition product of symmetric sequences does not always commute with colimits taken in the right hand entry. That is, for $B\colon \mathcal{I}\to \SymSeq_{\mathsf{C}}$ a small diagram and $A\in \SymSeq_{\mathsf{C}}$, the universal map \begin{equation} \label{eq:colim-circ}\colim_{i\in \mathcal{I}} (A\circ B_i) \to A \circ (\colim_{i\in \mathcal{I}} B_i) \end{equation} is not an isomorphism in general. Thus the box-product fails to be strictly monoidal in this setting. 

Let us write $\SymSeq=\SymSeq_{\mathsf{C}}$ and $\boxcirc$ for the box-product in $\SymSeq^{\bdel}$ (in words we refer to $\boxcirc$ as the \textit{box-circle product}). Let $\Cal{X},\Cal{Y},\Cal{Z}, \dots$ be cosimplicial symmetric sequences. We will systematically interpret expressions of the form $\Cal{X} \boxcirc \Cal{Y} \boxcirc \Cal{Z}$ to be expanded \textit{from the left}, i.e., 
    \[\Cal{X} \boxcirc \Cal{Y} \boxcirc \Cal{Z} := (\Cal{X} \boxcirc \Cal{Y}) \boxcirc \Cal{Z}, \; \Cal{X} \boxcirc \Cal{Y} \boxcirc \Cal{Z} \boxcirc \Cal{W} := ((\Cal{X} \boxcirc \Cal{Y}) \boxcirc \Cal{Z} ) \boxcirc \Cal{W}, \;\dots \]
and note that via the universal map in (\ref{eq:colim-circ}) there is always a canonical comparison map $\theta$ of the form    
    \begin{equation} \label{eq:theta} \theta\colon  \Cal{X}\boxcirc \Cal{Y}\boxcirc \Cal{Z}=(\Cal{X} \boxcirc \Cal{Y}) \boxcirc \Cal{Z}  \to \Cal{X} \boxcirc (\Cal{Y} \boxcirc \Cal{Z})\end{equation}
which likely fails to be invertible. However, $\theta$ is sufficient to provide a suitable description of monoids with respect to $\boxcirc$, i.e., Definition \ref{boxcirc-monoid}, below. First, we note that the unit $I\in \SymSeq_{\mathsf{C}}$ induces a unit $\underline{I} \in \SymSeq_{\mathsf{C}}$ as the constant cosimplicial object on $I$ in that there are isomorphisms \[ \Cal{X} \boxcirc \underline{I} \cong \Cal{X} \cong \underline{I} \boxcirc \Cal{X}.\] 

For instance, the right isomorphism is obtained by noting that for any $p,q$ the map $d^{p+1}\circ\id$ in the following \[ \xymatrix{
        \underline{I}^p \circ \Cal{X}^{q+1} 
        & \\
        \underline{I}^p \circ \Cal{X}^q \ar[u]^-{\id\circ d^0} \ar[r]^-{d^{p+1}\circ \id}
        &
        \underline{I}^{p+1} \circ \Cal{X}^{q} 
        } \]
    is just the identity (and hence has an inverse). Therefore, the inclusion of the vertex $\underline{I}^0\circ \Cal{X}^n$ into the diagram defining $(\Cal{X}\boxcirc \Cal{X})^n$ is right cofinal (i.e., induces an isomorphism on colimits).

\begin{definition} \label{boxcirc-monoid} By $\boxcirc$-monoid in $\SymSeq^{\bdel}$, we mean a cosimplicial symmetric sequence $\Cal{X}$ together with maps $m\colon \Cal{X} \boxcirc \Cal{X} \to \Cal{X}$ and $u\colon \underline{I} \to \Cal{X}$ so that the following associativity (\ref{eq:assoc-boxcirc}) and unitality (\ref{eq:unit-boxcirc}) diagrams commute
    \begin{equation} \label{eq:assoc-boxcirc} \xymatrix{  
        \Cal{X} \boxcirc \Cal{X} \boxcirc \Cal{X} \ar[r]^{\theta} \ar[d]^-{=}&
        \Cal{X} \boxcirc (\Cal{X} \boxcirc \Cal{X}) \ar[r]^-{\id \boxcirc m} &
        \Cal{X} \boxcirc \Cal{X} \ar[d]^-{m} \\
        (\Cal{X} \boxcirc \Cal{X}) \boxcirc \Cal{X} \ar[r]^-{m \boxcirc \id} &
        \Cal{X} \boxcirc \Cal{X} \ar[r]^-{m} &
        \Cal{X}
        }\end{equation}
    and
   \begin{equation} \label{eq:unit-boxcirc} \xymatrix{ 
        \Cal{X} \boxcirc \underline{I} \ar[r]^-{\id\boxcirc u} \ar[dr]_-{\cong}
        &
        \Cal{X} \boxcirc \Cal{X} \ar[d]^-{m}
        & 
        \underline{I}\boxcirc \Cal{X} \ar[l]_-{u\boxcirc \id} \ar[dl]^-{\cong}\\
        & \Cal{X}  &
        } 
    \end{equation}
\end{definition} 

\begin{remark}
We remark that in the language of Ching \cite{Ching-oplax} (see also Day-Street \cite{Day-St}), $\boxcirc$ admits a \textit{normal oplax monoidal structure} by defining \[ \Cal{X}_1\boxcirc \cdots \boxcirc \Cal{X}_k:= (\cdots (( \Cal{X}_1 \boxcirc \Cal{X}_2) \boxcirc \Cal{X}_3) \cdots)  \boxcirc \Cal{X}_k \] and obtaining grouping maps from the universal map in (\ref{eq:colim-circ}). Our notion of $\boxcirc$-monoids are \textit{normal oplax monoids} with respect to such structure by appealing to Ching \cite[3.4]{Ching-oplax}, noting in particular that four-fold and higher associativity diagrams are known to commute given the commutativity of (\ref{eq:assoc-boxcirc}).
\end{remark}

\begin{proposition}
    \label{def:box-circ-prod-map}
    The cosimplicial symmetric sequence $C(\Cal{O})$ (see \eqref{al:CO}) admits a natural $\boxcirc$-monoid structure, i.e., there are maps $m\colon C(\Cal{O})\boxcirc C(\Cal{O})\to C(\Cal{O})$ and $u \colon \underline{I}\to C(\Cal{O})$ which satisfy associativity and unitality. 
\end{proposition}

\begin{proof}
    The map $m$ is that constructed in Example \ref{ex:C(O)-box}. The unit $I\to J$ provides a coaugmentation $I\to C(\Cal{O})$ which in turn induces a map $u\colon \underline{I}\to C(\Cal{O})$. 
    
    Associativity (\ref{eq:assoc-boxcirc}) follows from a routine calculation, observing that \[d^0 \colon (\fco \boxcirc \fco)^{q} \to (\fco \boxcirc \fco)^{q+1}\] is induced by $d^0 \circ \id\colon \fco^r\circ \fco^s\to \fco^{r+1} \circ \fco^s$ for $r+s=q$. Similarly, the right-hand triangle from the unitality diagram (\ref{eq:unit-boxcirc}) is granted by the following commuting diagrams\[ \xymatrix{
        \underline{I}^p \circ C(\Cal{O})^q \ar[r]^-{u^p\circ \id} \ar[d]^-{\cong}
        & 
        C(\Cal{O})^p \circ \fco^q \ar[d]^-{m_{p,q}}
        \\ 
        \fco^q \ar[r]^-{d^0\cdots d^0} 
        & 
        \fco^{p+q}
        }\]
    for all $p,q$. A similar argument provides the commutativity of the other side of the unitality diagram. 
\end{proof}

Theorem \ref{main}(a) is then obtained as a corollary to the following proposition, the proof of which is deferred to Section \ref{sec:1a}. As such, the aim of the following sections is to set up a precise framework to describe what is meant by $A_{\infty}$-operad.

\begin{proposition}  \label{main-1}
    If $\Cal{X}$ is a $\boxcirc$-monoid in $\SymSeq_{\Spt}^{\bdel}$, then $\Tot \Cal{X}$ is an $A_{\infty}$-monoid with respect to the composition product (i.e., $A_{\infty}$-operad). 
\end{proposition}

\section{$\mathbf{N}$-colored operads with levels} \label{sec:N-lev}
In this section we develop our theory of $\mathbf{N}=\{0,1,2,\dots\}$-colored operads with levels, which we refer to as $\Nl$-operads. The motivating principle behind our constructions is to provide a framework to fatten-up the usual notion of operads and their algebras. For this section $(\mathsf{C},\otimes, \mathbf{1})$ will denote a given cocomplete closed, symmetric monoidal category with initial object $\ptC$. We first recall the classical theory of colored operads. 

\subsection{Colored operads}

Colored operads (sometimes also referred to as \textit{multicategories}) offer a generalization of operads to encode more nuanced algebraic operations on their algebras. We give an overview of their pertinent details below and refer the reader to Leinster \cite{Lei} or Elmendorf-Mandell \cite{EM-colop} for more information. As before, we will only need to consider colored operads in the category of spectra. 

\begin{definition} Let $C$ be a nonempty set, i.e., a set of \textit{colors}. A \textit{$C$-colored operad} $\Cal{M}$ in $\mathsf{C}$ consists of \begin{itemize}
    \item Objects $\mathcal{M}(c_1,\dots,c_n;d)\in \mathsf{C}$ for all $(c_1,\dots,c_n;d)\in C^{\times n}\times C$ and $n\geq 0$
    \item A unit map $\mathbf{1} \to \Cal{M}(c;c)$ for all $c\in C$
    \item Composition maps of the form \begin{align*} \label{eq:colored-op-comp}
        \Cal{M}(c_1,\dots,c_n;d)\otimes \Cal{M}(p_{1,1}, \dots, p_{1,k_1};c_1)&\otimes \cdots \otimes \Cal{M}(p_{n,1},\dots,p_{n,k_n};c_n)\\
    &\to  \Cal{M}(p_{1,1},\dots, p_{n,k_n};d)
\end{align*}
\end{itemize} 
subject equivariance, associativity and unitality conditions (see, e.g.,  \cite[2.1]{EM-colop}).
\end{definition} 

An \textit{algebra} over $\mathcal{M}$ is a $C$-colored object, i.e., $X=\{X_c\}_{c\in C}$ such that $X_c\in \mathsf{C}$ for all $c\in C$, together with maps for each tuple $(c_1,\dots,c_n;d)\in C^{\times n}\times C$ of the form \[\Cal{M}(c_1,\dots,c_n;d)\otimes X_{c_1}\otimes \cdots \otimes X_{c_n} \to X_d\] the collection of which is required to satisfy equivariance, associativity and unitality conditions. 

Berger-Moerdijk provide a list of examples in \cite[\textsection 1.5]{BM-trees}; of note is that for $C=\{\pt\}$, a one-colored operad is just an operad in the classical sense. The following constructions are also motivated by White-Yau \cite{White-Yau} wherein a composition product for $C$-colored operads is provided.

\subsection{$\Nl$-objects}
The purpose of this section is to introduce the notion of a nonsymmetric, $\mathbf{N}$-colored sequence \textit{with levels} in $\mathsf{C}$. We will refer to these as \textit{$\Nl$-objects}. 
In our framework, $\Nl$-objects will play a role analogous to symmetric sequences for classical (one-color) operads, though we note that we do not yet impose any symmetric group actions on our $\Nl$-objects. Let $\mathbf{s}$ denote the set $\{1,\dots,s\}$ (note that $\mathbf{0}=\emptyset$). 

\begin{definition} 
\label{N-circhat-k}
For $k\geq 0$, let $\mathbf{N}^{\hat{\circ}k}$ denote the set of tuples of orbits
\[ 
\mathbf{N}^{\hat{\circ}k}:=\left\{ \left(n^1,(n^2_1,\cdots,n^2_{n^1})_{\Sigma_{n^1}} , \cdots, (n^{k}_1,\cdots, n^k_{n^{k-1}})_{\Sigma_{n^{k-1}}}\right): n^j_i \geq 0\; \forall i,j\right\}
\]
where $n^{j}$ is inductively defined as $\sum_{i=1}^{n^{j-1}} n^j_i$ and we set $n^0:=1$. We then treat $\Ncirc{k}$ as a category with only identity morphisms. 
\end{definition}

Note that the superscripts in Definition \ref{N-circhat-k} are used for indexing and are not powers, we will adhere to this convention throughout the document. 
Elements $\underline{p}\in \Ncirc{k}$ will be referred to as \textit{profiles}, we will often suppress the orbit subscript and write $(n_1,\dots,n_s)$ for the orbit $(n_1,\dots,n_s)_{\Sigma_s}$.

\begin{definition} Given $\underline{p}=(n^1, \dots, (n^{k}_i)_{i\in \mathbf{n^{k-1}}})\in \Ncirc{k}$, we define the \textit{weight} of $\underline{p}$ to be the integer
$n^k=\sum_{i\in \mathbf{n^{k-1}}} n^k_i$. For $t\in \mathbf{N}$, we write $\Ncirc{k}[t]$ for the set of profiles $\underline{p}\in \Ncirc{k}$ of weight $t$.
\end{definition} 

\begin{example}
\label{Nl-ex-1}Computing small examples we see \begin{align*}
\Ncirc{0} & = \{\emptyset\},
&
\Ncirc{2} & \cong \{ (n, (k_1,\dots,k_n)): n,k_i\geq 0\},
\\
\Ncirc{1} & \cong \mathbf{N},
&
\Ncirc{3} & \cong \{(n, (k_1,\dots,k_n), (t_1,\dots, t_k): k=k_1+\dots+k_n, n,k_i,t_j\geq 0\}. \nonumber
\end{align*}
\end{example}

\begin{remark}
Note that profiles in $\Ncirc{\ell}$ are in bijective correspondence to indexing factors of $\ell$-fold iterates of $\circhat$ from (\ref{comp-prod-nonsym}), therefore objects indexed on $\Ncirc{\ell}$ naturally arise when evaluating $\ell$-fold iterates of the composition product of symmetric sequences (Definition \ref{comp-prod-def-orb}) from the left. 
\end{remark}

Given $\underline{p}=(n^1, (n^2_i)_{i\in\mathbf{n^1}}\dots, (n^{\ell}_i)_{i\in\mathbf{n}^{\ell-1}})\in\Ncirc{\ell}[t]$, the term \[(X_1\circ \cdots \circ X_{\ell})[\underline{p}]\] is the collection of factors in $(X_1\circ\cdots \circ X_{\ell})[t]$ corresponding to the indexing tuples \[(n^j_1,\dots,n^j_{n^1})_{\Sigma_{n^{j-1}}}\in\mathsf{Sum}_{n^{j-1}}^{n^j},\] for $j=1,\dots, \ell$.

\begin{definition} 
Given profiles $\underline{p},\underline{q}\in\Ncirc{k}$ we define their \textit{amalgamation} $\underline{p}\amalg \underline{q}$ to be the orbit of the levelwise disjoint union of the two profiles. In other words, given \begin{align*} 
\underline{p} & =(n^1,(n^2_{i})_{i\in \mathbf{n^1}}, (n^3_{i})_{i\in \mathbf{n^2}},\dots, (n^{k}_{i})_{i\in \mathbf{n^{k-1}}}), \\
\underline{q} & =(m^1,(m^2_{j})_{j\in \mathbf{m^1}}, (m^3_{j})_{j\in \mathbf{m^2}},\dots, (m^{k}_{j})_{j\in \mathbf{m^{k-1}}}),
\end{align*}
then $\underline{p}\amalg\underline{q}$ is given by\begin{equation*}
\underline{p}\amalg \underline{q}:=\Big((n^1,m^1),\big( (n^2_{i})_{i\in \mathbf{n^1}} \amalg (m^2_{j})_{j\in \mathbf{m^1}} \big) , \dots, \big( (n^{k}_{i})_{i\in \mathbf{n^{k-1}}} \amalg (m^{k}_{j})_{j\in \mathbf{m^{k-1}}}\big) \Big).\end{equation*}
\end{definition}

\begin{remark} Note that $\underline{p}\amalg \underline{q}$ is \textit{not} an element of any $\Ncirc{k}$ as its first entry is not a singleton. However, if $\underline{p}_i\in \Ncirc{k}[t_i]$ for $i=1,\dots,n$ then \[\big( n, \underline{p_1}\amalg\cdots \amalg \underline{p_n} \big)\in \Ncirc{k+1}[t_1+\cdots+t_n].\]

For instance, if $\underline{p}=(2,(2,3))$ and $\underline{q}=(3,(2,3,4))$ then \[\big(2,\underline{p}\amalg\underline{q}\big)=\big(2,(2,3)_{\Sigma_2},(2,3,2,3,4)_{\Sigma_5}\big)\in \Ncirc{3}[14].\] 
\end{remark}

\begin{definition} 
An \textit{$\Nl$-object} $\mathcal{P}$ in a symmetric monoidal category $\mathsf{C}$ is a functor \[ \Cal{P}\colon \coprod_{\ell\geq 0} \Ncirc{\ell}\times \mathbf{N} \to \mathsf{C}.\] 

Equivalently, $\mathcal{P}=(\mathcal{P}_{k})_{k\geq 0}$ such that $\mathcal{P}_{k}$ is a functor $\Ncirc{k}\times \mathbf{N} \to \mathsf{C}$. We also refer to $\Nl$-objects as \textit{$\mathbf{N}$-colored objects with levels}. 
We further say an $\Nl$-object $\Cal{P}$ is \textit{reduced} if \begin{itemize}
    \item For $\ell\geq 1$, $\Cal{P}_{\ell}(\underline{p};t)= \ptC$ if $\underline{p}\notin \Ncirc{\ell}[t]$
    \item $\Cal{P}_0(\emptyset;1)=\mathbf{1}$ 
    \item $\Cal{P}_0(\emptyset; n)=\ptC$ for $n\neq 1$. 
\end{itemize} 
Recall that $\ptC$ denotes the initial object of $\mathsf{C}$.
\end{definition}
Note if $\Cal{P}$ is reduced then $\Cal{P}$ is determined by a functor $\coprod_{\ell\geq 0} \Ncirc{\ell} \to \mathsf{C}$. We will mostly be concerned with reduced $\Nl$-objects, but benefit from this more general definition when we discuss algebras in Section \ref{sec:algebras}.

\subsection{A composition product for $\Nl$-objects}
The aim of this section is develop a monoidal composition product for $\Nl$-objects so that we may encode $\Nl$-operads as monoids. 

\begin{definition}
\label{circ-dot-composite}
Let $\underline{p}=(n^1,(n^2_i)_{i\in\mathbf{n^1}},\dots,(n^k_i)_{i\in\mathbf{n^{k-1}}}) \in \Ncirc{k}$ and let $\ell_1,\dots,\ell_k\geq 0$ be given. Let $\underline{Q}$ denote a collection of unordered sequences of profiles $(\underline{q}_1^j,\cdots, \underline{q}^j_{n^{j-1}})$ for $j=1,\dots,k$ such that $\underline{q}^j_i\in\Ncirc{\ell_j}[n^j_i]$. 

We define the \textit{composite} of $\underline{p}$ and $\underline{Q}$ to be the profile $\underline{p}\circ \underline{Q}\in\Ncirc{(\ell_1+\dots+\ell_k)}$ given as follows \[\underline{p}\circ \underline{Q}:=\big(\underline{q^1}, \underline{q}^2_1\amalg \cdots \amalg \underline{q}^2_{n^1}, \cdots, \underline{q}_1^k\amalg \cdots \amalg \underline{q}_{n^{k-1}}^k\big).\]
\end{definition}

Let \[\Ncirc{k} \ltimes \left(\coprod_{\ell_1,\cdots,\ell_k\geq 0}\Ncirc{\ell_1} \times \cdots \times \Ncirc{\ell_k}\right)\] be the collection of all pairs $(\underline{p}, \underline{Q})$ such that 
\begin{align*} 
\underline{p}=(n^1,(n^2_i)_{i\in\mathbf{n^1}},\dots,(n^k_i)_{i\in\mathbf{n^{k-1}}}), \quad \quad
\underline{Q}=(\underline{q}^1,\cdots ,(\underline{q}^k_j)_{j\in \mathbf{n^{k-1}}})
\end{align*}
so that the composite $\underline{p} \circ \underline{Q}$ is defined (i.e., $\underline{q}^j_i\in \Ncirc{\ell_j}[n^j_i]$).  

\begin{remark} \label{rem:trees-n-stuff}
    It is convenient to think of an element $\underline{p}=(n^1,\dots,(n^k_i))\in \Ncirc{k}[t]$ as describing a family of planar rooted trees (see, e.g. \cite{Ching-thesis}) with $t$ leaves and $k$ levels. More precisely, the numbers $n^j_i$ describe the valence (number of input edges) to the $i$-th node at the $j$-th level, and a tree in this family is determined by a family of morphisms $\varphi_j \colon \mathbf{n^j}\to\mathbf{n^{j-1}}$ for $1\leq j < k$ such that $|\varphi_j^{-1}(i)|=n_i^j$ for all $i,j$. 
    
    Let $\underline{Q}$ be so that $\underline{p}\circ \underline{Q}$ is defined. From this perspective, a tree in the family corresponding to $\underline{p}\circ\underline{Q}$ is build by ``blowing up'' each node $n^j_i$ from $\underline{p}$ by a tree from the family corresponding to the profile $\underline{q^j_i}$ from $\underline{Q}$.
\end{remark}

\begin{definition} \label{def:tensor} We define the \textit{tensor $\tensorhat$} of reduced $\Nl$-objects $\Cal{Q}^1,\cdots,\Cal{Q}^k$ to be the left Kan extension of the following \begin{equation} \xymatrix{ 
    \coprod_{k\geq 0} \left( \Ncirc{k} \ltimes (\coprod_{\ell_1,\cdots,\ell_k\geq 0}\Ncirc{\ell_1} \times \cdots \times \Ncirc{\ell_k}) \right) \ar[rr]^-{\Cal{Q}^1 \circhat \cdots \circhat \Cal{Q}^k} \ar[d]^-{(\underline{p}, \underline{Q}) \mapsto \underline{p} \circ \underline{Q}} 
    && \mathsf{C} 
    \\ \coprod_{\ell\geq 0} \mathbf{N}^{\circhat \ell} \ar[rr]^-{\Cal{Q}^1 \tensorhat \cdots \tensorhat \Cal{Q}^k}_-{\text{left Kan ext.}} 
    && \mathsf{C} }\end{equation}
such that if $\underline{p}\circ \underline{Q}\in \Ncirc{\ell_1+\cdots+\ell_k}[t]$, then \[  (\Cal{Q}^1 \circhat \cdots \circhat \Cal{Q}^k)(\underline{p}\circ \underline{Q}; t) :=  \Cal{Q}_{\ell_1}^1 (\underline{q}^1;n^1) \otimes \bigotimes_{i\in \mathbf{n^1}} \Cal{Q}_{\ell_2}^{2}(\underline{q}^2_i;n^2_i) \cdots \otimes \bigotimes_{i\in\mathbf{n^{k-1}}} \Cal{Q}^k_{\ell_k}(\underline{q}^k_i;n^k_i) . \]
\end{definition}

Note then that $(\Cal{Q}^{\tensorhat k})_{\ell} \cong \coprod_{\ell_1+\dots+\ell_k=\ell} \Cal{Q}_{\ell_1} \hat{\circ} \cdots \hat{\circ} \Cal{Q}_{\ell_k}$, more specifically: \begin{equation}
   (\Cal{Q}^{\tensorhat k})_{\ell}(\underline{p};t) \cong \coprod_{\ell_1+\dots+\ell_k=\ell} \coprod_{\underline{p}=\underline{p}' \circ \underline{Q}} \; (\Cal{Q} \circhat \cdots \circhat \Cal{Q})(\underline{p}\circ \underline{Q}; t)
    \end{equation}
where we note that the summands $\ell_j$ are ordered.

\begin{definition} \label{ccp} 
Let $\Cal{P}$ and $\Cal{Q}$ be reduced $\Nl$-objects in $\mathsf{C}$. Their \textit{nonsymmetric composition product} $\circdot$ is defined as the coend $\Cal{P}_{-} \otimes_{\mathsf{N}} \Cal{Q}^{\tensorhat -}$ where $\mathsf{N}$ denotes the category of finite sets $\mathbf{n}$ for $n\geq 0$ with only identity morphisms. That is, \[(\Cal{P}{\odot} \Cal{Q})_{\ell}\cong \coprod_{k\geq 0} \Cal{P}_k \tensordot (\Cal{Q}^{\hat{\otimes} k})_{\ell}.
\]
\end{definition}

We use the notation $\tensordot$ to designate the product $\Cal{P}_k \tensordot (\Cal{Q}_{\ell_1} \circhat \cdots \circhat \Cal{Q}_{\ell_k})$ is evaluated at a profile $(\underline{p};t)$ as follows \[(\Cal{P}_k \tensordot (\Cal{Q}_{\ell_1}\circhat \cdots \circhat \Cal{Q}_{\ell_k})) (\underline{p};t) \cong \coprod_{\underline{p}=\underline{p}' \circ \underline{Q}} \Cal{P}_k(\underline{p}';s') \otimes (\Cal{Q}^1 \circhat \cdots \circhat \Cal{Q}^k)(\underline{p}\circ \underline{Q}; t) 
\] 
where $\underline{p}'\in\Ncirc{k}[s']$ and $\underline{Q}$ is a family $(\underline{q}_i^j)$ as in (\ref{circ-dot-composite}) with $\underline{q}_i^j\in\Ncirc{\ell_{j}}[s_i^j]$.

We necessarily then have \[\label{s-def}
\underline{p}'=\big(s^1, (s^2_1,\dots,s^2_{s^1}), \dots , (s^k_1,\dots, s^k_{s^{k-1}})\big)\] and can further describe $\Cal{P} \circdot \Cal{Q}$ as \begin{align}\label{p=p'Q} 
    (\Cal{P} \circdot \Cal{Q})_{\ell}(\underline{p};t) \cong \coprod_{\ell_1+\dots+\ell_k=\ell} \coprod_{\underline{p}=\underline{p}' \circ \underline{Q}} \Cal{P}_k(\underline{p}';s') \otimes \bigotimes_{j=1}^{k} \left( \bigotimes_{i\in\mathbf{n^j}} \Cal{Q}_{\ell_j}(\underline{q}^j_i;n^j_i) \right)
    \end{align}

\begin{example} 
We will evaluate $(\Cal{P}\circdot\Cal{Q})_3$ at \[\underline{p}=(n, (k_i)_{i\in \mathbf{n}}, (t_j)_{j\in \mathbf{k}})\in\Ncirc{3}[t]\] for $\Cal{P}, \Cal{Q}$ reduced $\Nl$-objects. Set $k:=k_1+\dots+k_n$, we observe \[ \left(\Cal{P}_2\tensordot(\Cal{Q}_1\hat{\circ}\Cal{Q}_2) \right)(\underline{p};t)=
\coprod_{\underline{p}=(n,(\underline{q}_1\amalg \cdots \amalg \underline{q}_n))} \Cal{P}_2(n, (s_1,\dots,s_n); t)\otimes \left( \Cal{Q}_1(n;n)\otimes \bigotimes_{i\in\mathbf{n}} \Cal{Q}_2 (\underline{q}_i;s_i)\right)\]
where  $\underline{q}_i\in \Ncirc{2}[s_i]$. 

Using the language of Remark \ref{rem:trees-n-stuff}, we think of the above as partitioning the set of nodes $(t_j)_{j\in\mathbf{k}}$ from $\underline{p}$ into $n$ sets of size $k_1,\dots,k_n$, e.g., by defining a map $\varphi\colon \mathbf{k}\to\mathbf{n}$ such that $|\varphi^{-1}(i)|=k_i$ for $i=1,\dots,n$. Such a partition determines $n$ profiles $\underline{q}_i=(k_i,(t_j)_{j\in\varphi^{-1}(i)}) \in \Ncirc{2}[s_i]$ for $i=1,\dots,n$ where necessarily $s_i$ is the sum $\sum_{j\in \varphi^{-1}(i)} t_j$. This precisely determines all possible ways of expressing the family of trees associated to $\underline{p}$ by a ``vertex blowup'' of the form $\underline{p}=\underline{p}'\circ \Cal{Q}$, where $\underline{p}'\in \Ncirc{2}[t]$, $\underline{q}^1\in \Ncirc{1}[n]$ and each $\underline{q}^2_i\in \Ncirc{2}$. The term $\left(\Cal{P}_2\tensordot(\Cal{Q}_1\hat{\circ}\Cal{Q}_2) \right)(\underline{p};t)$ is then obtained by using $\Cal{P}$ to evaluate $\underline{p}'$ and $\Cal{Q}$ to evaluate the profiles from $\underline{Q}$. 

Similarly, \begin{align*}  
(\Cal{P}_2\tensordot (\Cal{Q}_2 \circhat \Cal{Q}_1))(\underline{p};t) &
=\Cal{P}_2 \left(k, (t_j)_{j\in\mathbf{k}};t\right)\otimes \left( \mathcal{Q}_2(n, (k_i)_{i\in \mathbf{n}}; k) \otimes \bigotimes_{j\in\mathbf{k}} \Cal{Q}_1(t_j;t_j) \right), 
 \\
(\Cal{P}_1\tensordot \Cal{Q}_3)(\underline{p};t) & = \Cal{P}_1(t;t)\otimes \Cal{Q}_3(\underline{p};t),
 \\
(\Cal{P}_3\tensordot(\Cal{Q}_1\circhat\Cal{Q}_1\circhat\Cal{Q}_1))(\underline{p};t) & 
=\Cal{P}_3(\underline{p};t)\otimes \left( \Cal{Q}_1(n;n) \otimes \bigotimes_{i\in\mathbf{n}} \Cal{Q}_1(k_i;k_i) \otimes \bigotimes_{j\in\mathbf{k}} \Cal{Q}_1(t_j;t_j)\right). 
\end{align*}
\end{example}

\begin{proposition} \label{N_l-monoidal}
The category of $\Nl$-objects equipped with the composition product $\circdot$ is monoidal.
\end{proposition}

\begin{proof}
It is straightforward to verify that $\circdot$ has a two-sided unit, $\mathcal{I}$, given by $\mathcal{I}_1(n;n)=\mathbf{1}$ and $\Cal{I}=\ptC$ otherwise. For $\Nl$-objects $\Cal{P},\Cal{Q},\Cal{R}$, there is a natural isomorphism $(\Cal{P}\circdot \Cal{Q})\circdot \Cal{R}\cong \Cal{P}\circdot (\Cal{Q}\circdot \Cal{R})$ induced by the natural isomorphisms \begin{align} \label{asoc-circdot}
\Big( \Cal{P}_n &\tensordot \big(\Cal{Q}_{k_1}\circhat \cdots \circhat \Cal{Q}_{k_n}\big)\Big)\tensordot \big(\Cal{R}_{\ell_{1,1}}\circhat \cdots \circhat \Cal{R}_{\ell_{n,k_n}}\big) 
\\
&\cong \Cal{P}_n \tensordot \Big( \big(\Cal{Q}_{k_1} \tensordot (\Cal{R}_{\ell_{1,1}}\circhat \cdots \circhat \Cal{R}_{\ell_{1,k_1}}) \big) \circhat \cdots \circhat \big(\Cal{Q}_{k_n} \tensordot (\Cal{R}_{\ell_{n,1}}\circhat \cdots \circhat \Cal{R}_{\ell_{n,k_n}}) \big)\Big) \nonumber
\end{align}
obtained by a tedious but ultimately straightforward calculation. The remainder of the monoidal category axioms follow from similar observations. 
\end{proof}

\begin{definition} 
A \textit{nonsymmetric $\Nl$-operad} is a reduced $\Nl$-object $\Cal{P}$ which is a monoid with respect to $\circdot$. That is, there are unital and associative maps of $\Nl$-objects $\xi\colon \Cal{P}\circdot \Cal{P}\to\Cal{P}$ and $\varepsilon\colon \Cal{I}\to \Cal{P}$, i.e., such that the following diagrams commute \[ \xymatrix{ \Cal{P}\circdot \Cal{P} \circdot \Cal{P} \ar[r]^-{\xi\circdot\id} \ar[d]^-{\id\circdot \xi}& \Cal{P}\circdot \Cal{P} \ar[d]^-{\xi} \\ \Cal{P}\circdot \Cal{P} \ar[r]^-{\xi} & \Cal{P}} \hspace{1cm} \xymatrix{ \Cal{P} \circdot \Cal{I} \ar[r]^-{\id\circdot \varepsilon} & \Cal{P}\circdot \Cal{P} \ar[d]^-{\xi} & \Cal{I} \circdot \Cal{P} \ar[l]_-{\varepsilon \circdot \id} \\ & \Cal{P} \ar[ul]^-{\cong} \ar[ur]_{\cong} }\]
\end{definition}

\subsection{Algebras over a nonsymmetric $\Nl$-operad}
Let $(\widehat{-})$ denote the inclusion of $\mathbf{N}$-colored objects to $\Nl$-objects given by \[\text{$\widehat{X}_0(\emptyset;n)=X[n]$ and $\widehat{X}_k=\ptC$ for $ k\geq 1$.}\] 

Note that $\widehat{X}$ is not reduced, but a straightforward modification of Definition \ref{def:tensor} provides that $\big(\widehat{X}^{\tensorhat n}\big)_0\cong X^{\circhat n}$ and $\big(\widehat{X}^{\tensorhat n}\big)_k\cong \ptC$ for $k\geq 1$. Similarly, $(\widehat{-})$ is left adjoint to $\mathrm{Ev}_0$ which takes values in nonsymmetric sequences and is defined at an $\Nl$-object $\Cal{P}$ as \[(\mathrm{Ev}_0\Cal{P})[n]:=\Cal{P}_0(\emptyset;n).\] 

If $\Cal{P}$ is a nonsymmetric $\Nl$-operad then $\Cal{P}\circdot \widehat{X}$ remains concentrated at level $0$ and hence defines a monad on $\mathbf{N}$-colored objects \[\Cal{P}\circdot(-)\colon X\mapsto \mathrm{Ev}_0(\Cal{P}\circdot \widehat{X}).\] 

\begin{definition}  \label{def-nonsym}
We say that an $\mathbf{N}$-colored object $X$ is an \textit{algebra} over an nonsymmetric $\Nl$-operad $\mathcal{P}$ if there is an action map \[\mathcal{P}\circdot (X)\xrightarrow{\;\; \mu\;\;} X\] which is associative and unital in that the following diagrams commute. \[ \xymatrix{ \Cal{P}\circdot \Cal{P} \circdot (X) \ar[r]^<(.3){\xi\circdot\id} \ar[d]^{\id\circdot \mu}& \Cal{P}\circdot (X) \ar[d]^{\mu} \\ \Cal{P}\circdot (X)\ar[r]^{\mu} & X } \hspace{1cm} \xymatrix{ \Cal{P}\circdot (X) \ar[r]^{\mu} & X \\ \Cal{I}\circdot (X)  \ar[u]^{\varepsilon} \ar[ur]_{\cong} }\]
\end{definition} 

We denote by $\AlgN{P}$ the category of algebras over a nonsymmetric $\Nl$-operad $\Cal{P}$ along with $\Cal{P}$-action preserving maps. Note that an action map $\mu$ consists of pieces \[\mu_k\colon \mathcal{P}_{k} \tensordot (X^{\circhat k})\to X\] for $k\geq 0$ and that $\AlgN{P}$ is complete and cocomplete and moreover that limits are built in the underlying category of $\mathbf{N}$-colored objects. 

\subsection{Change of $\Nl$-operads adjunction}
Given a map of nonsymmetric $\Nl$-operads $\sigma\colon \Cal{P}\to \Cal{Q}$ and a $\Cal{P}$-algebra $X$ we define $\Cal{Q} \circdot_{\Cal{P}}(X)$ by the reflexive coequalizer \begin{equation*}
\label{circ-dot-over}
    \Cal{Q}\circdot_{\Cal{P}}(X):= \colim \left( \xymatrix{
    \Cal{Q}\circdot  (X) & \ar@<-1ex>[l] \ar@<.5ex>[l]  \Cal{Q}\circdot\Cal{P}\circdot (X) }\right).
\end{equation*}
The top map above is given by $\Cal{P}\circdot (X)\xrightarrow{\;\mu_{\Cal{P}}\;} X$ and the bottom is induced by the composite \[\Cal{Q}\circdot \Cal{P}\xrightarrow{\,\id \circdot \sigma\,} \Cal{Q}\circdot \Cal{Q}\xrightarrow{\;\xi_{\Cal{Q}}\;} \Cal{Q}.\] 

The resulting object $\Cal{Q}\circdot_{\Cal{P}}(X)$ inherits a natural $\Cal{Q}$ algebra structure and the construction fits into an adjunction as in the following proposition.

\begin{proposition} 
Given a map of nonsymmetric $\Nl$-operads $\mathcal{P}\xrightarrow{\; \sigma \;}\mathcal{Q}$ there is a change of nonsymmetric $\Nl$-operads adjunction \begin{equation*} \xymatrix{
    \AlgN{P} \ar@<.75 ex>[r] ^{\Cal{Q}\circdot_{\Cal{P}}(-)} &
    \AlgN{Q} \ar@<.75ex>[l]^{\sigma^*} }
\end{equation*}
with right adjoint $\sigma^*$ given by restriction along $\sigma$. 
\end{proposition}

\subsection{A forgetful functor to $\mathbf{N}$-colored operads} \label{N_l-is-N}  We describe forgetful functor $\mathbb{U}$ from $\Nl$-operads to $\mathbf{N}$-colored operads (specifically, nonsymmetric $\mathbf{N}$-colored operads). Given $\underline{p}=((n^1,\cdots, (n^{\ell}_i)_{i\in \mathbf{n^{\ell-1}}} )\in\Ncirc{\ell}$, we set $s(\underline{p})$ to be the unordered list of the elements of the levels of $\underline{p}$, i.e., \[s(\underline{p}):=\left\{n^j_i: j\in\{1,\cdots,n\}, i\in \mathbf{n^j} \right\}. \] 

Given an $\Nl$-object $\mathcal{Q}$ we define $\mathbb{U}\Cal{Q}$ by \begin{equation} 
    (\mathbb{U}\mathcal{Q})(c_1,\dots,c_k;t):=\coprod_{s(\underline{p})=(c_1,\dots,c_k)} \mathcal{Q}_{\ell}(\underline{p};t)
\end{equation}
where the coproduct ranges over $\underline{p}\in \coprod_{\ell \geq 0} \Ncirc{\ell}$. We leave the proof of the following proposition to the reader.

\begin{proposition}
If $\Cal{P}$ is an $\Nl$-operad then $\mathbb{U}\Cal{P}$ is a (nonsymmetric) $\mathbf{N}$-colored operad. Furthermore, the categories $\AlgN{P}$ and $\mathsf{Alg}_{\mathbb{U}\Cal{P}}$ are equivalent. 
\end{proposition}

\subsection{Symmetric $\Nl$-objects} We now impart symmetric group actions on our $\Nl$-objects in a way that captures operadic composition. Denote by $\Cal{I}^{\Sigma}$ the $\Nl$-object in $(\mathsf{C},\otimes, \mathbf{1})$ with \[\Cal{I}^{\Sigma}_{\ell}(\underline{p};t)=\begin{cases} 
    \mathbf{\Sigma}[n] & \ell=1, \underline{p}=n=t \\
    \ptC & \text{otherwise} 
    \end{cases}\] 
    
Recall here that $\mathbf{\Sigma}[n]=\coprod_{\sigma\in\Sigma_n} \mathbf{1}$. Note that $\Cal{I}^{\Sigma}$ is a nonsymmetric $\Nl$-operad whose composition maps are induced by the block matrix inclusions \[ \Sigma_n \times (\Sigma_{k_1}\times\cdots\times \Sigma_{k_n}) \to \Sigma_{k_1+\cdots+k_n}. \]  Moreover the data of an algebra over $\Isym$ is precisely that of a symmetric sequence; i.e.,  $\mathsf{Alg}^{\omega}_{\Cal{I}^{\Sigma}}\cong \SymSeq$. 

\begin{definition}
An $\Nl$-object $\Cal{P}$ \textit{symmetric} if $\Cal{P}$ has compatible right and left actions of $\Cal{I}^{\Sigma}$ in that the following diagram must commute \begin{equation*} 
    \xymatrix{
    \Isym\circdot \Cal{P} \circdot \Isym \ar[rr]^{\mu_{\ell}\circdot\id} \ar[d]^{\id\circdot \mu_{r}} && \Cal{P}\circdot \Isym \ar[d]^{\mu_r}
    \\
    \Isym\circdot \Cal{P} \ar[rr]^{\mu_{\ell}} && \Cal{P} }
\end{equation*}
where $\mu_{\ell}$ (resp. $\mu_r$) denotes the left (resp. right) action map of $\Isym$ on $\Cal{P}$. 
\end{definition}

In other words, a symmetric $\Nl$-object is an $(\Isym,\Isym)$-bimodule. Note that $\Isym \circdot(X)\cong \Sigma{\cdot} X$ is the free symmetric sequence on $X$ (see also Remark \ref{def:Sigma-copowered}).

\subsection{Symmetric $\Nl$-operads}

\begin{definition} Let $\Cal{P},\Cal{Q}$ be $(\Isym,\Isym)$-bimodules. We define their \textit{symmetric composition product}, denoted $\Cal{P}\circdot_{\Sigma} \Cal{Q}$, as the (reflexive) coequalizer (calculated in symmetric $\Nl$-objects)\[ 
    \Cal{P}\circdot_{\Sigma} \Cal{Q} :=\Cal{P}\circdot_{\Cal{I}^{\Sigma}} \Cal{Q}\cong \colim \left( \xymatrix{
        \Cal{P}\circdot \Cal{Q} &
        \Cal{P}\circdot \Cal{I}^{\Sigma} \circdot \Cal{Q} \ar@<.5ex>[l] \ar@<-.75ex>[l] }
    \right) \]
where the two maps are induced by the left and right actions actions of $\Cal{I}^{\Sigma}$ on $\Cal{Q}$ and $\Cal{P}$.
\end{definition}

Note that $\Cal{P}\circdot_{\Sigma} \Cal{Q}$ inherits left and right $\Isym$ actions by those on $\Cal{P}$ and $\Cal{Q}$ respectively, and so remains an $(\Isym,\Isym)$-bimodule. Moreover, $\Cal{I}^{\Sigma}$ is a two-sided unit for $\circdot_{\Sigma}$ and symmetric $\Nl$-objects equipped with the product $(\circdot_{\Sigma}, \Isym)$ is a monoidal category. 
\begin{remark}
Since $\Isym$ is concentrated at level $1$, is it possible to further describe the object $\Cal{P}\circdot_{\Sigma} \Cal{Q}$ in terms of its constituent parts. In particular, \begin{equation*}
    (\Cal{P}\circdot_{\Sigma} \Cal{Q})_{\ell} \cong \coprod_{k\geq 0} \coprod_{\ell_1+\dots+\ell_k=\ell} \Cal{P}_{k} \tensordot_{\Sigma} ( \Cal{Q}_{\ell_1}\circhat \cdots \circhat \Cal{Q}_{\ell_k})
\end{equation*}
where $\Cal{P}_{k} \tensordot_{\Sigma} ( \Cal{Q}_{\ell_1}\circhat \cdots \circhat \Cal{Q}_{\ell_k})$ is obtained as the coequalizer \begin{equation*}
    \colim \left( \xymatrix{
        \Cal{P}_k \tensordot (\Cal{Q}_{\ell_1}\circhat \cdots \circhat \Cal{Q}_{\ell_k}) &
        \Big( \Cal{P}_k \tensordot (\underbrace{\Isym_1 \circhat \cdots \circhat \Isym_1}_k) \Big) \tensordot  (\Cal{Q}_{\ell_1}\circhat \cdots \circhat \Cal{Q}_{\ell_k}) \ar@<-1ex>[l] \ar@<.5ex>[l]
    } \right) 
\end{equation*}
such that the top is induced by the right action of $\Isym$ on $\Cal{P}$ and the bottom map is induced by the isomorphism (\ref{asoc-circdot}) and the left action of $\Isym$ on $\Cal{Q}$.
\end{remark}

\begin{definition} \label{def:sym-Nl}
    A \textit{symmetric $\Nl$-operad} is a reduced symmetric $\Nl$-object $\Cal{P}$, which is a monoid with respect to $\circdot_{\Sigma}$. That is, there is a multiplication map $\xi\colon \Cal{P}\circdot_{\Sigma}\Cal{P}\to \Cal{P}$ and unit map $\varepsilon \colon \Cal{I}^{\Sigma} \to \Cal{P}$ that satisfy the usual associativity and unitality conditions. 
\end{definition}

\subsection{Algebras over symmetric $\Nl$-operads} \label{sec:algebras}
We now define an algebra over a symmetric $\Nl$-operad $\Cal{P}$. Note than algebra over a symmetric $\Nl$-operad is a symmetric $\Nl$-object concentrated at level $0$, that is, an $\Isym$-algebra or symmetric sequence. As before, given a symmetric $\Nl$-operad $\Cal{P}$, let  \[\Cal{P}\circdot_{\Sigma}(-)\colon X\mapsto \mathrm{Ev}_0(\Cal{P} \circdot_{\Sigma}  \widehat{X})\] be the associated monad on $\SymSeq$.

\begin{definition} 
    A symmetric sequence $X$ is an \textit{algebra} over a symmetric $\Nl$-operad $\Cal{P}$ if there is an action map $\mu\colon \Cal{P}\circdot_{\Sigma} (X) \to X$ which is associative and unital (as in Definition \ref{def-nonsym} with $\circdot$ replaced by $\circdot_{\Sigma}$ q.v.)
\end{definition}

We denote by $\AlgSym{P}$ the category of symmetric $\Cal{P}$-algebras with $\Cal{P}$-algebra preserving maps; for simplicity we will frequently use $\Alg{P}$ instead when there is no room for confusion. We note that $\mu$ consists of maps \[\mu_k \colon \Cal{P}_k \tensordot_{\Sigma} (X^{\circhat k})\to X\] where the action of $\Isym$ on $X^{\circhat k}$ agrees with that for symmetric sequences discussed in Section \ref{sec:comp-prod-rmks}. Furthermore, \[\mu_0\colon I\cong \Cal{P}_0\tensordot_{\Sigma}  (X^{\circhat 0}) \to X\] gives a unit map for $X\in \Alg{P}$ and we note that an algebra $X$ over $\Cal{P}$ will always be reduced, i.e., $X[0]=\ptC$.

\begin{example}[Free symmetric $\Cal{P}$-algebra on a symmetric sequence] 
    Given a symmetric sequence $X$, the object $\Cal{P}\circdot_{\Sigma}(X)$ is the free $\Cal{P}$-algebra on $X$ and fits into an adjunction \[ \xymatrix{ 
        \SymSeq_{\mathsf{C}} \ar@<.75ex>[r]^-{\Cal{P}\circdot_{\Sigma}(-)} 
        & 
        \Alg{P} \ar@<.5ex>[l]^-{\Cal{U}} 
    }\] where $\Cal{U}$ is the forgetful functor. In particular, $\Oper \circdot_{\Sigma}(X)$ (see Definition \ref{def-of-Oper}) is the \textit{free operad} on $X$ (see, e.g.,  \cite[9.4]{AC-1}).
\end{example}

We leave the proof of the following to the reader as it follows from standard arguments as in \cite[3.29]{Har-1} or \cite[4.3]{Bor}.

\begin{proposition} If $(\mathsf{C},\otimes,\mathbf{1})$ is closed symmetric monoidal which contains all small limits and colimits, then all small limits and colimits exist in $\Alg{P}$. Limits and filtered colimits are built in the underlying category of symmetric sequences and are further reflected by the forgetful functor $\Cal{U}$.

General colimits shaped on a small diagram $\Cal{D}$ are constructed by the following (reflexive) coequalizer (whose colimits are constructed in $\SymSeq$): \[ \colim_{d\in\Cal{D}} X_d\cong \colim\left( \xymatrix{
     \Cal{P}\circdot_{\Sigma} \left( \colim_{d\in\Cal{D}} X_d \right) &
     \Cal{P}\circdot_{\Sigma} \left( \colim_{d\in\Cal{D}} \Cal{P}\circdot_{\Sigma}(X_d)\right) \ar@<.5ex>[l] \ar@<-.5ex>[l] 
    } \right).\]
\end{proposition}

\subsection{Modules over $\Cal{P}$-algebras}

\begin{definition} \label{A-infty-alg} 
     Let $\Cal{P}$ be a symmetric $\Nl$-operad and $\Cal{W}$ be a $\Cal{P}$-algebra. Let $M$ be a symmetric sequence. We say that $M$ is an \textit{$\Cal{W}$-module} if there are maps of the form \[\eta_{\ell}\colon \Cal{P}_{\ell} \tensordot_{\Sigma}  \left( \Cal{W}^{\circhat (\ell-1)} \circhat M\right)\to M\] for $\ell\geq 1$ that satisfy associativity (\ref{assoc-A-infty-alg}) and unitality (\ref{unit-A-infty-alg}). If $M$ is concentrated at level $0$ we say that the object $M[0]$ is a \textit{$\Cal{W}$-algebra}. 
\end{definition}

    Set $\xi\colon \Cal{P}\circdot_{\Sigma} \Cal{P}\to \Cal{P}$ to be the multiplication on $\Cal{P}$ and $\mu\colon \Cal{P}\circdot_{\Sigma} (\Cal{W})\to \Cal{W}$ the action map on $\Cal{W}$. Let $\ell:=\ell_1+\dots+\ell_k$. Associativity and unitality amounts to the commutitivity of the following diagrams \begin{equation} \label{assoc-A-infty-alg}
        \xymatrix{
            \left(\Cal{P}_{k}\tensordot_{\Sigma}  (\Cal{P}_{\ell_1}\circhat \cdots \circhat \Cal{P}_{\ell_k})\right) \tensordot_{\Sigma}  \left( \Cal{W}^{\circhat (\ell-1)} \circhat M\right) \ar[r]^<(.2){\xi_{\ell}\otimes_{\Sigma}\id} \ar[d]^{\cong}
            &
            \Cal{P}_{\ell} \tensordot_{\Sigma}  \left( \Cal{W}^{\circhat (\ell-1)} \circhat M\right) \ar[dd]^{\eta_{\ell}}
            \\
            \Cal{P}_{k}\tensordot_{\Sigma}  \left( (\Cal{P}_{\ell_1}\tensordot_{\Sigma}  \Cal{W}^{\circhat \ell_1} ) \circhat \cdots\circhat (\Cal{P}_{\ell_k} \tensordot_{\Sigma}  ( \Cal{W}^{\circhat \ell_k-1}\circhat M ) ) \right) 
            \ar[d]^<(.3){\id\tensordot_{\Sigma} (\mu_{\ell_1} \circhat \cdots \circhat \mu_{\ell_{k-1}}\circhat \eta_{\ell_k})}
            &
            \\
            \Cal{P}_k \tensordot_{\Sigma}  \left( \underbrace{\Cal{W}\circhat \cdots \circhat \Cal{W}}_{k-1} \circhat M \right) \ar[r]^{\eta_k}
            &
        M,
        }
    \end{equation} and
    \begin{equation} \label{unit-A-infty-alg}
        \xymatrix{
            \Cal{P}_2 \tensordot_{\Sigma}  \left( (\Cal{P}_0\tensordot_{\Sigma}  \Cal{W}^{\circhat 0}) \circhat (\Cal{P}_1\tensordot_{\Sigma}  M) \right) \ar[rr]^<(.3){\id\tensordot_{\Sigma} (\mu_0 \circhat \eta_1)} \ar[d]^{\cong}
            &&
            \Cal{P}_2\tensordot_{\Sigma}  (\Cal{W} \circhat M) \ar[dd]^{\eta_2}
            \\
            \left( \Cal{P}_2 \tensordot_{\Sigma}  (\Cal{P}_0\circhat \Cal{P}_1) \right) \tensordot_{\Sigma}  M \ar[d]^{\xi_1\tensordot_{\Sigma} \id}
            &&
            \\
            \Cal{P}_1 \tensordot_{\Sigma}  M \ar[rr]^{\eta_1}
            &&
            M.}
    \end{equation} Recall that $\mu_0\colon I\cong \Cal{P}_0\tensordot_{\Sigma}  \Cal{W}^{\circhat 0}\to \Cal{W}$ is the unit map for $\Cal{W}$.

\begin{remark}
    We encourage the reader to compare the above definition with that of modules over algebras over an operad, e.g.,  as in May \cite[Definition 3]{May}. In \cite[1.5.1]{BM-trees} an example of a $2$-colored operad whose algebras are pairs $(A,M)$ of an $\Cal{O}$-algebra $A$ along with an $A$-module $M$ is provided. The pair $(\Cal{W}, M)$ can be described analogously as an algebra over an $\mathbf{N}_+:=\{\pt, 0,1,2,\dots\}$-colored operad with levels, though we will not require such description.  
\end{remark}

\begin{definition} \label{def:equiv-of-Nl-ops} We say a map $\Cal{P}\to\Cal{Q}$ of (symmetric) $\Nl$-operads in some symmetric monoidal model category $\mathsf{C}$ is an \textit{equivalence} if for any $\underline{p}\in\Ncirc{k}[t]$ the induced map $\Cal{P}_k(\underline{p};t)\to\Cal{Q}_k(\underline{p};t)$ is a weak equivalence in $\mathsf{C}$. We write $\Cal{P}\simeq \Cal{Q}$ if there is a zig-zag of equivalences of (symmetric) $\Nl$-operads connecting $\Cal{P}$ and $\Cal{Q}$. 

In the special case that $\Cal{P}\simeq \Oper$ then we say that a $\Cal{P}$-algebra $\Cal{W}$ is an \textit{$A_{\infty}$-operad} and that modules over $\Cal{W}$ are \textit{$A_{\infty}$-algebras}.
\end{definition}

\section{Examples of symmetric $\Nl$-operads}
In this section we describe some examples of symmetric $\Nl$-operads of interest, specifically the coendomorphism $\Nl$-operads on a given cosimplicial symmetric sequence. We begin by describing $\Oper$ -- the symmetric $\Nl$-operad whose algebras are (one-color) operads as some of its properties will be essential in what is to come. Our eventual goal is to prove that the coendomorphism $\Nl$-operad on a $\Sigma$-free symmetric sequence $\Cal{X}$ (see Remark \ref{def:Sigma-copowered}) is indeed a symmetric $\Nl$-operad; with the particular example of $\Cal{A}= \mathsf{coEnd}(\SD_+)$ in mind (see Section \ref{sec:A}). 

Though we write most of this section for a general closed cocomplete symmetric monoidal category $\mathsf{C}$, we invite the reader to think particularly of the cases when $\mathsf{C}=\Spt$ or $\Top$.

\subsection{The symmetric $\Nl$-operad $\Oper$} \label{sec:oper-def}
We begin by describing $\Oper$ for the category $\mathsf{Set}$ of sets.

\begin{definition} \label{def-of-Oper} 
    Let $\Sigma$ denote the symmetric sequence in $(\mathsf{Set}, \times, \pt)$ with $\Sigma[n]=\Sigma_n$ and define a reduced $\Nl$-object as follows. For $\underline{p}\in \Ncirc{k}[t]$ we set \[ \Oper_{\ell}( \underline{p};t) := \hom\left( \Sigma[t] , \Sigma^{\boxcirc \ell}[\underline{p}] \right)^{\Sigma_t}\]
\end{definition}

\begin{remark}
    Note there are isomorphisms \begin{equation} \label{eq:oper-iso} \Oper_{\ell} (\underline{p};t)=\hom \left(\Sigma[t], \Sigma^{\circ \ell}[\underline{p}] \right)^{\Sigma_t} \cong \hom \left(\pt, \Sigma^{\circ \ell} [\underline{p}] \right) \cong \Sigma^{\circ \ell} [\underline{p}].\end{equation}
    Computing some small examples of $\Oper$, we note that 
    \begin{align*} 
    &\Oper_0(\emptyset;1) \cong * && \Oper_1(\emptyset;n) \cong \emptyset \quad (n \neq 1)
    \\
    &\Oper_1(n;n) \cong \Sigma_n \quad (n \geq 0) &&  \Oper_1(n;m)\cong \emptyset \quad (n\neq m \geq 0)
    \\
    &\Oper_2\left(n, (k_1,\dots,k_n);k\right)  \cong \Sigma_n \times_{\Sigma_{p_1}\times \cdots \times \Sigma_{p_n}} \Sigma_{k}  
\end{align*}
where $p_1,\dots,p_m$ denotes the multiplicities of distinct integers among $k_1,\dots,k_n$, $k=\sum_{i=1}^n k_i$, and $\Sigma_{p_1}\times \cdots \times \Sigma_{p_m}$ acts on $\Sigma_k$, e.g.,  by permutation of block matrices \[\Sigma_{k_1} \times \cdots \times \Sigma_{k_n} \leq \Sigma_k.\] Similarly, let $q_1,\dots,q_r$ denotes the multiplicities of the distinct integers among $t_1,\dots,t_k$ and set \[\underline{p} = (n, (k_1,\dots,k_n), (t_1,\dots,t_k)) \in \Ncirc{3}[t].\]  Then \begin{align*}
    \Oper_3(\underline{p};t) & \cong  \Sigma_n \times_{\Sigma_{p_1}\times \cdots \times \Sigma_{p_m}} \Sigma_k \times_{\Sigma_{q_1}\times \cdots \times \Sigma_{q_r}} \Sigma_t 
    \end{align*} 
\end{remark}

\begin{proposition} \label{prop:Oper-is-Nl}
    $\Oper$ is a symmetric $\Nl$-operad.
\end{proposition} 

\begin{proof} As we will see, $\Oper$ is particularly special as the structure maps \begin{equation} \label{eq:xi} 
    \xi_{k,(\ell_1,\dots,\ell_k)}\colon \Oper_{k} \tensordot_{\Sigma} (\Oper_{\ell_1} \circhat \cdots \circhat \Oper_{\ell_k}) \to \Oper_{\ell}
\end{equation} which comprise $\xi\colon \Oper \circdot_{\Sigma} \Oper \to \Oper$ consist of isomorphisms once evaluated at a profile $\underline{p}\in \Ncirc{\ell}[t]$. 

That $\Oper$ is symmetric follows from the first part of the proof of Proposition \ref{prop:coend-X}.  The unit map $\epsilon \colon \Isym \to \Oper$ is obtained via the identity morphisms \[ \Isym_1(n;n)\cong \Sigma_n \to \Sigma_n \cong \Oper_1(n;n) \] and the initial morphism elsewhere. Let us now produce the desired map (\ref{eq:xi}) at a profile $\underline{p}\in \Ncirc{\ell}[t]$. 

For the reader who finds the following constructions a bit opaque, we first provide the following intuition: for $\ell\geq 0$ let $\circ_{\ell} \colon \SymSeq^{\times \ell}\to \SymSeq$ be the functor $\circ_{\ell} (X_1,\dots,X_{\ell}) =X_1 \circ \cdots \circ X_{\ell}$. Since $\circ$ is strictly monoidal, there are \textit{isomorphisms} \begin{equation} \label{eq:Xi-fo} \Xi_{k,(\ell_1,\cdots,\ell_k)}\colon \circ_k (\circ_{\ell_1},\cdots, \circ_{\ell_k}) \xrightarrow{\cong} \circ_{\ell_1+\cdots+\ell_k}\end{equation} such that $\circ_{\bullet}$ is a \textit{nonsymmetric functor-operad} (see, e.g.,  McClure-Smith \cite[\textsection 4]{MS-box}, omitting the requirement of symmetric group actions). 
Moreover, the composition maps $\xi_{k,(\ell_1,\dots,\ell_k)}$ are precisely the morphisms which prescribe the equivariance of the isomorphism $\Xi_{k,(\ell_1,\dots,\ell_k)}$ once evaluated at a particular string of inputs, given that evaluation at a profile in $\Ncirc{\ell}$ is the same as evaluating a symmetric sequence \textit{from the left}. For instance, $\xi_{3,(2,1,3)}$ provides the isomorphisms (natural in $X_1,\dots,X_6$) \[ (X_1\circ X_2) \circ X_3 \circ (X_4\circ X_5 \circ X_6) \cong  X_1 \circ \cdots \circ X_6\] 
and moreover, given $\underline{p}\in\Ncirc{\ell}[t]$, the desired map $\xi_{k,(\ell_1,\dots,\ell_k)}[\underline{p}]$ may be thought of a precisely arising from the isomorphism \[ \left(\Sigma^{\circ \ell_1} \circ \cdots \circ \Sigma^{\circ \ell_k}\right)[\underline{p}]\xrightarrow{\cong}\Sigma^{\circ \ell} [\underline{p}].\]

We describe $\xi_{2,(1,2)}$ first and note the general case follows a similar argument. Let $\underline{p}\in \Ncirc{3}[t]$ and note that \[ 
    \Oper_2 \tensordot_{\Sigma}  (\Oper_1 \circhat \Oper_2) )[\underline{p}] \cong \coprod_{\underline{p}=(n,(\underline{p}_1 \amalg \cdots \amalg \underline{p}_n))} \Oper_2 \tensordot_{\Sigma}  (\Oper_1 \circhat \Oper_2) )[(n, (\underline{p}_1,\dots,\underline{p}_n))]. \] 
Fix $\underline{p}_i\in \Ncirc{2}[s_i]$ for $i=1,\dots,n$ such that $\underline{p}=(n, (\underline{p_1} \amalg \cdots \amalg \underline{p_n}))$ and set $\underline{p}'=(n,(s_1,\dots,s_n))\in \Ncirc{2}[t]$. We then observe
\begin{align} \label{al:comp}
    (\Oper_2 \tensordot_{\Sigma}  (\Oper_1  &\circhat \Oper_2  ) )[(n, (\underline{p}_1,\dots,\underline{p}_n))]
    \\
    &\cong \Sigma^{\circ 2}[\underline{p}'] \times_{S(\underline{p}')} \left( \Sigma[n] \times ( \Sigma^{\circ 2} [\underline{p}_1] \times \cdots \times \Sigma^{\circ 2}[\underline{p}_n] \right) \nonumber
    \\
    &\cong  \Sigma^{\circ 3} [(n, (\underline{p}_1,\cdots,\underline{p}_n))] \xrightarrow{\iota} \Sigma^{\circ 3}[\underline{p}] \cong \Oper_3(\underline{p};t)\nonumber
\end{align}
such that $S(\underline{p'})=\Sigma_n \times \prod_{i=1} \Sigma_{s_i}$ and $\iota$ is the natural inclusion obtained from the assumption $\underline{p}=(n, (\underline{p}_1 \amalg \cdots \amalg \underline{p}_n))$.

The desired map $\xi_{2,(1,2)}[\underline{p}]$ is induced by the coproduct of composites (\ref{al:comp}) for all $\underline{p}=(n,(\underline{p}_1,\dots,\underline{p}_n))$. Note further that \textit{as sets} there is an isomorphism \[ 
    \coprod_{\underline{p}=(n,(\underline{p}_1 \amalg \cdots \amalg \underline{p}_n))} \Sigma^{\circ 3} [(n,(\underline{p}_1,\cdots, \underline{p}_n))] \cong \Sigma^{\circ 3}[\underline{p}]
\] 
since $\circ$ is strictly monoidal in the category of symmetric sequences of sets. Thus, $\xi_{2,(1,2)}[\underline{p}]$ is invertible and more generally $\xi_{k,(\ell_1,\dots,\ell_k)}$ evaluated at any profile in $\Ncirc{\ell}$ is also invertible.

Associativity of $\xi$ then follows from the associativity $\Xi$ as in (\ref{eq:Xi-fo}). That is, for $(n,(k_1,\dots,k_n))\in \Ncirc{2}[k]$ and for $i=1,\dots,n$, $\underline{q}_i=(k_i,(\ell_{i,1},\cdots,\ell_{i,k_i})) \in \Ncirc{2}[t_i]$ 
the associativity relation \begin{align*}
     \xi_{k,(\ell_{1,1},\cdots, \ell_{n,k_n})}\,& ( \xi_{n,(k_1,\cdots,k_n)} \tensordot_{\Sigma} \id) \\
     &=\xi_{n,(t_1,\cdots,t_n)} \, \big(\id \tensordot_{\Sigma}  (\xi_{k_1,(\ell_{1,1},\cdots,\ell_{1,k_1})} \circhat \cdots \circhat \xi_{k_n,(\ell_{n,1},\cdots,\ell_{n,k_n})})\big)\end{align*}
evaluated at some $\underline{p}\in \Ncirc{\ell}$ follows from the commutative square of isomorphisms
\[ \xymatrix{
    \left(\left(\Sigma^{\circ \ell_{1,1}} {\circ} \cdots {\circ} \Sigma^{\circ \ell_{1,k_1}}\right){\circ} \cdots {\circ} \left(\Sigma^{\circ \ell_{n,1}} {\circ} \cdots {\circ} \Sigma^{\circ \ell_{n,k_n}}\right)\right) [\underline{p}] \ar[r] \ar[d] 
    & 
    \left(\Sigma^{\circ t_1} {\circ} \cdots {\circ} \Sigma^{\circ t_n}\right) [\underline{p}] \ar[d]
    \\
    \left( \Sigma^{\circ \ell_{1,1}} {\circ}  \cdots {\circ} \Sigma^{\circ \ell_{n,k_n}}\right)[\underline{p}] \ar[r]
    &
    \Sigma^{\circ \ell}[\underline{p}]
}\]

Similarly, the unitality condition is satisfied by the more obvious isomorphisms \[ \left( ( \underbrace{\Sigma \circ \cdots \circ \Sigma}_{n}) \right)[\underline{p}] \cong \Sigma^{\circ n}[\underline{p}] \cong \left( \underbrace{(\Sigma) \circ \cdots \circ (\Sigma)}_{n} \right)[\underline{p}]\] for all $n\geq 0$ and $\underline{p}\in \Ncirc{1}$ (i.e., $\underline{p}=p\geq 0$). 
\end{proof} 

\begin{remark}
    Let $(\mathsf{C},\otimes, \mathbf{1})$ be a closed symmetric monoidal category with finite coproducts. We write $\Oper^{\mathsf{C}}$ for the image of $\Oper$ in $\mathsf{C}$ under $\Sigma_n\mapsto \mathbf{\Sigma}[n]\cong \coprod_{\sigma\in \Sigma_n} \mathbf{1}$. That is, given a profile $\underline{p}\in \Ncirc{k}[t]$ we set \[ \Oper^{\mathsf{C}}(\underline{p};t)= \Map^{\mathsf{C}}\left( \mathbf{\Sigma}[t], \mathbf{\Sigma}^{\circ k} [\underline{p}] \right)^{\Sigma_t}.\]
\end{remark}

Before showing that $\Oper$ encodes (one-color) operads as its algebras we first demonstrate another class of symmetric $\Nl$-operads. 

\subsection{Coendomorphism symmetric $\Nl$-operads}
Recall as in Section 6 that $(\mathsf{C}, \otimes, \mathbf{1})$ denotes a closed cocomplete symmetric monoidal category and $\mathbf{\Sigma}$ is the symmetric sequence in $\mathsf{C}$ with $\mathbf{\Sigma}[k]=\coprod_{\sigma\in \Sigma_k} \mathbf{1}.$
\begin{definition} \label{def:coend}
    Let $\Cal{X}\in \SymSeq_{\mathsf{C}}^{\bdel}$ and set $\mathsf{coEnd}(\Cal{X})$ to be the reduced  $\Nl$-object given at $(\underline{p};t)\in \Ncirc{\ell}[t]$ by \[ \mathsf{coEnd}(\Cal{X})_{\ell}(\underline{p};t):= \Map_{\Dres} \left( \Cal{X}[t] , \Cal{X}^{\boxcirc \ell} [\underline{p}] \right)^{\Sigma_t}.\]
\end{definition}

\begin{example} Unravelling the above definition, $\mathsf{coEnd}(\Cal{X})_1(k;k)$ consists of all $\Sigma_k$-equivariant cosimplicial maps $\Cal{X}[k] \to \Cal{X}[k]$. Let $(\underline{q};k)=(n,(k_1,\dots,k_n);k)\in \Ncirc{2}[k]$ and recall the description of $H(k_1,\dots,k_n)\leq \Sigma_k$ from Definition \ref{def:H}. Then, $\mathsf{coEnd}(\Cal{X})_2(\underline{q};k)$ consists of all $\Sigma_k$-equivariant cosimplicial maps of the form \[ 
    \Cal{X}[k] \to (\Cal{X}\boxcirc \Cal{X}) [\underline{q}]\cong \mathbf{\Sigma}[k] \otimes_{H(k_1,\dots,k_n)} \Cal{X}[n] \square (\Cal{X}[k_1] \otimes\cdots \otimes \Cal{X}[k_n]).
\] 

Further, $\mathsf{coEnd}(\Cal{X})$ is \textit{quadratic} in that it is generated by its first two levels as follows: let $\underline{p}=(n, (k_i)_{i\in \mathbf{n}}, (t_j)_{j\in \mathbf{k}}$) and set $k:=\sum_{i=1}^n k_i$ and $t:=\sum_{j=1}^k t_j$. Then, $\mathsf{coEnd}(\Cal{X})_3(\underline{p};t)$ consists of cosimplicial maps $\psi$ that fit into the following $\Sigma_t$-equivariant diagram \begin{equation*}
\xymatrix{
    \Cal{X} [t] \ar[r]^{\psi_1} \ar[dr]_{\psi} & 
    (\Cal{X}\boxcirc\Cal{X})[k, (t_j)_{j\in\mathbf{k}}] \ar[d]^{\psi_2\boxcirc \textrm{id}} 
    \\ 
    & (\Cal{X}\boxcirc\Cal{X}\boxcirc \Cal{X}) [ n, (k_i)_{i\in\mathbf{n}}, (t_j)_{j\in\mathbf{k}}].
    }
\end{equation*}
such that $\psi_2\in \mathsf{coEnd}(\Cal{X})_2(n,(k_1,\dots,k_n);k)$, i.e., $\psi_2\colon \Cal{X}[k]\to (\Cal{X}\boxcirc \Cal{X})[n,(k_1,\dots,k_n)]$ is $\Sigma_k$-equivariant. Said differently, there is an isomorphism \[
    \mathsf{coEnd}(\mathcal{X})_3(\underline{p};t)\cong  \mathsf{coEnd}(\mathcal{X})_2(n, (k_i)_{i\in\mathbf{n}};k) \otimes_{\Sigma_k} \mathsf{coEnd}(\mathcal{X})_2(k, (t_j)_{j\in\mathbf{k}}; t)
\] where $\Sigma_k$ acts by shuffling the factors $t_1,\dots,t_k$ of $\Cal{X}_2(k, (t_j)_{j\in\mathbf{k}};t)$ in accordance to the $\Sigma_k$ equivariance of maps in $\Cal{X}_2(n, (k_i)_{i\in\mathbf{n}};k)$. In general, given a profile \[\underline{p}=(n^1,(n^2_i)_{i\in\mathbf{n^1}}, \dots, (n^{\ell}_i)_{i\in \mathbf{n^{\ell-1}}}) \in\Ncirc{\ell}\] the object $\mathsf{coEnd}(\mathcal{X})_{\ell} (\underline{p};n^{\ell})$ is isomorphic to  \begin{align} \label{eq:coEnd-quad}
     \mathsf{coEnd}(\mathcal{X})_2\big(n^1, (n^2_i)_{i\in \mathbf{n^1}};n^2\big) \otimes_{\Sigma_{n^2}}  \cdots   \otimes_{\Sigma_{n^{\ell-1}}} \mathsf{coEnd}(\mathcal{X})_2\big(n^{\ell-1}, (n^{\ell}_i)_{i\in\mathbf{n^{\ell-1}}};n^{\ell}\big).\end{align} \end{example}

\begin{remark} \label{def:Sigma-copowered}
    We would like to be able to say that $\coend{X}$ is a symmetric $\Nl$-operad for \textit{any} cosimplicial symmetric sequence $\Cal{X}$, however this seems to not be the case. The issue seems to be based on the potential non-invertibility of $\theta$ (as in \eqref{eq:theta}) and similarly how $\boxcirc$ fails to be a \textit{strictly} monoidal product for cosimplicial symmetric sequences. However, there is a class of cosimplicial symmetric sequences on which we get the desired symmetric $\Nl$-structure on $\coend{\Cal{X}}$. 
    
    Let us say that $\Cal{X}$ is \textit{$\Sigma$-free} if there is a sequence $\{\Cal{Y}[n]\}_{n\geq 0}$ of cosimplicial objects in $\mathsf{C}$ with \[ \Cal{X}[n] = \Sigma_n {\cdot} \Cal{Y}[n] \cong \mathbf{\Sigma}[n]\otimes \Cal{Y}[n]\]
    and such that the $\Sigma_n$ action on $\Cal{X}[n]$ is trivial on $\Cal{Y}[n]$ for all $n$. In such case we write $\Cal{X}= \Sigma {\cdot} \Cal{Y}$. The benefit for us is that if $\Cal{X}$ is $\Sigma$-free, then $\theta$ has an inverse (which is constructed in the following proposition), and so $\boxcirc$ \textit{is} a monoidal product when restricted to $\Sigma$-free cosimplicial symmetric sequences.
\end{remark}

\begin{proposition} \label{prop:coend-X}
    If $\Cal{X}\in \SymSeq_{\mathsf{C}}^{\bdel}$ is $\Sigma$-free, then $\mathsf{coEnd}(\Cal{X})$ is a symmetric $\Nl$-operad.
\end{proposition}

\begin{proof}
    This argument is rather long and somewhat tedious, so we break it up into several steps. The first step is to show that $\mathsf{coEnd}(\Cal{X})$ is symmetric, in fact $\Sigma$-freeness is not required for this part. 
    
    Let $\ell\geq 0$. The left action of $\Isym$ on $\coend{X}_{\ell}$ is obtained by $\Sigma_t$ action on the maps $\Cal{X}[t]\to \Cal{X}^{\boxcirc \ell}[\underline{p}]$ which comprise $\coend{X}_{\ell}$. The right action of $\Isym\circhat \cdots \circhat \Isym$ on $\coend{X}_{\ell}$ is obtained, e.g., , at $\ell=2$ as follows. For a profile $\underline{q}=(n,(k_1,\cdots,k_n))$, we observe \[(\Isym\circhat \Isym)(\underline{q};k)\cong \Sigma_n\ltimes (\Sigma_{k_1} \times \cdots \times \Sigma_{k_n}) \leq \Sigma_k\] acts via the $\Sigma_k$-equivariance of \[ 
        \Cal{X}[k] \to \mathbf{\Sigma}[k]\otimes_{H(k_1,\dots,k_n)} \Cal{X}[n] \square (\Cal{X}[k_1] \otimes\cdots \otimes \Cal{X}[k_n]).
    \] 
    The general case follows a similar argument. 
    
    Second, we produce a multiplication map \[\xi\colon \coend{X}\circdot_{\Sigma} \coend{X}\to \coend{X}.\] Two ingredients are crucial to this step. First, is the existence of maps \begin{equation} \mu_{\ell_1,\dots,\ell_k} \colon \Cal{X}^{\boxcirc \ell_1} \boxcirc \cdots \boxcirc \Cal{X}^{\boxcirc \ell_k} \to \Cal{X}^{\boxcirc \ell}\end{equation} 
    for each tuple $\ell_1,\dots,\ell_k$ such that $\ell_1+\cdots + \ell_k=\ell$ which are inverse to the induced map by $\theta$ (see (\ref{eq:theta})). It is this step for which $\Sigma$-freeness of $\Cal{X}$ seems essential and such maps $\mu$ are granted by utilizing the structure of $\Oper$. Write $\Cal{X}=\Sigma {\cdot} \Cal{Y}$ and for $\underline{p}=(n^1, \cdots, (n^k_i)_{i\in \mathbf{n^{k-1}}})$ set \[ \Cal{Y}^{\square k}[\underline{p}]=\Cal{Y}[n^1] \square \left(\bigotimes_{i\in \mathbf{n^1}} \Cal{Y}[n^2_i] \right) \square \cdots \square \left( \bigotimes_{i\in \mathbf{n^{k-1}}} \Cal{Y}[n^k_i] \right).\]
    Note in the above, we are utilizing the box product for $\mathsf{C}^{\bdel}$ which is strictly monoidal. 
    
    For simplicity we describe the map $\mu_{1,2}\colon \Cal{X} \boxcirc (\Cal{X} \boxcirc \Cal{X})\to \Cal{X}^{\boxcirc 3}$ and note the general case follows from a similar argument . Note that $\Cal{X} \boxcirc (\Cal{X} \boxcirc \Cal{X})$ takes as inputs profiles of the form $(n,(\underline{p_1},\cdots, \underline{p_n}))$ for some unordered list of profiles $\underline{p_i}\in \Ncirc{2}$. 
    
    Fix a specific profile $(n,(\underline{p_1}\amalg \cdots \amalg \underline{p_n}))=\underline{p}$ and for $i=1,\dots,n$, write $\underline{p_i}=(k_i, (t_{i,1},\cdots, t_{i,k_i})) \in \Ncirc{2}[t_i]$. There is an inclusion induced as follows \begin{align} 
         \Cal{X} &\boxcirc (\Cal{X} \boxcirc \Cal{X}) ) [n, (\underline{p_1},\cdots,\underline{p_n})] 
         \\
        & \cong (\mathbf{\Sigma}[n] \otimes  \Cal{Y}[n]) \square \left( \left( \mathbf{\Sigma}^{\circ 2}[\underline{p_1}] \otimes \Cal{Y}^{\square 2}[\underline{p_1}]\right)  \otimes \cdots \otimes \left( \mathbf{\Sigma}^{\circ 2}[\underline{p_n}] \otimes \Cal{Y}^{\square 2}[\underline{p_n}]\right) \right) \nonumber
        \\
        & \cong \left(\mathbf{\Sigma}[n] \otimes_{\Sigma_n} \Big(\coprod_{\dagger}\mathbf{\Sigma}[t] \otimes_{\Sigma_{t_1} \times \cdots \times \Sigma_{t_n}} \mathbf{\Sigma}^{\circ 2}[\underline{p_1}] \times \cdots \times \mathbf{\Sigma}^{\circ 2}[\underline{p_n}] \Big) \right) \otimes \Cal{Y}^{\square 3}[\underline{p}] \nonumber 
        \\ 
        & \cong \mathbf{\Sigma}^{\circ 3} [n, (\underline{p_1},\cdots, \underline{p_n})] \otimes \Cal{Y}^{\square 3}[\underline{p}] \xrightarrow{(*)} \mathbf{\Sigma}^{\circ 3}[\underline{p}] \otimes \Cal{Y}^{\square 3}[\underline{p}]  \cong \Cal{X}^{\boxcirc 3} [\underline{p}] \nonumber
    \end{align}
    where $\dagger$ runs over all $\Sigma_n$ permutations of $t_1,\cdots,t_n$ and $(*)$ is induced by the natural inclusion $\iota\colon \Sigma^{\circ 3} [n, (\underline{p_1},\cdots, \underline{p_n})] \to \Sigma^{\circ 3} [\underline{p}]$. Moreover, the map $\mu_{1,2}$ at profile $\underline{p}$ is then induced from the inclusion described above via the isomorphism\[
        \Big(\Cal{X}\boxcirc (\Cal{X} \boxcirc \Cal{X}) \Big)[\underline{p}] \cong \coprod_{(n,(\underline{p_1}\amalg \cdots \amalg \underline{p_n}))=\underline{p}} \Big(\Cal{X}\boxcirc (\Cal{X} \boxcirc \Cal{X}) \Big)\big [n,(\underline{p_1},\dots,\underline{p_n})\big].
    \]
    A straightforward computation then shows that $\mu_{1,2}$ is inverse to $\theta$.
    
    The second ingredient to producing $\xi$ is a map \begin{equation} \label{eq:def-Gamma}
        \coend{X}_{\ell_1} \circhat \cdots \circhat \, \coend{X}_{\ell_k} \xrightarrow{\;\;\Gamma\;\;} \Map_{\Dres} \left(\Cal{X}^{\boxcirc k}, \Cal{X}^{\boxcirc \ell_1} \boxcirc \cdots \boxcirc \Cal{X}^{\boxcirc \ell_k} \right)^{\Sigma}
    \end{equation}
    which we construct as follows. Let $\alpha_i\colon \Cal{X}\to \Cal{X}^{\boxcirc \ell_i}$ for $i=1,\dots, k$. The map $\Gamma$ is induced by the assignment $(\alpha_1,\dots,\alpha_k)\mapsto \alpha_1 \boxcirc \cdots \boxcirc \alpha_k$,  where, e.g.,  if $k=2$ and $\underline{p}=(n,(t_1,\cdots,t_n))\in \Ncirc{2}[t]$ then \[(\alpha_1\boxcirc \alpha_2)[\underline{p}] \colon \Cal{X}^{\boxcirc 2}[\underline{p}] \to (\Cal{X}^{\ell_1} \boxcirc \Cal{X}^{\ell_2})[\underline{p}]\] is obtained levelwise by the maps $\alpha_1[n] \colon \Cal{X}[n]\to \Cal{X}^{\boxcirc \ell_1}[n]$ and $\alpha_2[t_i]\colon \Cal{X}^{\boxcirc \ell_2}[t_i]$ for $i=1,\dots,n$.

   With these two ingredients in place, the composition $\xi$ is obtained via the composition \begin{align*} 
            \Map_{\Dres}&\left(\Cal{X},\Cal{X}^{\boxcirc k}\right)^{\Sigma} \tensordot_{\Sigma}\left( \Map_{\Dres}\left(\Cal{X},\Cal{X}^{\boxcirc \ell_1}\right)^{\Sigma}  \circhat \cdots \circhat \Map_{\Dres}\left(\Cal{X},\Cal{X}^{\boxcirc \ell_k}\right)^{\Sigma} \right)
            \\
            &\xrightarrow{\id \tensordot_{\Sigma}  \Gamma}  
            \Map_{\Dres} \left(\Cal{X},\Cal{X}^{\boxcirc k}\right)^{\Sigma} \tensordot_{\Sigma}  \Map_{\Dres} \left(\Cal{X}^{\boxcirc k},  \Cal{X}^{\boxcirc \ell_1} \boxcirc \cdots \boxcirc \Cal{X}^{\boxcirc \ell_k} \right)^{\Sigma} 
            \\
            &\xrightarrow{\text{comp.}} 
            \Map_{\Dres} \left(\Cal{X}, \Cal{X}^{\boxcirc \ell_1} \boxcirc \cdots \boxcirc \Cal{X}^{\boxcirc \ell_k}\right)^{\Sigma} \\
            &\xrightarrow{(\mu_{\ell_1,\dots,\ell_k})_*}
            \Map_{\Dres}\left(\Cal{X},\Cal{X}^{\boxcirc \ell}\right)^{\Sigma}.
        \end{align*}
    
    Fortunately, the unit map is simpler to describe. We obtain $\varepsilon\colon \Isym\to \coend{X}$ as the morphism \[\mathbf{\Sigma}[n]\to \Map_{\Dres}(\Cal{X}[n], \Cal{X}[n])^{\Sigma_n}\]  adjoint to the action map $\mathbf{\Sigma}[n] \otimes \Cal{X}[n] \to \Cal{X}[n]$ which expresses the $\Sigma_n$ equivariance of $\Cal{X}[n]$. 
    
    Showing that $\xi$ and $\varepsilon$ satisfy the appropriate associativity and unitality conditions is a tedious though ultimately straightforward and may be adapted from the (somewhat simpler) proof of Proposition \ref{main-1} found in Section \ref{sec:1a}. 
\end{proof}

\subsection{$\Oper$-algebras are operads}
Our aim is now to show that $\Oper$-algebras indeed model (one-color) operads. 

\begin{proposition} 
\label{Oper-is-operads}
There is an equivalence of categories between algebras over $\Oper^{\mathsf{C}}$ and operads in $\mathsf{C}$.
\end{proposition}

\begin{proof} 
We show that a symmetric $\Oper$-algebra is necessarily an operad and note that the argument is readily reversed to show the converse statement. Suppose $\Cal{W}$ is a symmetric $\Oper$-algebra. Note, $\Oper_2\tensordot_{\Sigma}  \Cal{W}^{\circhat 2}\to \Cal{W}$ consists of maps \begin{equation} \label{Oper-alg-1}
    \Oper_2 (n, (k_1,\dots,k_n);k)\tensordot_{\Sigma} \left( \Cal{W}[n] \otimes \Cal{W}[k_1] \otimes \cdots \otimes \Cal{W}[k_n]\right) \to \Cal{W}[k]
\end{equation}
for each $\underline{p}=(n,(k_1,\dots,k_n))\in\Ncirc{2}$. Fix such a profile $\underline{p}$ and let $p_1,\dots,p_m$ be the multiplicities of the distinct factors $d_1,\dots,d_m$ among $k_1,\dots,k_n$. Coequalizing the actions of $\Isym$ identifies the symmetric group actions (resp. with $k_i$ replacing $n$) \[ \mathbf{\Sigma}[n] \otimes \Cal{W}[n]\to \Cal{W}[n]\] with the right action of $\Isym$ given in the proof of Proposition \ref{prop:coend-X}. Thus, (\ref{Oper-alg-1}) yields $\Sigma_k$-equivariant map of the form\begin{equation}
\label{Oper-alg-2}
    \mathbf{\Sigma}[k] \otimes_{H(k_1,\dots,k_n)} \Cal{W}[n]\otimes \Cal{W}[k_1]\otimes \cdots \otimes \Cal{W}[k_n]\to \Cal{W}[k]
\end{equation} which moreover obeys the correct equivariance, e.g.,  as  described in May \cite{May}. Said differently, (\ref{Oper-alg-2}) is the factor $(\Cal{W}\circ \Cal{W})[n,(k_1,\dots,k_n)]$ (as in Definition \ref{comp-prod-def-orb}) and the collection of all such maps then pieces together to form \[m\colon \Cal{W}\circ \Cal{W}\to \Cal{W}.\] 

Since $\Cal{W}\in \Alg{\Oper}$ there is a commutative diagram of the form
\[ \xymatrix{
    \left( \Oper_2 \tensordot_{\Sigma}  (\Oper_1 \circhat \Oper_2) \right) \tensordot_{\Sigma}  (\Cal{W}^{\circhat 3})
    \ar[r]^<(.1){\cong} \ar[dd]^{\xi_{2,(1,2)}\tensordot_{\Sigma}  \id}
    &
    \Oper_2 \tensordot_{\Sigma}  \left( ( \Oper_1\tensordot_{\Sigma}  (\Cal{W})) \circhat (\Oper_2 \tensordot_{\Sigma}  (\Cal{W}^{\circhat 2})) \right) \ar[d]^{\id \tensordot_{\Sigma} (\mu_1\circhat \mu_2)}
    \\
    & \Oper_2 \tensordot_{\Sigma}  (\Cal{W}\circhat \Cal{W}) \ar[d]^{\mu_2}
    \\
    \Oper_3 \tensordot_{\Sigma}  (\Cal{W}^{\circhat 3}) \ar[r]^{\mu_3} & \Cal{W}. }
\] 

The composite of the right side maps describes \[\Cal{W}\circ (\Cal{W}\circ \Cal{W})\xrightarrow{\, \id\circ m\,} \Cal{W}\circ \Cal{W} \xrightarrow{\; m\;} \Cal{W}\] and by construction the bottom map describes \[(\Cal{W}\circ \Cal{W})\circ \Cal{W} \xrightarrow{\, m\circ \id\,} \Cal{W}\circ \Cal{W} \xrightarrow{\; m\;} \Cal{W}.\] Associativity of $m$ follows as $\xi_{2,(1,2)}$ is an isomorphism. 

To produce the unit $u\colon I\to \Cal{W}$ we first recall that \[\mu_0\colon I\cong \Cal{P}_0\tensordot_{\Sigma}  (\Cal{W}^{\circhat 0}) \to \Cal{W}\] provides the unit map $u$ on $\Cal{W}$. There is then a commuting diagram \begin{equation*}
        \xymatrix{
            \Oper_2 \tensordot_{\Sigma}  \left( (\Oper_0\tensordot_{\Sigma}  (\Cal{W}^{\circhat 0})) \circhat (\Oper_1\tensordot_{\Sigma}  (\Cal{W})) \right) \ar[rr]^<(.2){\id\tensordot_{\Sigma} (\mu_0 \circhat \mu_1)} \ar[d]^{\cong}
            &&
            \Oper_2\tensordot_{\Sigma}  (\Cal{W} \circhat \Cal{W}) \ar[dd]^{\mu_2}
            \\
            \left( \Oper_2 \tensordot_{\Sigma}  (\Oper_0\circhat \Oper_1) \right) \tensordot_{\Sigma}  (\Cal{W}) \ar[d]^{\xi_1\tensordot_{\Sigma} \id}
            &&
            \\
            \Oper_1 \tensordot_{\Sigma} ( \Cal{W}) \ar[rr]^{\mu_1}
            &&
            \Cal{W}}
    \end{equation*} the composite of top and right arrows of which results in \[I\circ \Cal{W}\xrightarrow{\;u \circ \id\;} \Cal{W}\circ \Cal{W} \xrightarrow{\;m\;} \Cal{W}\] and the left and bottom arrows are all isomorphisms. Commutativity of the other unitality diagram follows a similar analysis.
\end{proof} 

\begin{corollary}
\label{oper-oper}
    Let $\Cal{W}$ be an operad, i.e., $\Oper$-algebra. Let $M\in \mathsf{C}$ and denote by $\bar{M}$ the symmetric sequence concentrated at level $0$ with $\bar{M}[0]=M$. Then, $M$ is an $\Cal{W}$-algebra (in the sense of Definition \ref{A-infty-alg}) if and only if $M$ is an $\Cal{W}$-algebra in the classic sense. 
\end{corollary}

\begin{proof}
    As in Definition \ref{A-infty-alg}, a $\Cal{W}$-algebra consists of maps \[\Oper_{\ell} \tensordot_{\Sigma}  (\Cal{W}^{\circhat (\ell-1)} \circhat \bar{M})\to \bar{M}.\] 
    
    Note, since $\bar{M}$ is concentrated at $0$, the only nontrivial contributors to such maps will have profiles which end in a string of $0$. In particular, for $\ell=2$ there are maps of the form \[\Oper_2(n,(0,\dots,0);0)\tensordot_{\Sigma}  \left( \Cal{W}[n] \otimes M^{\otimes n}\right)\to M.\] 
    
    Since $\Oper_2(n,(0,\dots,0);0)\cong \mathbf{\Sigma}[0]\cong \mathbf{1}$, the above maps descends to \[\Cal{W}[n]\otimes_{\Sigma_n} M^{\otimes n} \to M\] after coequalizing. Associativity and unitality follow a similar argument as the proof of Proposition \ref{Oper-is-operads}.
\end{proof}

\begin{remark}
 \label{trees-rmk}
    Though our description of $\Oper$ is new, descriptions of an $\mathbf{N}$-colored operad whose algebras are operads is not new. Berger-Moerdijk describe an $\mathbf{N}$-colored operad $\Cal{M}_{\textsf{Op}}$ in terms of trees whose algebras are operads in \cite[1.5.6]{BM-trees} (see also Dehling-Vallette \cite{Deh-Val}). Applying the forgetful functor $\mathbb{U}$ from Section \ref{N_l-is-N} to $\Oper$ yields an isomorphic $\mathbf{N}$-colored operad to that of Berger-Moerdijk, i.e., $\mathbb{U}\Oper\cong \Cal{M}_{\mathsf{Op}}$.
\end{remark}

\subsection{A model for $A_{\infty}$-operads} \label{sec:A}
We will now focus on a particular coendomorphism $\Nl$-operad in $\mathsf{Top}$, namely that on the cosimplicial symmetric sequence $\SD$ with $\SD[n]=\Sigma_n {\cdot} \Delta^{\bullet}$. 

\begin{proposition}\label{A-equiv-to-E} 
There is an equivalence of $\Nl$-operads $\mathsf{coEnd}(\SD) \to \Oper^{\mathsf{Top}}$. 
\end{proposition}

\begin{proof} Note that equivalences of $\Nl$-operads are computed levelwise (Definition \ref{def:equiv-of-Nl-ops}) and that a morphism $f\colon X\to Y$ of cosimplicial objects in $\mathsf{Top}$ induces a map $(\Sigma {\cdot} X)^{\boxcirc k}\to (\Sigma {\cdot} Y)^{\boxcirc k}$ for $k\geq 1$. If additionally there is a retract $r\colon Y\to X$ of $f$ there is a map $\mathsf{coEnd}(\Sigma {\cdot} X)\to \mathsf{coEnd}(\Sigma{\cdot} Y)$ on coendomorphism operads induced by post-composition with $f$ and pre-composition with $r$.

Since there are morphisms $\pt \xrightarrow{\sim}\Delta^n\xrightarrow{\sim} \pt$ for all $n \geq 0$ (i.e., by inclusion at a vertex) we then have  \begin{align*}
    \Map_{\Dres}\left(\SD,(\SD)^{\boxcirc k}\right)^{\Sigma}
    \xrightarrow{(\dagger)}\Map_{\Dres}\left(\Sigma{\cdot}\underline{\pt}, (\Sigma{\cdot}\underline{\pt})^{\boxcirc k}\right)^{\Sigma}\cong \Map\left(\mathbf{\Sigma}, \mathbf{\Sigma}^{\circ k}\right)^{\Sigma}
\end{align*}
for all $k \geq 0$, where $\underline{\pt}$ denotes the constant cosimplicial object on $\pt\in \mathsf{Top}$. Moreover, since $\pt \to \Delta^n\to \pt$ consists of weak equivalences between fibrant and cofibrant objects for all $n$, the indicated map $(\dagger)$ is a weak equivalence in $\mathsf{Top}$.
\end{proof}

Note that for $\underline{p}\in \Ncirc{k}[t]$, $\Oper^{\mathsf{Top}}_{k}(\underline{p};t)$ is just the discrete space $\mathbf{\Sigma}^{\circ k}[\underline{p}]$. Similarly, $\Oper^{\Top}\cong \Oper^{\mathsf{Top}}_+$ will encode operads in $(\Top, \wedge, S^0)$ and thus also in $\Spt$ via the tensoring of $\Spt$ over $\Top$. 

\begin{remark} \label{rem:rho}Note the functor $(-)_+ \colon (\mathsf{Top}, \times, \pt) \to (\Top, \wedge, S^0)$ which adds a disjoint basepoint induces isomorphisms of pointed spaces \[ \Map_{\Dres}^{\Top}  \left( \SD_+, (\SD_+)^{\boxcirc k} \right)^{\Sigma}\cong \Map_{\Dres}^{\mathsf{Top}} \left( \SD, (\SD)^{\boxcirc k} \right)^{\Sigma}_+. \]
Thus, there is an isomorphism \[\mathsf{coEnd}(\SD_+)\cong \mathsf{coEnd}(\SD)_+\] of $\Nl$-operads in $\Top$. For ease of notation we write $\Cal{A}$ for this $\Nl$-operad and note that Proposition \ref{A-equiv-to-E} provides a map $\rho\colon \Cal{A}\xrightarrow{\sim} \Oper^{\Top}$, i.e., $\Cal{A}$ is a suitably ``fattened-up'' version of $\Oper$ which will encode $A_{\infty}$-operads as its algebras, similar to $\mathbb{A}$ encoding $A_{\infty}$-monoids in Example \ref{ex:loop-space-A-infty-monoid}. 
\end{remark}

\section{$A_{\infty}$-operad structure on the derivatives of $\Id_{\Alg{O}}$}
\label{proofs}

The aim of this final section is to prove Theorem \ref{main}. We begin by proving Proposition \ref{main-1} which as a corollary provides a proof of the Theorem \ref{main}(a). In Section \ref{O-and-dern-equiv} we prove Theorem \ref{main}(b).

\subsection{Proof of Theorem \ref{main}(a)} \label{sec:1a} Since $C(\Cal{O})$ is a $\boxcirc$-monoid (see Proposition \ref{def:box-circ-prod-map}), Theorem \ref{main}(a) will follow from Proposition \ref{main-1}, which we prove below. 

\begin{proof}[Proof of Proposition \ref{main-1}]

Let $\Cal{X}$ be a $\boxcirc$-monoid in $\SymSeq^{\bdel}_{\Spt}$ whose multiplication we denote by $m\colon \Cal{X}\boxcirc \Cal{X}\to\Cal{X}$. We aim to show that $\Tot \Cal{X}$ is an algebra over $\Cal{A}$. We define maps $\lambda_{\ell}$ as follows (note the notation $\dot{\wedge}$ as $\tensordot$ from Definition \ref{ccp} for the monoidal category $(\Spt,\wedge,S)$)\[
    \lambda_{\ell}\colon \Cal{A}_{\ell}\dot{\wedge}_{\Sigma} (\Tot \Cal{X})^{\circhat \ell} \to \Tot \Cal{X}    
\]

For simplicity we first describe the $\ell=2$ case. Let $\underline{p}=(n,(k_1,\dots,k_n))\in\Ncirc{2}[k]$. Let $\psi\in \mathcal{A}_2(\underline{p};t)$, and let $\alpha,\beta\colon \SD_+\to \Cal{X}$ be maps of cosimplicial symmetric sequences. Define $\gamma$ at level $k$ by the composite\begin{equation*} 
\xymatrix{
    (\SD_+)^{\boxcirc 2} [n, (k_1,\dots,k_n)] \ar[rr]^-{\alpha[n]\boxcirc \beta[k_1,\dots,k_n]} 
    && (\Cal{X}^{\boxcirc2})[n, (k_1,\dots,k_n)] \ar[d]^{m_*}\\
    (\SD_+)[k] \ar[u]^-{\psi[k]} \ar@{.>}[rr]^{\gamma[k]} && \Cal{X}[k] }
\end{equation*}
where $\alpha[n] \boxcirc \beta[k_1,\dots,k_n]$ is provided via the map $\Gamma$ from (\ref{eq:def-Gamma}), the construction of which may be readily altered to give a map  \[
        \Gamma \colon \left( \Map_{\Dres}^{\Spt} \left(\SD_+, \Cal{X}\right)^{\Sigma}\right)^{\circhat \ell} \to \Map_{\Dres}^{\Spt}\left((\SD_+)^{\boxcirc \ell}, \Cal{X}^{\boxcirc \ell}\right)^{\Sigma}.\] 

In general, $\lambda_{\ell}$ is given by the following composite (compare with \cite[(1.13)]{AC-coalg}) \begin{align*}
    \Map_{\Dres}^{\Top} &\left(\SD_+, (\SD_+)^{\boxcirc \ell}\right)^{\Sigma}\dot{\wedge}_{\Sigma}  \left(\Map_{\Dres}^{\Spt}\left(\SD_+, \Cal{X}\right)^{\Sigma}\right)^{\circhat \ell}\\
    &
    \xrightarrow{\id\dot{\wedge}_{\Sigma}  \Gamma} 
    \Map_{\Dres}^{\Top}\left(\SD_+, (\SD_+)^{\boxcirc\ell} \right)^{\Sigma} \dot{\wedge}_{\Sigma}  \Map_{\Dres}^{\Spt}\left((\SD_+)^{\boxcirc \ell}, \Cal{X}^{\boxcirc \ell}\right)^{\Sigma} 
    \\
    &
    \xrightarrow{\text{compose}} \Map_{\Dres}^{\Spt}\left(\SD_+, \Cal{X}^{\boxcirc \ell} \right)^{\Sigma} 
    \\
    & \xrightarrow{m_*} \Map_{\Dres}^{\Spt}\left(\SD_+, \Cal{X}\right)^{\Sigma}
\end{align*}
where the composition map is adjoint to the composite of evaluation maps
\begin{align}
    &\SD_+ \wedge \Map_{\Dres}^{\Top}( \SD_+, (\SD_+)^{\boxcirc \ell} )^{\Sigma} \to  (\SD)^{\boxcirc \ell}_+, \\
    &(\SD_+)^{\boxcirc \ell} \dot{\wedge}_{\Sigma} \Map_{\Dres}^{\Spt} ((\SD_+)^{\boxcirc \ell}, \Cal{X}^{\boxcirc \ell} )^{\Sigma} \to \Cal{X}^{\boxcirc \ell}.
    \nonumber \end{align}
and $m_*$ is induced by the $\boxcirc$-monoid structure on $\Cal{X}$.


To show that $\lambda$ is associative we consider the following diagram, with $\psi'\in \Cal{A}_n$, $\psi_i \in \Cal{A}_{k_i}$ for $i=1,\dots,n$ such that the composite $\xi(\psi';\psi_1,\dots,\psi_n)=\psi \in \Cal{A}_k$. 

\[ \xymatrix{
    &(\SD_+)^{\boxcirc k} \ar[rr]
    &&
    \Cal{X}^{\boxcirc k} \ar[dl]^-{\theta_*} \ar@/^2pc/[dddl]^{m_*}
    \\
    (\SD_+)^{\boxcirc k_1} \boxcirc \cdots \boxcirc (\SD_+)^{\boxcirc k_n} \ar[ur]^-{\mu_{k_1,\dots,k_n}} \ar[rr]
        &&
    {\Cal{X}^{\boxcirc k_1}\boxcirc \cdots \boxcirc \Cal{X}^{\boxcirc k_n}}  \ar[d]^-{m_*\boxcirc \cdots \boxcirc m_*} &
    \\ 
    (\SD_+)^{\boxcirc n} \ar[u]^{\psi_1\boxcirc \cdots \boxcirc\psi_n} \ar[rr]
    &&
    \Cal{X}\boxcirc \cdots \boxcirc \Cal{X} \ar[d]^-{m_*} &
    \\ 
    \SD_+ \ar[u]^<(.3){\psi'} \ar@{.>}@/^1pc/[rr]_{\gamma'} \ar@{.>}@/_1pc/[rr]_{\gamma} \ar@/_2pc/[uuur]|!{[u];[ur]}\hole|!{[uu];[uur]}\hole^<(.6){\psi}  
    &&
    \Cal{X}
    }\]
Note here that $m_*$ is induced by repeatedly applying the pairing $m\colon \Cal{X}\boxcirc \Cal{X}\to\Cal{X}$ from the left, i.e., \[ \Cal{X}\boxcirc \Cal{X} \boxcirc \cdots \boxcirc \Cal{X} \xrightarrow{m\boxcirc \id\boxcirc \cdots \boxcirc \id} \cdots \xrightarrow{m\boxcirc \id\boxcirc\id} \Cal{X} \boxcirc \Cal{X} \boxcirc \Cal{X} \xrightarrow{m\boxcirc \id}\Cal{X} \boxcirc \Cal{X} \xrightarrow{m} \Cal{X}.\] 

The dashed morphisms $\gamma$ and $\gamma'$ are induced, respectively, by  $\lambda_k \xi_{n,(k_1,\dots,k_n)}$ and $\lambda_n (\id \tensordot_{\Sigma} (\lambda_{k_1} \circhat \cdots \circhat \lambda_{k_n})$. Note as well that $\mu_{k_1,\dots,k_n}$ is as in the proof of Proposition \ref{prop:coend-X}, and $\theta_*$ is the grouping map induced by $\theta$ (see Section \ref{boxcirc-monoid}), by which it follows that $\gamma$ and $\gamma'$ must agree. 

For unitality we recall that  $\epsilon\colon \Isym\to \Cal{A}$ is induced by the inclusion at $\id_{\bdel}$ and therefore the composite $\lambda_1[n]\varepsilon[n]$ in the following diagram \[
\xymatrix{
    \Cal{A}_1(n;n) \wedge_{\Sigma_n} \Map_{\Dres}^{\Spt} \left(\SD_+[n],\Cal{X}[n]\right)^{\Sigma_n}
    \ar[r]^-{\lambda_1[n]} & \Map_{\Dres}^{\Spt}\left(\SD_+[n],\Cal{X}[n]\right)^{\Sigma_n}
    \\ 
    \left( \Sigma_n \right)_+ \wedge_{\Sigma_n} \Map_{\Dres}^{\Spt}\left(\SD_+[n],\Cal{X}[n]\right)^{\Sigma_n}
    \ar[u]^-{\epsilon[n]} \ar[ru]_-{\cong}  &
}\]
is given by $S^0\wedge \Tot \Cal{X}[n]\xrightarrow{\cong} \Tot \Cal{X}[n]$.
\end{proof}

\subsection{An equivalence of $A_{\infty}$-operads between $\Cal{O}$ and $\partial_* \Id_{\Alg{O}}$} \label{O-and-dern-equiv}
We now show that the induced operad structure on $\partial_* \Id_{\Alg{O}}$ from Proposition \ref{main-1} agrees with the induced $\Cal{A}$-algebra structure on $\Cal{O}$, thus proving Theorem \ref{main}(b). Let $\rho \colon \Cal{A}\xrightarrow{\sim} \Oper$ be the map described in Remark \ref{rem:rho} and note an operad $\Cal{O}\in \Alg{\mathsf{Oper}}$ is in algebra over $\Cal{A}$ via the forgetful functor $\rho^*$.

\begin{proof}[Proof of Theorem \ref{main}(b)]
     By equivalence of $A_{\infty}$-operads we mean equivalence of $\Cal{A}$-algebras which restricts to an equivalence of underlying symmetric sequences. 
     
     Recall there is a natural coaugmentation $\Cal{O}\to C(\Cal{O})$ via $\Cal{O}\to J$. We have shown in Section \ref{sec:fibrant-model} that the coface $k$-cubes associated to \[ \Cal{O}\to C(\Cal{O}) \quad \text{and} \quad \partial_*\Id_{\Alg{O}}\to \holim_{\bdel^{\leq n-1}} \partial_* ( (UQ)^{\bullet+1})\]  are equivalent. Denoting these $k$-cubes by $\Cal{X}_k$ and $\Cal{Y}_k$, respectively, we note for $k\geq n\geq 1$ that as $\Cal{Y}_k[n]$ is homotopy cartesian so is $\Cal{X}_k[n]$. That is to say, for all $n\geq 1$ \begin{equation*}
    \label{On-->der_n}
        \Cal{O}[n] \xrightarrow{\sim} \holim_{\bdel}(  C(\Cal{O})[n]).
    \end{equation*}
    
    Let $\underline{\Cal{O}}$ be the constant cosimplicial object in $\SymSeq$ on $\Cal{O}$. From the above, the coaugmentation $\Cal{O}\to C(\Cal{O})$ induces a map of cosimplicial symmetric sequences $\varphi\colon \underline{\Cal{O}} \to  C(\Cal{O})$ such that  $\Tot \underline{\Cal{O}} \xrightarrow{\sim} \Tot C(\Cal{O})$. Moreover, $\underline{\Cal{O}}$ inherits a natural $\boxcirc$-monoid structure induced by the operad structure maps $\Cal{O}\circ\Cal{O}\to\Cal{O}$ and $I\to \Cal{O}$, and $\varphi$ respects this structure (i.e., is a map of $\boxcirc$-monoids).
    
    For each $n\geq 0$ we have \begin{align*}
        \Map_{\Dres}^{\Spt} (\Sigma_n {\cdot} \Delta^{\bullet}_+, \underline{\Cal{O}}[n])^{\Sigma_n}\xrightarrow{\sim}& 
        \Map_{\Dres}^{\Spt} (\Sigma_n {\cdot} \underline{\Delta^0}_+, \underline{\Cal{O}}[n])^{\Sigma_n} \\
        \cong& \Map^{\Spt}( \Sigma_n {\cdot} S^0, \Cal{O}[n])^{\Sigma_n}\cong \Map^{\Spt}(S^0, \Cal{O}[n])\cong \Cal{O}[n].
    \end{align*} and therefore, $\Tot \underline{\Cal{O}}\xrightarrow{\sim} \Cal{O}$. Thus, there are commuting diagrams for all $n\geq 0$ \begin{equation} \xymatrix{
    \Cal{A}_n \tensordot_{\Sigma}  (\rho^*\Cal{O})^{\circhat n} \ar[d]
    & 
    \Cal{A}_n \tensordot_{\Sigma}  (\Tot \underline{\Cal{O}})^{\circhat n} \ar[d] \ar[r] \ar[l]
    &
    \Cal{A}_n \tensordot_{\Sigma}  (\Tot C(\Cal{O}))^{\circhat n} \ar[d]
    \\
    \rho^*\Cal{O}
    &
    \Tot \underline{\Cal{O}} \ar[r]^-{\sim} \ar[l]_-{\sim}
    &
    \Tot C(\Cal{O}) 
    }\end{equation}
    where the left is the $\Cal{A}$-algebra structure map on $\rho^* \Cal{O}$ (which must factor through $\Oper$) and the right is the $\Cal{A}$-algebra structure map on $\partial_* \Id_{\Alg{O}}$.\end{proof}

\subsection{A class of $\partial_* \Id_{\Alg{O}}$-algebras} 
Though it will follow abstractly from Theorem \ref{O-and-dern-equiv}, the following corollary show that it is possible to describe an action of $\partial_* \Id_{\Alg{O}}$ explicitly on the $\TQ$-completion of sufficiently connected $\Cal{O}$-algebras. Recall that $X\simeq X_{\TQ}^{\wedge}$ for $0$-connected $X \in \Alg{O}$.

\begin{corollary} \label{TQ-completion-is-der-alg}
    Any $0$-connected $\Cal{O}$-algebra $X$ is weakly equivalent to an algebra over $\partial_* \Id_{\Alg{O}}$ via $X\mapsto X_{\TQ}^{\wedge}$. 
\end{corollary}    
    
\begin{proof}
    A straightforward modification of the proof of Proposition \ref{def:box-circ-prod-map} permits a well-defined map of cosimplicial diagrams \[r\colon  C(\Cal{O}) \boxcirc C(X)\to C(X)\] which endows $C(X)$ with the structure of a left module over $ C(\Cal{O})$. Strictly speaking we do need to be careful here, as $ C(\Cal{O})$ is not a strict monoid, so the module structure is obtained by replacing the right-most instances of $ C(\Cal{O})$ with $C(X)$ in (\ref{eq:assoc-boxcirc}) and (\ref{eq:unit-boxcirc}). Nonetheless, a straightforward adaptation of the proof of Proposition \ref{main-1} demonstrates maps \[\Cal{A}_{\ell} \tensordot_{\Sigma}  \left( (\Tot C(\Cal{O}))^{\circhat (\ell-1)} \circhat \bar{X}_{\TQ}^{\wedge}\right)\to \bar{X}_{\TQ}^{\wedge}\] where $\bar{X}_{\TQ}^{\wedge}$ is the symmetric sequence concentrated at level $0$ with value $X_{\TQ}^{\wedge}$, as required of Definition \ref{A-infty-alg}.
\end{proof}

\begin{remark}
    One intent of the above is to motivate the analogous statement for algebras over the derivatives of the identity in \textit{spaces}, which \textit{a priori} seems a bit more mysterious. Using the model $\partial_* \Id_{\mathsf{Top}_*} =\holim_{\bdel} C(\underline{S})$ (see Remark \ref{rem:C(S)}) we further show in \cite{Clark-2} that for any $\underline{S}$-coalgebra $Y$ in spectra (e.g., $Y=\Sp X$) the derived primitives $\Prim_{\underline{S}}(Y)$ inherits the structure of an algebra over $\partial_* \Id_{\mathsf{Top}_*}$ via a pairing of cosimplicial objects with respect to $\boxcirc$ (see also \cite{Ching-thesis}, \cite{Heuts-vn}, \cite{BR}).
    
    In this framework, Corollary \ref{TQ-completion-is-der-alg} tells us that any $0$-connected $X\in \Alg{O}$ is equivalent to its derived primitives $\Prim_{B(\Cal{O})}( \TQ(X) )$ (with respect to a suitable coalgebra structure on $B(\Cal{O})$, see Section \ref{B(O)}) as $\partial_* \Id_{\Alg{O}}\simeq \Cal{O}$-algebras. Note also that $\mathsf{Prim}_{B(\Cal{O})}(\TQ(X))\simeq X^{\wedge}_{\TQ}$. As such, one possible future avenue for our work is to try to push this result to work for any nilpotent $\Cal{O}$-algebra. This could potentially be used to prove that any nilpotent $\Cal{O}$-algebra is equivalent to its $\TQ$-completion (see also \cite{Blom}, \cite{Niko}, and further compare with \cite{BK}, \cite{AK}, \cite{Carl}, \cite{BH-2} that any nilpotent space is equivalent to its completion with respect to $\U \Sp$).

\end{remark}

\bibliographystyle{plain}
\bibliography{biblo}


\end{document}